\documentclass[10pt,reqno,final]{amsart}
\usepackage{amsfonts}
\usepackage[notcite,notref]{showkeys}
\usepackage{url,hyperref,multirow}
\usepackage{color}
\usepackage{cases}
\usepackage{subfigure}
\usepackage[ruled,vlined]{algorithm2e}
\usepackage{algorithmic}
\usepackage{subeqnarray}
\usepackage{amsmath}

\allowdisplaybreaks[4]
\usepackage{epsfig,amsmath,bm,caption2,epstopdf}
\usepackage{amssymb,version,graphicx,fancybox,mathrsfs,pifont,booktabs}

\catcode`\@=11 \theoremstyle{plain}
\@addtoreset{equation}{section}   
\renewcommand{\theequation}{\arabic{section}.\arabic{equation}}
\@addtoreset{figure}{section}
\renewcommand\thefigure{\thesection.\@arabic\c@figure}
\@addtoreset{table}{section}
\renewcommand\thetable{\thesection.\@arabic\c@table}

\newtheorem{thm}{\bf Theorem}[section]
\newtheorem{cor}{\bf Corollary}[section]
\newtheorem{prop}{Proposition}[section]

\newtheorem{lmm}{\bf Lemma}[section]

\newenvironment{lemma}{\begin{lmm}}{\end{lmm}}
\theoremstyle{remark}
\newtheorem{remark}{\bf Remark}[section]
\theoremstyle{definition}
\newtheorem{defn}{Definition}[section]

\newtheorem{exm}{\bf Example} 
\newenvironment{example}{\begin{exm}} {\end{exm}} 

\def \ri {{\rm i}}

\newcommand{\bs}[1]{\boldsymbol{#1}}

\newcommand \dint {\displaystyle \int}
\newcommand \dsum {\displaystyle\sum}

\newcommand \R {\mathbb{R}}

\def \cb {\color{blue}}

\begin{document}

\title[Integral fractional Laplacian] {Fast Fourier-like  Mapped Chebyshev Spectral-Galerkin  Methods for PDEs with  Integral
Fractional Laplacian in  Unbounded Domains}
\author[C. Sheng, \, J. Shen, \, T. Tang, \,  L. Wang \,  \& \,  H. Yuan]{ Changtao Sheng${}^{1}$, \; Jie Shen${}^{2},$ \; Tao Tang${}^{3},$ \; Li-Lian Wang${}^{1}$ \; and \; Huifang Yuan${}^{4}$}

\thanks{${}^{1}$Division of Mathematical Sciences, School of Physical
		and Mathematical Sciences, Nanyang Technological University,
		637371, Singapore. The research of the  authors is partially supported by Singapore MOE AcRF Tier 2 Grants:  MOE2018-T2-1-059 and MOE2017-T2-2-144. Emails: ctsheng@ntu.edu.sg (C. Sheng) and  lilian@ntu.edu.sg (L. Wang).\\
	\indent ${}^{2}$Department of Mathematics, Purdue University, West Lafayette, IN 47907-1957, USA. The work of of the author is partially supported by NSF DMS-1620262, DMS-1720442 and AFOSR FA9550-16-1-0102.  Email: shen7@purdue.edu (J. Shen).\\
	\indent  ${}^{3}$Division of Science and Technology, BNU-HKBU United International College, Zhuhai, Guangdong, China,  and SUSTech  International Center for Mathematics,  Southern University of Science and Technology, Shenzhen, China. The work of this author is partially supported by
	the NSF of China (under the Grant No. 11731006) and the Science Challenge Project
(No. TZ2018001).  Email: tangt@sustech.edu.cn (T. Tang).\\
	\indent ${}^{4}$Department of Mathematics, Southern University of Science and Technology, Shenzhen, China.
	Email:  
yuanhf@sustech.edu.cn (H. Yuan).\\
\indent The last two authors would like to thank NTU and SUSTech  International Center for Mathematics for hosting their mutual visits devoted to this collaborative work.
         }



 \begin{abstract}  In this paper, we propose a fast  spectral-Galerkin method for solving PDEs involving integral fractional Laplacian in $\mathbb{R}^d$, which is  built upon two essential components: (i)   the Dunford-Taylor formulation of the fractional Laplacian; and (ii)  Fourier-like bi-orthogonal mapped Chebyshev functions (MCFs) as basis functions.  As a result,
 the fractional Laplacian can be fully diagonalised, and the complexity of solving  an elliptic fractional PDE is quasi-optimal, i.e., $O((N\log_2N)^d)$ with $N$ being the number of modes in each spatial direction.   Ample numerical tests for various decaying exact solutions show that the convergence of the fast solver perfectly matches the order of theoretical error estimates.  With a suitable time-discretisation, the fast solver can be directly applied to a large class of nonlinear  fractional PDEs. As an example, we solve  the fractional nonlinear Schr{\"o}dinger equation by using  the fourth-order time-splitting method together with the proposed MCF-spectral-Galerkin method.
\end{abstract}


\maketitle

\vspace*{-10pt}
\section{Introduction}  Diffusion is the movement of a substance from an area of high concentration to an area of low concentration, which is a  ubiquitous physical process   in nature.
The normal diffusion models rooted in Brownian motion have been well-studied in years.
However,  numerous experimental and scientific evidences have shown that many  phenomena and complex systems involve anomalous diffusion, where the underlying stochastic processes are non-Brownian \cite{montroll1965random,metzler2000random,metzler2004restaurant}.  Notably,    the  fractional   models  have emerged as a powerful tool in modelling  anomalous diffusion  in  diverse fields (see, e.g., \cite{shlesinger1987levy,hatano1998dispersive,mccay1981theory,cushman1993nonlocal,benson2000application,brockmann2006scaling,sims2008scaling,carmichael2015fractional} and the references therein)  over the past two decades.  The   nonlocal  operators typically involved therein 
 include   the Riemann-Liouville, Caputo and Riesz fractional integrals/derivatives, or the fractional Laplacian.   They share some common and interwoven difficulties,  e.g.,  the nonlocal and  singular behaviours, so they are much more challenging and difficult to deal with than the usual local operators.  The recent  works \cite{lischke2018fractional,Bonito2018Numer} provide an up-to-date review in particular for numerical issues with several versions of fractional Laplacian.
The interested readers may also  refer to  \cite{silling2000reformulation,dubook} for  nonlocal modelling in many other applications.

A large volume of  literature is available for numerical solutions of  one-dimensional spatial and temporal fractional differential equations, which
particularly include the finite difference methods/finite element methods (see, e.g., \cite{deng2008finite,duo2015computing,guo2018high,jin2013error,jin2018numerical,zhang2018riesz} and many references therein), and spectral methods (see, e.g.,  \cite{chen2018jacobi,hou2017fractional,mao2016efficient}).
In this work, we are mainly interested in the integral fractional Laplaican in multiple dimensions, which is deemed as one of the most challenging nonlocal operators for both computation and analysis.
It is known that  for $s\in (0,1),$ the fractional Laplacian of $u\in  \mathscr{S}(\mathbb{R}^d)$
(the functions of  the Schwartz class) is defined  by the Fourier transform:
\begin{equation}\label{Ftransform}
\begin{split}
 (-\Delta)^s u(x):={\mathscr F}^{-1}\big[|\xi|^{2s}  {\mathscr F}[u](\xi)\big](x), \quad  \forall\, x \in {\mathbb R}^d.
\end{split}
\end{equation}
Equivalently, it can  be defined by the point-wise formula (cf. \cite[Prop. 3.3]{Nezza2012BSM}):
\begin{equation}\label{fracLap-defn}
(-\Delta)^s u(x)=C_{d,s}\, {\rm p.v.}\! \int_{\mathbb R^d} \frac{u(x)-u(y)}{|x-y|^{d+2s}}\, {\rm d}y,\quad x\in {\mathbb R}^d,
\end{equation}
where ``p.v." stands for the principle value and the normalisation constant
\begin{equation}\label{Cds}
C_{d,s}:=\Big(\int_{\mathbb R^d} \frac{1-\cos \xi_1}{|\xi|^{d+2s}}\,{\rm d}\xi\Big)^{-1}=
\frac{2^{2s}s\Gamma(s+d/2)}{\pi^{d/2}\Gamma(1-s)}.
\end{equation}
As a result, to evaluate fractional diffusion of $u$ at a spatial point, information involving all spatial points is needed. If $u$ is defined on a bounded domain $\Omega$, we first extend it to zero
outside $\Omega,$ and then use the above definition.

As many physically motivated fractional diffusion models  are naturally set in unbounded domains,
the development of effective  solution methods has attracted much recent attention.  In general,   the existing  approaches can be  classified into the following  two categories.
\begin{itemize}
\item This first is to approximate the solution by the orthogonal basis functions, and fully use the analytic properties of fractional Laplacian performing  on the basis
 (see, e.g.,   \cite{chen2018laguerre,mao2017hermite,tang2018hermite,tang2019rational}).
Based on some analytic fractional calculus formulas of  generalised Laguerre functions,
  Chen et al. \cite{chen2018laguerre} developed  an efficient spectral method for one-dimensional  fractional Laplacian on the whole line.  Using the  property that the Hermite functions are invariant  under the Fourier transform,   Mao and Shen  \cite{mao2017hermite}  proposed the  Hermite spectral-Galerkin method in the transformed domain based on  the Fourier definition \eqref{Ftransform}.
  Tang et al. \cite{tang2018hermite}  explicitly evaluated the Hermite fractional differentiation matrices  and implemented the spectral-collocation methods based on some elegant analytic tools.   The idea was extended to the rational approximation in  \cite{tang2019rational}.
  It is noteworthy that due to  the singular and non-separable factor  $|\xi|^{2s}$ in \eqref{Ftransform},   these methods become complicated even for $d=2$, and  computationally prohibitive for $d\ge 3.$
  \smallskip
  \item The second is to use suitable   equivalent formulations of the fractional Laplacian to
alleviate   its notorious numerical difficulties.
  In   Caffarelli and Silvestre \cite{caffarelli2007extension}, the $d$-dimensional fractional Laplacian is extended to  a $d+1$ dimensional  elliptic operator with degenerating/singular coefficients in the additional dimension.   This groundbreaking extension, together with the follow-up works for the fractional Laplacian in bounded domains, provides a viable alternative for its mathematical and numerical treatment (see, e.g., \cite{nochetto2014pde,nochetto2016pde,ainsworth2018hybrid} for finite element methods).
     On the other hand,
  the variational form corresponding to the fractional Laplacian can be formulated  as the Dunford-Taylor formula  (cf. \cite[Thm. 4.1]{bonito2019numerical}): for any $u,v\in H^{s}(\mathbb{R}^d)$ with $s\in (0,1)$,
\begin{equation}\label{DTfor0}
\left((-\Delta)^{\frac{s}2} u,(-\Delta)^{\frac{s}2} v\right)_{L^2(\mathbb{R}^d)}=C_s \int_0^\infty t^{1-2s} \int_{\mathbb R^d}
 \big((-\Delta)(\mathbb I-t^2 \Delta)^{-1} u \big)(x)\,  v(x)\, {\rm d}x\,{\rm d}t,
\end{equation}
where $\mathbb I$ is the identity operator and $C_s={2 \sin (\pi s)}/{\pi}.$  In particular, for fractional Laplacian in a bounded domain  $\Omega\subseteq \mathbb R^d,$   we have
\begin{equation}\label{DTfor0b}
\left((-\Delta)^{\frac{s}2} \tilde u,(-\Delta)^{\frac{s}2} \tilde v\right)_{L^2(\Omega)}=C_s \int_0^\infty t^{1-2s} \int_{\Omega}
 \big((-\Delta)(\mathbb I-t^2 \Delta)^{-1} \tilde  u \big)(x)\,  v(x)\, {\rm d}x\,{\rm d}t,
\end{equation}
where $\tilde u$ denotes the zero extension of $u.$ Very recently,    the  finite element method  with sinc quadrature (in  $t$),    was implemented and analysed   in \cite{bonito2019numerical,bonito2019sinc} based on \eqref{DTfor0b}. For each quadrature node $t_j,$  one  solves the elliptic problem:
\begin{equation}\label{elliptA}
-t^2_j\,\Delta w_j+w_j=\tilde u\quad {\rm in}\;\; {\mathbb R}^d,\;\; {\rm i.e.,}\;\;  w_j= (\mathbb I-t_j^2 \Delta)^{-1} \tilde u(x),
\end{equation}
where the unbounded domain has to be truncated, and the side of the domain depends on $t_j$.  In fact,  many sinc quadrature points should be used to resolve the singularity near  $t=0,$   but the problem \eqref{elliptA}  becomes  stiff and sharp boundary layers at $\partial \Omega$ can occur.

%

\end{itemize}

\smallskip
We also remark that  direct discretization of  the integral fractional Laplacian on bounded domains  based on the definition \eqref{fracLap-defn},  was
discussed in some recent works (see, e.g., \cite{huang2014numerical,duo2018finite} for finite difference methods; and   \cite{acosta2017fractional,acosta2017short,ainsworth2017aspects,ainsworth2018hybrid,deng2018time} for finite element methods).


In this paper, we develop  a fast spectral-Galerkin method for PDEs involving integral fractional Laplacian in  $\mathbb R^d$.  Consider, for example,  the model equation:
\begin{equation}\label{dDsLap00}
(-\Delta)^s u(x)+\gamma u(x)=f(x) \;\;\;  {\rm in}\;\;  \mathbb{R}^d; \quad
  u(x)=0 \;\;\;  {\rm as}\;\;  |x|\to \infty,
\end{equation}
where $s\in (0,1)$ and $\gamma>0.$ The efficient spectral algorithm is built upon two essential components{\cb :} (i) the
Dunford-Taylor formulation \eqref{DTfor0}  for the fractional Laplacian; and (ii) the approximation of the solution by the tensorial Fourier-like bi-orthogonal mapped Chebyshev functions.  As a result,
the complexity of solving
\eqref{dDsLap00} is $O((N\log_2N)^d),$ where $N$ is the degree of freedom along each spatial dimension.   The integration in $t$ (in \eqref{DTfor0}) can be evaluated exactly by  using  such a formulation and basis, so the main computational cost is  from the MCF expansions with FFT.  In fact, the framework is also applicable to  Hermite functions, but Hermite approximation is less compelling for  at least two reasons (i) the lack of FFT; and (ii) slow decay of the solution or the source term.  As opposite to usual
Laplacian,  the fractional Laplacian of a function with typical exponential or algebraic decay will  decay algebraically at  much slower rate (see Propositions \ref{case1:expo}-\ref{case1:alg} of this paper).  Thus, the MCF approximation is more preferable.  Indeed,  ample numerical results show  that the fast solver for \eqref{dDsLap00} has a convergence perfectly in agreement with
the theoretical estimate for various decaying exact solutions tested.

The rest of this paper is organized as follows. In section \ref{sect2},  we first  introduce the mapped Chebyshev functions and generate the Fourier-like bi-orthogonal MCFs in one dimension.
In section \ref{sect3}, we describe the fast MCF-spectral-Galerkin method built upon the Dunford-Taylor formulation of the fractional Laplacian. We conduct the error estimates and provide ample numerical results to show the convergence order of the solver is perfectly in agreement with the theoretical prediction in section \ref{sect4}.  In the final section, we apply the solver to spatial discretisation of the  fractional nonlinear Schr{\"o}dinger equation, and also conclude the paper with some final remarks.

\section{Fourier-like mapped Chebyshev functions}\label{sect2}

In this section,   we  introduce the mapped Chebyshev functions, from which
we construct the Fourier-like bi-orthogonal MCFs as one of the important tools for the
efficient spectral algorithms to be designed in the forthcoming section.



\subsection{Mapped Chebyshev functions}
Let  $T_n(y)=\cos(n\, {\rm arccos}(y)),$ $y\in \Lambda:=(-1,1)$  be the Chebyshev polynomial of degree $n$. The Chebyshev polynomials  satisfy the three-term recurrence relation
\begin{equation}
\label{recu}
T_{n+1}(y)=2y T_n(y)-T_{n-1}(y),\quad n\ge 1,
\end{equation}
with $T_0(y)=1$ and $T_1(y)=y.$
They form  a complete orthogonal system in $L_{\omega}^2(\Lambda)$, namely,
\begin{equation}\label{Chebpoly}
\dint_{\Lambda}T_n(y)T_m(y)\omega(y)\,{\rm d} y=\dfrac{\pi c_n}{2}\delta_{nm} \quad {\rm with}\quad \omega(y)=(1-y^2)^{-\frac{1}{2}},
\end{equation}
where $\delta_{nm}$ is the Kronecker symbol, and $c_0=2$ and $c_n=1$ for $n\ge 1$.
Recall the recurrence formulas (cf. \cite{szeg1939orthogonal}):
\begin{equation}\label{recu1}
y T_n(y)=(T_{n+1}(y)+T_{n-1}(y))/2,\;\;\; (1-y^2)T_n'(y)=\dfrac{n}{2}(T_{n-1}(y)-T_{n+1}(y)).
\end{equation}

We now define the mapped Chebyshev functions (MCFs) as in \cite{ben2002modified,shen2011spectral,shen2014approximations}.
\begin{defn}\label{defnMCF} {\em
Introduce the one-to-one algebraic mapping 
\begin{equation}\label{AlgeMapp}
x=\frac{y}{\sqrt{1-y^2}},\quad y=\frac{x}{\sqrt{1+x^2}},\quad x\in {\mathbb R},\;\; y\in \Lambda,
\end{equation}
and  define the MCFs as
 \begin{equation}\label{defiOMCF}
\mathbb{T}_n(x)= \frac{1} {\sqrt{c_n\pi/2}} \sqrt{1-y^2}\, T_n(y)= \frac{1}{\sqrt{c_n\pi/2}} \frac 1{\sqrt{1+x^2}}\, T_n\Big(\frac x {\sqrt{1+x^2}}\Big),
\end{equation}
for $x\in \mathbb R$ and integer $n\ge 0.$}
\end{defn}
\begin{remark}\label{Rmk21} {\em To enhance the resolution of MCFs, one can   incorporate
a scaling parameter $\nu>0$. More precisely, using the mapping 
\begin{equation*}
x=\frac{\nu\,y}{\sqrt{1-y^2}},\quad y=\frac{x}{\sqrt{\nu^2+x^2}}, \quad x\in {\mathbb R},\;\; y\in \Lambda,
\end{equation*}
the scaled MFCs can be defined as
\begin{equation*}
\mathbb T_{n}^\nu (x):=\frac{1}{\sqrt{c_n\pi/2}} \frac {\nu^{1/2}}{\sqrt{\nu^2+x^2}}\, T_n\Big(\frac x {\sqrt{\nu^2+x^2}}\Big)=\nu^{-\frac{1}{2}} \mathbb T_{n}(x/\nu).
\end{equation*}
In fact, it is straightforward to extend the properties  and algorithms from the usual MCFs to the scaled MFCs. For clarity of presentation,  we shall not carry the scaling parameter in the algorithm descriptions and  error analysis, but use it in the numerical experiments.}
\end{remark}

We have the following important properties of the MCFs, which can be shown readily by using the definition \ref{defnMCF} and the properties of Chebyshev polynomials in \eqref{recu}-\eqref{recu1} (also see \cite[Proposition 2.4]{shen2014approximations}).
\begin{prop}\label{orthMCF} The MCFs are orthonormal in $L^2(\mathbb R),$
and   we have
\begin{equation}\label{derivOrth}
\begin{split}
&{S}_{mn}={S}_{nm}= \dint_{\R}\mathbb{T}'_n(x)\,\mathbb{T}'_m(x)\,{\rm d}x\\
&=\begin{cases}
\displaystyle \frac{1}{c_n}\Big(\frac{(4c_{n-1}-c_{n-2})(n-1)^2}{16}+\frac{(4c_{n+1}-c_{n+2})(n+1)^2}{16}-\frac{c_{n}}{4}\Big), & {\rm if}\;\; m=n, \\[6pt]
\displaystyle \frac{1}{\sqrt{c_nc_{n+2}}}\Big(\frac{(c_{n}-c_{n+2})(n+1)}{8}-\frac{c_{n+1}(n+1)^2}{4}\Big), & {\rm if}\;\; m=n+2,\\[6pt]
\displaystyle \frac{1}{\sqrt{c_nc_{n+4}}}\Big(\frac{c_{n+2}(n+1)(n+3)}{16}\Big),  & {\rm if}\;\; m=n+ 4.\\
\end{cases}
\end{split}
\end{equation}
\end{prop}

\subsection{Fourier-like bi-orthogonal  MCFs in one dimension}
Let $\mathbb {P}_{N}$ be the set of all polynomials of degree at most $N$, and  define
the finite dimensional  space
\begin{equation}
\label{apprspac}
\mathbb V_{\!N}:=\big\{\phi: \phi(x)=g(x)\varphi(x), \; \forall \,  \varphi\in \mathbb{P}_N\big\},
\end{equation}
where $x,y$ are associated with the mapping \eqref{AlgeMapp} and
\begin{equation}\label{weigfunc}
g(x):= \sqrt{\omega(y) \dfrac{{\rm d}y}{{\rm d} x}}= \frac{1}{\sqrt{1+x^2}}=\sqrt{1-y^2}:=G(y).
\end{equation}
Note  that we have
\begin{equation}
\label{apprspac2}
\mathbb V_{\!N}:={\rm span}\big\{\mathbb{T}_n(x)\,:\, 0\leq n\leq N\big\}.
\end{equation}

Following the spirit of \cite{She.W09},  we next introduce a Fourier-like basis of $\mathbb V_{\!N}$, which is orthogonal in both $L^2$- and $H^1$-inner products.   For this purpose,
let $\bs S$ be a square matrix of order $N+1$ with entries given by \eqref{derivOrth},
and let $\bs I$ the identity matrix of the same size.
Note from \eqref{derivOrth} that $\bs S$ is  a symmetric positive definite matrix with nine nonzero diagonals.   Thus, all the eigenvalues are real and eigenvectors are orthonormal.  To this end, let $\bs{E}=(e_{jk})_{j,k=0,\cdots,N}$ be the matrix formed by the orthonormal eigenvectors of $\bs{S}$, and
$\bs \Sigma={\rm diag}\{\lambda_k\}$ be the diagonal matrix of  the corresponding eigenvalues. Thus, we have
\begin{equation}
\label{eige}
\bs{S}\bs{E}= \bs{E}\,\bs \Sigma, \quad \bs{E}^t\bs{E}=\bs{I}.
\end{equation}

We remark that  with an even and odd separation, we can work with two symmetric positive definite seven-diagonal sub-matrices to compute the eigenvalues and eigenvectors of $\bs S$, which  should be more stable for large $N.$

\begin{lemma}\label{partial-diagonal}  Let $\bs E=(\bs e_0,\bs e_1,\cdots, \bs e_N)$ be the matrix of the  eigenvectors of $\bs S,$ i.e., $\bs S \bs e_p=\lambda_p \bs e_p $ for $0\le p\le N.$ Define
\begin{equation}\label{neweigefunc}
\widehat{\mathbb T}_p(x):=\dsum_{j=0}^{N}e_{jp}\mathbb{T}_j(x), \quad \bs e_p=(e_{0p},e_{1p},\cdots, e_{Np})^t,\quad  0\leq p\leq N.
\end{equation}
Then $\{\widehat{\mathbb T}_p\}^N_{p=0}$  form an equivalent basis of  $\mathbb{V}_{\!N},$  and they are bi-orthogonal in the sense that
\begin{equation}\label{bi-orth-eqn}
(\widehat{\mathbb T}_p,\widehat{\mathbb T}_q)_{L^2(\mathbb{R})}=\delta_{pq},\quad  \big(\,\widehat{\mathbb T}_p', \widehat{\mathbb T}_q'\,\big)_{L^2(\mathbb{R})}=\lambda_p\delta_{pq},
\quad 0\le p,q\le N.
\end{equation}
\end{lemma}
\begin{proof} In view of the definition \eqref{neweigefunc},  we infer from the orthogonality of MCFs  and  \eqref{eige} that
\begin{equation*}\label{newbasmtrix}
\begin{split}
(\widehat{\mathbb T}_p,\widehat{\mathbb T}_q)_{L^2(\mathbb{R})}& =\sum_{j=0}^{N} \sum_{k=0}^N e_{kp}e_{jq}(\mathbb{T}_k,\mathbb{T}_j)_{L^2(\mathbb{R})}=\sum_{j=0}^{N} \sum_{k=0}^N e_{jq}\delta_{jk}e_{kp} = \sum_{k=0}^N  e_{kq} e_{kp}= \bs e_q^t \bs e_p = \delta_{pq}.
\end{split}\end{equation*}
Similarly, we can show that
\begin{equation*}\label{newbasmtrix2}
\begin{split}
\big(\widehat{\mathbb T}_p',\widehat{\mathbb T}_q'\big)_{L^2(\mathbb{R})}& =\sum_{j=0}^{N} \sum_{k=0}^N e_{kp}e_{jq}(\mathbb{T}_k',\mathbb{T}_j')_{L^2(\mathbb{R})}=\sum_{j=0}^{N} \sum_{k=0}^N e_{jq}S_{jk}e_{kp}
 =(\bs{E}^t \bs{S}\bs{E})_{pq}=(\bs \Sigma)_{pq}=\lambda_p\delta_{pq}.
\end{split}\end{equation*}
This ends the proof.
\end{proof}

\section{MCF-spectral-Galerkin method based on Dunford-Taylor formulation}\label{sect3}

In this section, we describe the fast MCF spectral-Galerkin algorithm for a model elliptic problem with integral  fractional Laplacian. We then apply the solver for spatial discretisation of some nonlinear fractional PDEs in the next section.

\subsection{Some notation}
Denote by  $\mathscr{S}(\mathbb{R}^d)$ the functions of  the Schwartz class, and  let $\mathscr{S}'(\mathbb{R}^d)$
be the topological dual of  $\mathscr{S}(\mathbb{R}^d).$ For any $u\in \mathscr{S}(\mathbb{R}^d),$ its Fourier transform is given by
$$
{\mathscr F}[u](\xi)=\frac {1}{(2\pi)^{\frac d 2}}\int_{\mathbb R^d} u(x) e^{-\ri \xi\cdot x}{\rm d} x,\quad \forall\, \xi \in {\mathbb R}^d.
$$
For real $s\ge 0,$ we define  the fractional Sobolev space (cf. \cite[P. 530]{Nezza2012BSM}):
\begin{align}\label{Hssps}
H^s(\mathbb{R}^d)=\Big\{u\in L^{2}(\mathbb {R}^d):\,  \lVert u\rVert_{H^s(\mathbb{R}^d)}^2=\int_{\mathbb R^d} (1+\lvert\xi\rvert^{2s})
\big|\mathscr{F}[u](\xi)\big|^{2}{\rm d}\xi<+\infty\Big\},
\end{align}
and  an analogous definition for the case $s<0$  is to set
\begin{align}\label{Hsspsf}
H^s(\mathbb{R}^d)=\Big\{u\in \mathscr{S}'(\mathbb{R}^d):\,  \lVert u\rVert_{H^s(\mathbb{R}^d)}^2=\int_{\mathbb R^d} (1+\lvert\xi\rvert^{2})^s
\big|\mathscr{F}[u](\xi)\big|^{2}{\rm d}\xi<+\infty\Big\},
\end{align}
although in this case the space $H^s(\mathbb{R}^d)$ is not a subset of $L^2(\mathbb R^d).$

According to \cite[Prop. 3.4]{Nezza2012BSM}, we know that  for $s\in (0,1),$  the space $H^s(\mathbb{R}^d)$ can also be characterised by the fractional Laplacian defined in \eqref{fracLap-defn}, equipped with the norm
$$
\lVert u\rVert_{H^s(\mathbb{R}^d)}=\big(\|u\|_{L^2(\mathbb R^d)}^2+[u]_{H^s(\mathbb R^d)}^2\big)^{\frac 1 2},
$$
where $[u]_{H^s(\mathbb R^d)}$ is so-called Gagliardo (semi)norm of $u,$ given by
\begin{equation}\label{uHs}
[u]_{H^s(\mathbb R^d)}=\Big(\int_{\mathbb R^d}\int_{\mathbb R^d} \frac{|u(x)-u(y)|^2}{|x-y|^{d+2s}}{\rm d}x {\rm d} y\Big)^{\frac 1 2}.
\end{equation}
Indeed, by \cite[Prop. 3.6]{Nezza2012BSM}, we have that for $s\in (0,1),$
\begin{equation}\label{propS}
[u]_{H^s(\mathbb R^d)}^2=2C_{d,s}^{-1}\|(-\Delta )^{s/2} u \|^2_{L^2(\mathbb R^d)}.
\end{equation}

We have the following important space interpolation property  (cf. \cite[Ch. 1]{Agranovich2015Book}), which will be used for the error analysis later on.
\begin{lemma}\label{interpola}
For real $r_0,r_1\ge 0,$ let $r=(1-\theta) r_0+\theta r_1$ with $\theta \in[0,1]$. Then  for any $u\in H^{r_0}(\mathbb{R}^d)\cap H^{r_1}(\mathbb{R}^d),$  we have
\begin{equation}\label{Hrinterp}
\|u\|_{H^{r}(\mathbb{R}^d)}\le \|u\|^{1-\theta}_{H^{r_0}(\mathbb{R}^d)} \, \|u\|^{\theta}_{H^{r_1}(\mathbb{R}^d)}.
\end{equation}
In particular,  for $s\in [0,1],$ we have
\begin{equation}\label{orthNs22}
 \|u \|_{H^s(\mathbb R^d)}\le\|u \|_{L^2(\mathbb R^d)}^{1-s}\,
   \|u \|_{H^1(\mathbb R^d)}^{s}.
\end{equation}
\end{lemma}

\subsection{Dunford-Taylor formulation of the fractional Laplacian}
To fix the idea, we consider  
 \begin{equation}\label{dDsLap}
(-\Delta)^s u(x)+\gamma u(x)=f(x) \;\;\;  {\rm in}\;\;  \mathbb{R}^d; \quad
  u(x)=0 \;\;\;  {\rm as}\;\;  |x|\to \infty,
\end{equation}
where $s\in (0,1), \gamma>0$ and  $f\in H^{-s}(\mathbb{R}^d).$

A weak form for \eqref{dDsLap} is  to  find $u\in H^s(\mathbb R^d)$ such that
\begin{equation}\label{uvsh}
\begin{split}
{\mathcal B}(u,v)& =\big((-\Delta)^{s/2} u,  (-\Delta)^{s/2} v  \big)_{L^2(\mathbb{R}^d)}+\gamma(u,v)_{L^2(\mathbb{R}^d)}\\
& =[u,v]_{H^s({\mathbb R}^d)}+ \gamma(u,v)_{L^2(\mathbb{R}^d)}=(f,v)_{L^2(\mathbb{R}^d)},\quad \forall v\in  H^s(\mathbb R^d),
\end{split}
\end{equation}
where $[u,v]_{H^s({\mathbb R}^d)}$ induces the Gagliardo (semi)norm in  \eqref{uHs}.

By the definitions \eqref{Ftransform} and \eqref{Hssps}, we immediately obtain the continuity  and coercivity of the bilinear form $a(\cdot,\cdot)$, that is, for any $u,v\in H^{s}(\mathbb{R}^d),$
\begin{equation}
\label{contcoer}\begin{split}
 |{\mathcal B}(u,v)| \lesssim\|u\|_{H^{s}(\mathbb{R}^d)}\|v\|_{H^{s}(\mathbb{R}^d)},\quad  |\mathcal B(u,u)|\gtrsim\|u\|_{H^{s}(\mathbb{R}^d)}^2.
\end{split}\end{equation}
Then, we derive from the Lax-Milgram lemma (cf. \cite{babuska1972survey}) that the problem \eqref{uvsh} admits a unique solution satisfying
$$\|u\|_{H^{s}(\mathbb{R}^d)}\lesssim \|f\|_{H^{-s}(\mathbb{R}^d)}.$$
In view of the definitions in \eqref{Ftransform}-\eqref{fracLap-defn}, we have the equivalent forms of
  $[u,v]_{H^s({\mathbb R}^d)}$ as follows
  \begin{align}
  [u,v]_{H^s({\mathbb R}^d)}&= 
  \int_{\mathbb R^d}\int_{\mathbb R^d} \frac{(u(x)-u(y))(v(x)-v(y))}{|x-y|^{d+2s}}{\rm d}x {\rm d} y \label{physform}\\[4pt]
  &= \int_{\mathbb R^d}|\xi|^{2s} {\mathscr F}[u](\xi)\overline{{\mathscr F}[v](\xi)} \,{\rm d}\xi. \label{freqform}
  \end{align}
  It is noteworthy that the direct implementation of a numerical scheme based on \eqref{physform} (i.e., in physical space) is very difficult. Most of the existing works  (see, e.g.,
  \cite{mao2017hermite,tang2018hermite,tang2019rational}) are therefore mainly based on \eqref{freqform} (i.e.,  in the frequency space).   The Hermite function approaches can  take the advantage that the Fourier transforms of  Hermite functions  are  explicitly known.  However, in multiple dimensions,  the  non-separable/singular factor  $|\xi|^{2s}=(\xi_1^2+\cdots+\xi_d^2)^s$ makes the tensorial   approach computationally prohibitive.  On the other hand, the fractional Laplacian operator may become rather complicated when a coordinate transform is applied, so the mapped Chebyshev approximation can not be applied  in either of the above formulations.

  In what follows, we  resort to an alternative formulation of the fractional Laplacian that can overcome these numerical difficulties.
According to  \cite[Theorem 4.1]{bonito2019numerical}, we have  the following  Dunford-Taylor formulation of the integral fractional Laplacian.
%
\begin{lemma}\label{DTthm}
 For any $u,v\in H^{s}(\mathbb{R}^d)$ with $s\in (0,1)$, we have
\begin{equation}\label{DTfor}
\left((-\Delta)^{\frac{s}2} u,(-\Delta)^{\frac{s}2} v\right)_{L^2(\mathbb{R}^d)}=C_s \int_0^\infty t^{1-2s} \int_{\mathbb R^d}
 \big((-\Delta)(\mathbb I-t^2 \Delta)^{-1} u \big)(x)\,  v(x)\, {\rm d}x\,{\rm d}t,
\end{equation}
where $\mathbb I$ is the identity operator and
\begin{equation}
C_s=\frac{2 \sin (\pi s)}{\pi}.
\end{equation}
\end{lemma}

Let us   denote $ w=w(u,t):= (\mathbb I-t^2 \Delta)^{-1} u(x).$ Then there holds
\begin{equation}\label{wutprob}
-t^2\Delta w+w=u\quad {\rm in}\;\; {\mathbb R}^d, \quad {\rm so}\;\;\;
(-\Delta)(\mathbb I-t^2 \Delta)^{-1} u =-\Delta w=t^{-2}(u-w).
\end{equation}
 As a result,  we can rewrite the weak form \eqref{uvsh} as:  find $u\in H^s(\mathbb R^d)$ such that
\begin{equation}\label{auveqn2}
 {\mathcal B}(u,v)=C_s \int_0^\infty t^{-1-2s} (u-w, v)_{L^2(\mathbb{R}^d)}\,  {\rm d}t+\gamma\,(u,v)_{L^2(\mathbb{R}^d)}=(f,v)_{L^2(\mathbb{R}^d)},\;\; \forall v\in  H^s(\mathbb R^d),
\end{equation}
where $w=w(u,t)$ solves
\begin{equation}\label{coneta}
t^{2} (\nabla w,\nabla \psi)_{L^2(\mathbb{R}^d)}+(w, \psi)_{L^2(\mathbb{R}^d)}=(u, \psi)_{L^2(\mathbb{R}^d)}, \quad \forall \psi \in H^{1}(\mathbb{R}^{d}).
\end{equation}
It is evident that the wellposedness of \eqref{auveqn2}-\eqref{coneta}  follows from its equivalence to \eqref{uvsh}.

\subsection{The MCF spectral-Galerkin scheme and its implementation}
Define 
\begin{equation}\label{VnVn}
{\mathbb V}_{\!N}^d=\mathbb V_{\!N}\otimes\cdots \otimes \mathbb V_{\!N},
\end{equation}
which is the tensor product of $d$ copies of  $\mathbb V_{\!N}$  defined in \eqref{apprspac}.
Here,  ${\mathbb V}_{\!N}^1=\mathbb V_{\!N},$ and denote $I_N^d: C(\mathbb R^d)\to {\mathbb V}_{\!N}^d$ the tensorial mapped Chebyshev interpolation operator.
The MCF spectral-Galerkin approximation to \eqref{auveqn2}-\eqref{coneta} is to find  $u_N\in {\mathbb V}_{\!N}^d$ such that
\begin{equation}\label{auveqn23}
\begin{split}
 {\mathcal B}_N(u_N,v_N)&=C_s \int_0^\infty t^{-1-2s} (u_N-w_N, v_N)_{L^2(\mathbb{R}^d)}\,  {\rm d}t+\gamma(u_N,v_N)_{L^2(\mathbb{R}^d)}\\
 & =(I_N^d f,v_N)_{L^2(\mathbb{R}^d)},\quad \forall v_N\in \mathbb{V}_{\!N}^d,
 \end{split}
\end{equation}
where we find $w_N:=w_N(u_N,t)\in \mathbb{V}_{\!N}^d$ such that for any $t>0,$
\begin{equation}\label{bSeqn}
t^{2} (\nabla w_N,\nabla \psi)_{L^2(\mathbb{R}^d)}+(w_N, \psi)_{L^2(\mathbb{R}^d)}=(u_N, \psi)_{L^2(\mathbb{R}^d)}, \quad \forall \psi \in \mathbb{V}_{\!N}^d.
\end{equation}

Define the $d$-dimensional tensorial Fourier-like basis and denote the vector of the corresponding eigenvalues in \eqref{eige} by
\begin{equation}\label{mathTT}
 \widehat {\mathbb T}_n(x)=\displaystyle\prod^d_{j=1} \widehat {\mathbb T}_{n_j}(x_j), \quad x\in\mathbb{R}^d;\quad
\lambda_n=(\lambda_{n_1},\cdots,\lambda_{n_d})^t .
\end{equation}
Accordingly,  we have 
\begin{equation}
\label{apprspacRd}
\mathbb{V}_{\!N}^d={\rm span}\big\{ \widehat {\mathbb T}_n(x),~n\in\Upsilon_{\!N}\big\},
\end{equation}
where the index set
\begin{equation}\label{UpsilonA}
  \Upsilon_{\!N}:=\big\{n=(n_1,\cdots, n_d)\,:\, 0\le n_j\le N,\; 1\le j\le d\big\}.
\end{equation}

As an extension of \eqref{bi-orth-eqn}, we have the following attractive property of the tensorial Fourier-like MCFs.
\begin{thm}  For the tensorial Fourier-like MCFs, we have
\begin{equation}\label{pqN}
( \widehat {\mathbb T}_p,  \widehat {\mathbb T}_q)_{L^2(\mathbb{R}^d)}=\delta_{pq}\,;  \quad
\big(\nabla \widehat {\mathbb T}_p,  \nabla \widehat {\mathbb T}_q\big)_{L^2(\mathbb{R}^d)}=|\lambda_p|_1\,\delta_{pq},
\end{equation}
where $p,q\in \Upsilon_{\!N}$ and
\begin{equation}\label{pqnotes}
\delta_{pq}=\prod_{j=0}^d \delta_{p_jq_j},  \quad  |\lambda_p|_1= 
\lambda_{p_1}+\cdots+\lambda_{p_d}.
\end{equation}
\end{thm}
\begin{proof} One verifies by using the orthogonality  \eqref{bi-orth-eqn} and the definition \eqref{mathTT} that
\begin{equation*}\label{multibas1}
\begin{split}
( \widehat {\mathbb T}_p, \widehat {\mathbb T}_q)_{L^2(\mathbb{R}^d)}=( \widehat {\mathbb T}_{p_1}, \widehat {\mathbb T}_{q_1})_{L^2(\mathbb{R})}\cdots( \widehat {\mathbb T}_{p_d}, \widehat {\mathbb T}_{q_d})_{L^2(\mathbb{R})}=\delta_{p_1q_1}\cdots\delta_{p_dq_d}=\delta_{pq},
\end{split}\end{equation*}
and
\begin{equation*}\label{multibas2}
\begin{split}
(\nabla  \widehat {\mathbb T}_p,\nabla \widehat {\mathbb T}_q)_{L^2(\mathbb{R}^d)} & =\big\{( \widehat {\mathbb T}_{p_1}^\prime, \widehat {\mathbb T}_{q_1}^\prime)_{L^2(\mathbb{R})}
( \widehat {\mathbb T}_{p_2}, \widehat {\mathbb T}_{q_2})_{L^2(\mathbb{R})} \cdots( \widehat {\mathbb T}_{p_d}, \widehat {\mathbb T}_{q_d})_{L^2(\mathbb{R})}\big\}
\\&\quad +\big\{( \widehat {\mathbb T}_{p_1}, \widehat {\mathbb T}_{q_1})_{L^2(\mathbb{R})}( \widehat {\mathbb T}_{p_2}^\prime, \widehat {\mathbb T}^\prime_{q_2})_{L^2(\mathbb{R})}\cdots( \widehat {\mathbb T}_{p_d}, \widehat {\mathbb T}_{q_d})_{L^2(\mathbb{R})}\big\}
\\&\quad+
\cdots +\big\{( \widehat {\mathbb T}_{p_1}, \widehat {\mathbb T}_{q_1})_{L^2(\mathbb{R})}\cdots( \widehat {\mathbb T}_{p_{d-1}}, \widehat {\mathbb T}_{q_{d-1}})_{L^2(\mathbb{R})} ( \widehat {\mathbb T}_{p_d}^\prime, \widehat {\mathbb T}^\prime_{q_d})_{L^2(\mathbb{R})}\big\}
\\& =\lambda_{p_1} \delta_{p_1q_1}\cdots\delta_{p_dq_d} +\lambda_{p_2} \delta_{p_1q_1}\cdots\delta_{p_dq_d} +\cdots+\lambda_{p_d} \delta_{p_1q_1}\cdots\delta_{p_dq_d} \\& =(\lambda_{p_1}+\cdots+\lambda_{p_d})\delta_{pq}=|\lambda_p|_1\delta_{pq}.
\end{split}\end{equation*}
This ends the proof.
\end{proof}

Remarkably, the use of the Fourier-like MCF  can diagonalise the integral fractional Laplacian in the Dunford-Taylor formulation.
\begin{thm}\label{newtheorem}  Using the tensorial Fourier-like MCFs as basis functions,
the solution of \eqref{auveqn23}-\eqref{bSeqn} can be uniquely expressed as
\begin{equation}
\label{smsolurd}
u_N(x)=\dsum_{p\in \Upsilon_{\!N}}\frac{\hat f_p}{\gamma+|\lambda_p|_1^s}\, \widehat{\mathbb{T}}_p(x),\quad x\in {\mathbb R^d},
\end{equation}
where $\widehat {\mathbb T}_p, \lambda_p$ are defined in \eqref{mathTT}, and
\begin{equation}\label{fppp}
\hat f_p=(I_N^df,\widehat{\mathbb{T}}_p)_{L^2(\mathbb{R}^d)},\quad p\in \Upsilon_{\!N}.
\end{equation}
\end{thm}
 \begin{proof}
Write
 \begin{equation}\label{solufpsird}
u_{N} =\dsum_{p\in\Upsilon_{\!N}}\hat{u}_p\widehat{\mathbb{T}}_p(x),\quad
 w_{N} =\dsum_{p\in\Upsilon_{\!N}}\hat{w}_p\widehat{\mathbb{T}}_p(x),
\end{equation}
where $w_N$ is the unique solution of \eqref{bSeqn} associated with $u_N.$  For clarity, we split the proof into the following steps.

(i).  We first show that
$w_N$ can be uniquely determined by $u_N$ via
\begin{equation}\label{wNexpansion}
w_N=\dsum_{p\in\Upsilon_{\!N}}\frac{\hat{u}_p}{1+t^2|\lambda_p|_1}\widehat{\mathbb{T}}_p(x).
\end{equation}
  Substituting \eqref{solufpsird} into \eqref{bSeqn}, and taking $\psi=\widehat{\mathbb{T}}_q$ in \eqref{bSeqn}, we arrive at
\begin{equation*}\label{bSeqnrd}
\dsum_{p\in\Upsilon_{\!N}}\hat{w}_p\big\{t^{2} (\nabla \widehat{\mathbb{T}}_p,\nabla \widehat{\mathbb{T}}_q)_{L^2(\mathbb{R}^d)}+(\widehat{\mathbb{T}}_p, \widehat{\mathbb{T}}_q)_{L^2(\mathbb{R}^d)}\big\}=\dsum_{p\in\Upsilon_{\!N}}\hat{u}_p(\widehat{\mathbb{T}}_p,  \widehat{\mathbb{T}}_q)_{L^2(\mathbb{R}^d)}, \quad \forall q\in \Upsilon_{\!N}.
\end{equation*}
By the orthogonality \eqref{pqN}, we obtain
\begin{equation*}
\dsum_{p\in\Upsilon_N}\hat{w}_p\big\{t^{2}|\lambda_p|_1+1\big\}\delta_{pq}=\dsum_{p\in\Upsilon_N}\hat{u}_p\delta_{pq}, \quad \forall q\in \Upsilon_{\!N},
\end{equation*}
which implies \eqref{wNexpansion}, as
\begin{equation*}
\label{numesolufracLapl0rd}
\hat{w}_p=\frac{\hat{u}_p}{1+t^2|\lambda_p|_1},\quad \forall p\in \Upsilon_{\!N}.
\end{equation*}

(ii). We next prove the integral identity:
\begin{equation}\label{ytxi0}
\begin{split}
 \dint_0^\infty \frac{ t^{1-2s}\, |\lambda_p|_1^{1-s}}{  1+t^2|\lambda_p|_1}\,{\rm d}t=\frac \pi {2\sin (\pi s)}=\frac 1 {C_s}.
\end{split}
\end{equation}
Indeed, using a change of variable $y=t\sqrt{|\lambda_p|_1},$ we find readily that
\begin{equation*}\label{ytxi1}\begin{split}
 \dint_0^\infty \frac{t^{1-2s}\,|\lambda_p|_1^{1-s}}{  1+t^2|\lambda_p|_1}{\rm d}t
 =\int_0^{\infty} \frac{y^{1-2s}}{1+y^2}{\rm d}y=\frac \pi {2\sin (\pi s)},
\end{split}\end{equation*}
where we used  the  known formula \eqref{form1} below with $\mu=2-2s$ and $\nu=2$.
According to \cite[P. 325, P. 918, P.905]{Gradshteyn2015Book}, we have for $\nu\ge \mu\ge 0$ and $\nu\not=0,$
\begin{equation}\begin{split}\label{form1}
&\int_{0}^{\infty} \frac{x^{\mu-1}}{1+x^{\nu}} {\rm d} x=\frac{1}{\nu} \mathrm{B}\Big(\frac{\mu}{\nu}, 1-\frac{\mu}{\nu}\Big)=\frac{1}{\nu} \Gamma\Big(\frac{\mu}{\nu}\Big) \Gamma\Big(1-\frac{\mu}{\nu}\Big)
=\frac{\pi}{\nu\sin(\pi \mu/\nu)},
\end{split}\end{equation}
where we used the properties of the Beta and Gamma functions:
\begin{equation*}\label{form0}
\begin{split}
\mathrm{B}(x,y)=\frac{\Gamma(x)\Gamma(y)}{\Gamma(x+y)},\quad \Gamma(1-x)\Gamma(x)=\frac{\pi}{\sin(\pi x)}.
\end{split}
\end{equation*}


(iii). Finally, we can derive \eqref{smsolurd} with the aid of   \eqref{wNexpansion}-\eqref{ytxi0}.
It is evident that by \eqref{solufpsird}-\eqref{wNexpansion},
\begin{equation}\label{uNwN}
u_N-w_N= t^2 \dsum_{p\in\Upsilon_{\!N}}\frac{ |\lambda_p|_1}{1+t^2|\lambda_p|_1}\,  \hat{u}_p \widehat{\mathbb{T}}_p(x).
\end{equation}
Thus, substituting   \eqref{uNwN}  into \eqref{auveqn23} with $v_N=\widehat{\mathbb{T}}_q$, we obtain from \eqref{pqN} and \eqref{ytxi0} that
\begin{equation*}
\label{newschemerd}\begin{split}
{\mathcal B}_N(u_N,\widehat{\mathbb{T}}_q)&=\dsum_{p\in\Upsilon_{\!N}}\hat{u}_p\bigg\{C_s \int_0^\infty t^{1-2s} \dint_{\mathbb R^d}\frac{|\lambda_p|_1}{  1+t^2|\lambda_p|_1}\widehat{\mathbb{T}}_p(x)\widehat{\mathbb{T}}_q(x) {\rm d}x\,  {\rm d}t+\gamma(\widehat{\mathbb{T}}_p,\widehat{\mathbb{T}}_q)_{L^2(\mathbb{R}^d)}\bigg\}
\\&=\dsum_{p\in\Upsilon_{\!N}}\hat{u}_p\bigg\{\delta_{pq} \,C_s     \dint_0^\infty \frac{t^{1-2s}\, |\lambda_p|_1}{  1+t^2|\lambda_p|_1}{\rm d}t +\gamma \delta_{pq}\bigg\}
\\&=\dsum_{p\in\Upsilon_{\!N}}\hat{u}_p\big(|\lambda_p|_1^s+\gamma\big)\delta_{pq}=(I_N^df,\widehat{\mathbb{T}}_q)_{L^2(\mathbb{R}^d)},
\end{split}\end{equation*}
which implies
\begin{equation*}\label{spectralalgord}
\hat{u}_p=\frac{(I_N^df,\widehat{\mathbb{T}}_p)_{L^2(\mathbb{R}^d)} }{\gamma+|\lambda_p|_1^s},\quad  \forall p\in \Upsilon_{\!N}.
\end{equation*}
 Thus, we obtain  \eqref{smsolurd}-\eqref{fppp}  immediately.
 \end{proof}
 \begin{remark}\label{newalgo} {\em It is crucial to use the Fourier-like  MCFs as the basis functions
 for both $u_N$ and $w_N,$ so that we can take the advantage of the bi-orthogonality and explicitly evaluate the integration in $t.$
In other words, under the Fourier-like basis, the stiffness matrix of the linear system of \eqref{auveqn23}-\eqref{bSeqn}  becomes a diagonal matrix of the form
\begin{equation}
\label{stifmatrmul2}
 \widehat{\bs{S}}:=\big(\bs{\Sigma}\otimes \bs{I}\otimes\cdots\otimes\bs{I}+\bs{I}\otimes\bs{\Sigma}\otimes \bs{I}\otimes\cdots\otimes\bs{I}+\cdots+\bs{I}\otimes\cdots\otimes\bs{I}\otimes\bs{\Sigma}\big)^s,
\end{equation}
where  $\bs{\Sigma}$ is defined in \eqref{eige} and $\otimes$ denotes the tensor product operator as  before.
 }
\end{remark}
%
\begin{remark} {\em  The main cost of solving the system \eqref{spectralalgord} is devoted to the evaluation of   the right-hand side, which can be carried out by the fast Fourier transform (FFT) related to Chebyshev polynomials.}
\end{remark}

\section{Error estimates and numerical examples}\label{sect4}

In this section,  we  derive some relevant MCF approximation results, which are useful for  the error estimates of the  proposed MCF spectral-Galerkin scheme.

\subsection{Approximation by MCFs}

We first consider  $d$-dimensional $L^2$-orthogonal  projection: $\pi_N^d : L^2(\mathbb R^d)$ $\to  \mathbb{V}^d_{\!N}$ such that
\begin{equation}
\label{orthproj1}
\dint_{\R^d}(\pi^d_N u-u)(x)v (x) \,{\rm d}x=0,\quad \forall v\in \mathbb{V}^d_{\!N}.
\end{equation}

We intend to estimate the projection error in the fractional Sobolev norm, i.e.,  $\|\pi^d_N u-u\|_{H^s(\mathbb R^d)}.$ For this purpose, we introduce some notation and spaces of functions.

For notational convenience, the pairs of functions $(u,\breve u)$ and $(U,\breve U)$ associated with the mapping \eqref{AlgeMapp} have the relations
 \begin{equation}\label{pair}
 \begin{split} u(x) &=U(y(x)), \quad \breve{u}(x)=\frac{u(x)}{g(x)}=\frac{U(y)}{G(y)}=\breve{U}(y),
 \end{split}
\end{equation}
where  as in \eqref{weigfunc}, $g(x)= (1+x^2)^{-1/2}=\sqrt{1-y^2}=G(y).$
Define the differential operators
\begin{equation}\label{Doper}
\begin{split}
&D_{x_j} u :=\partial_{x_j}\big\{(1+x^2_j)^{\frac12}u\big\}\frac{{\rm d}x_j}{{\rm d}y_j}=a(x_j)\partial_{x_j}\big\{(1+x^2_j)^{\frac12}u\big\}=\partial_{y_j}\breve U,
\\& D_{x_j}^{k_j} u=a(x_j)\partial_{x_j}\Big\{a(x_j) \partial_{x_j}\Big\{\cdots\Big\{a(x_j)\partial_{x_j}\Big\{(1+x^2_j)^{\frac12}u\Big\}\Big\} \cdots\Big\}\Big\}=\partial_{y_j}^{k_j}\breve U,\quad
\end{split}
\end{equation}
for $  k_j\ge 1, 1\le j\le d,$ where $a(x_j)=dx_j/dy_j=(1+x^2_j)^{\frac32}$. Correspondingly, we define the $d$-dimensional Sobolev space 
\begin{equation}
\label{Bspacex}
B^m(\R^d)=\big\{u:D^k_{x}u\in L^2_{\varpi^{k+1}}(\mathbb{R}^d),~0\leq | k |_1\leq m\big\},\quad m=0,1,2,\cdots,
\end{equation}
where the differential operator and the weight function are
\begin{equation}\label{dkxj}
D^{k}_{x}u=D^{k_1}_{x_1}\cdots D^{k_d}_{x_d}u,\quad
\varpi^{k}(x)=\prod_{j=1}^{d}(1+x_j^2)^{-k_j}.
\end{equation}
It is equipped with the norm and semi-norm
\begin{equation}
\label{Bnormex2}\begin{split}
 \|u\|_{B^m(\mathbb{R}^d)}=\Big(\dsum_{0\leq |k|_1\leq m}\big\|D^{k}_{x} u\big\|^2_{L^2_{\varpi^{1+k}}(\mathbb{R}^d)}\Big)^{\frac 12},\quad
  |u|_{B^m(\mathbb{R}^d)}=\Big(\dsum^d_{j=1}\big\|D^m_{x_j} u\big\|^2_{L^2_{\varpi^{1+me_j}}(\mathbb{R}^d)}\Big)^{\frac 12},
\end{split}\end{equation}
 where $e_j=(0,\cdots,1,\cdots,0)$ be the $j$-th unit vector in $\mathbb{R}^d$.


\begin{thm}\label{th3.1} If $u\in B^m(\mathbb{R}^d)$ with integer $m\geq 1,$ then  we have
\begin{equation}
\label{numerr1}
\|\pi^d_N u-u\|_{H^{s}(\mathbb{R}^d)}\leq
cN^{s-m}|u|_{B^m(\mathbb{R}^d)},\quad 0\le s\le 1,
\end{equation}
where $c$ is a positive constant independent of $N$ and $u.$
\end{thm}
 \begin{proof}
 In view of \eqref{orthNs22},  it is necessary  to estimate the projection errors  in the $L^2$- and $H^1$-norms.
 According  to \cite[Theorems 3.1-3.2]{shen2014approximations}, we have
 \begin{equation}
\label{MCFmul1}
\|\pi_N^du-u\|_{L^2(\R^d)}\leq cN^{-m}|u|_{B^m(\R^d)},
\end{equation}
and
\begin{equation}
\label{MCFmul2}
\|\nabla(\pi^d_Nu-u)\|_{L^2(\R^d)}\leq cN^{1-m}|u|_{B^m(\R^d)}.
\end{equation}
Using Lemma \ref{interpola} and  \eqref{MCFmul1}-\eqref{MCFmul2}, we arrive at
\begin{equation*}
\|\pi^d_N u-u\|_{H^{s}(\mathbb{R}^d)}\leq
cN^{s-m}|u|_{B^m(\mathbb{R}^d)},\quad s\in[0,1].
\end{equation*}
This ends the proof.
 \end{proof}

We now turn to the error estimate for the interpolation operator.
Let $\{y_j,\rho_j\}_{j=0}^N$ be the Chebyshev-Gauss quadrature nodes and weights on $\Lambda=(-1,1)$. Denote the mapped nodes and weights by
\begin{equation}\label{mapnodes}
x_j=\frac{y_j}{\sqrt{1-y_j^2}}, \quad \omega_j=\frac{\rho_j}{1+y_j^2}, \quad 0 \leq j \leq N.
\end{equation}
Then by the exactness of the Chebyshev-Gauss quadrature, we have
\begin{equation}
\begin{split}
\int_{\mathbb{R}} u(x) v(x)\, {\rm d} x &=\int_{\Lambda} U(y) V(y)  (1-y^2)^{-\frac32}\, {\rm d} y=\int_{\Lambda}  \breve{U}(y)\breve{V}(y)(1-y^2)^{-\frac12}\, {\rm d}y\\ &=\sum_{j=0}^{N} \breve{U}(y_j)\breve{V}(y_j) \rho_{j}, \quad \forall\,\breve{U}\cdot\breve{V} \in \mathbb{P}_{2N+1}.
 \end{split}
\end{equation}
which, together with \eqref{apprspac}, implies the exactness of quadrature
\begin{equation}\label{quad2}
\int_{\mathbb{R}} u(x) v(x)\, {\rm d} x=\sum_{j=0}^{N} u(x_j) v(x_j) \omega_j, \quad \forall u\cdot v \in \mathbb{V}_{2N+1}.
\end{equation}
We now introduce the one-dimensional interpolation operator $I_N:C(\mathbb{R})\rightarrow\mathbb{V}_{\!N}$ such that
\begin{equation}
I_{N} u(x_j)=u(x_j), \quad 0 \leq j \leq N.
\end{equation}

With a little abuse of notation, we define the $d$-dimensional grids by
$x_j=(x_{j_1}, \cdots, x_{j_d}),\, j\in \Upsilon_{\!N},$
where $\{x_{j_k}\}_{k=1}^d$ are the mapped Chebyshev-Gauss nodes, and the index set  $\Upsilon_{\!N}$ is given in \eqref{UpsilonA}  as before.
We now consider the $d$-dimensional  MCF interpolation: $C(\mathbb R^d)\to {\mathbb V}^d_{\!N},$
\begin{equation}\label{inter2}
I^d_{N} u(x_j)=I_N^{(1)}\circ \cdots \circ I_N^{(d)} u(x_j), \quad  j\in \Upsilon_{\!N},
\end{equation}
where $I_N^{(k)}=I_N, 1\le k\le d,$ is the interpolation along $x_k$-direction.

In the error analysis, we also need the $L^2$-estimate of the  $d$-dimensional MCF interpolation.
For better description of the error, we introduce a second  semi-norm of   $B^m(\mathbb{R}^d)$ for $m\ge 1$ as follows
\begin{equation}\label{neweqn}
\begin{split}
[[u]]_{B^m(\mathbb R^d)} :=\bigg\{|u|_{B^m(\mathbb R^d)}^2
 +\sum_{j=1}^d  \sum_{k\not =j}\big\|D_{x_k}^{m-1}D_{x_j} u\big\|_{L^d_{\varpi^{1+(m-1)e_k+e_j}}(\mathbb{R}^d)}^2\bigg\}^{\frac 1 2},
\end{split}
\end{equation}
where the weight function $\varpi$ and $|u|_{B^m(\mathbb R^d)}$ are defined in \eqref{dkxj} and
\eqref{Bnormex2} as before.

We have the following $L^2$-estimates of the interpolation.
 \begin{thm}\label{th3.2} For $u\in B^m(\mathbb{R}^d)$ with the integer $m\geq 2,$ we have
\begin{equation}
\label{Intererr20}
\|I^d_N u-u\|_{L^2(\mathbb{R}^d)}\leq cN^{-m}[[u]]_{B^m(\mathbb{R}^d)},
\end{equation}
where $c$ is a positive constant independent of $N$ and $u.$
\end{thm}
\begin{proof}  Let  $I^C_{N}: C(\Lambda)\to  \mathbb{P}_{N}$ be the  Chebyshev-Gauss interpolation operator.  According to  \cite[Lemma 3.6]{shen2011spectral}, we have that for any $v\in L^2_{\omega}(\Lambda)$ and $v'\in L^2_{\omega^{-1}}(\Lambda)$ with $\omega(y)=(1-y^2)^{-\frac12}$,
\begin{equation}\label{stab1}
\|I_{N}^C v\|_{L^2_{\omega}(\Lambda)} \leq c\big(\|v\|_{L^2_{\omega}(\Lambda)}+N^{-1}\|(1-y^2)^{\frac12} v'\|_{L^2_{\omega}(\Lambda)}\big),
\end{equation}
where $c$ is a positive  constant independent of $N$ and $v.$ Moreover,  by  \cite[Theorem 3.41]{shen2011spectral},  we have the one-dimensional Chebyshev-Gauss interpolation error estimates,
\begin{equation}\label{estmate1}
\|(I_{N}^C v-v)'\|_{L^2_{\omega^{-1}}(\Lambda)} +N \|I_{N}^C v-v\|_{L^2_{\omega}(\Lambda)}
\leq cN^{1-m}\|(1-y^2)^{\frac{m}{2}} v^{(m)}\|_{L^2_{\omega}(\Lambda)}.
\end{equation}

%
%
%

In view of  \eqref{AlgeMapp}, we have
\begin{equation}
 \|I_{N}^d u-u\|_{L^{2}(\mathbb{R}^{d})}=\big\|I_{N}^{C,d}(U/G)-(U/G)\big\|_{L_{\omega}^{2}(\Lambda^{d})}=\big\|I_{N}^{C,d}\breve U-\breve U\big\|_{L_{\omega}^{2}(\Lambda^{d})},
\end{equation}
where $G(y)=\prod_{j=1}^{d}G(y_j)$ and $I_{N}^{C,d}:=I_{N}^{C,(1)}\circ\cdots \circ
 I_{N_d}^{C,(d)}$ with $I_N^{C,(k)}=I_N^C.$

 For clarity, we only prove the results with $d=2,$ as it is straightforward to extend the results to  the case with $d\ge 3.$
By virtue of the triangle inequality,  \eqref{stab1} and \eqref{estmate1}, we obtain that for $m\ge 2,$
\begin{equation}\label{proofL1x}
\begin{split}
\| I^{C,2}_N\breve U  - \breve U\|_{L^2_\omega(\Lambda^2)}& \leq \|I^{C,(1)}_{N} \breve U-\breve U\|_{L^2_\omega(\Lambda^2)}+\big\|I^{C,(1)}_{N}\circ\big(I^{C,(2)}_{N}\breve U-\breve U\big)\big\|_{L^2_\omega(\Lambda^2)}
\\&\leq c N^{-m}  \|(1-y^2_1)^{\frac{m}{2}} \partial^m_{y_1} \breve U\|_{L^2_{\omega}(\Lambda^2)} + c\Big\{\|I^{C,(2)}_{N} \breve U-\breve U\|_{L^2_\omega(\Lambda^2)}
\\&\quad +N^{-1}\big\|(1-y_1^2)^{\frac12}\partial_{y_1}\big(I^{C,(2)}_{N}\breve U-\breve U\big)\big\|_{L^2_\omega(\Lambda^2)} \Big\}
\\&\le cN^{-m} \Big\{ \big\|(1-y^2_1)^{\frac{m}{2}} \partial^m_{y_1}\breve U\big\|_{L^2_{\omega}(\Lambda^2)} +
\big\|(1-y^2_2)^{\frac{m}{2}} \partial^m_{y_2}\breve U\big\|_{L^2_{\omega}(\Lambda^2)}\\
& \quad + \big\|(1-y_1^2)^{\frac12}(1-y_2^2)^{\frac{m-1}{2}} \partial_{y_1}\partial_{y_2}^{m-1}\breve U\big\|_{L^2_\omega(\Lambda^2)}  \Big\}.
\end{split}
\end{equation}
Note that in the above derivation, we can switch the order of $I_N^{C,(1)}$ and $I_N^{C,(2)},$ so  we can add the term
$\big\|(1-y_1^2)^{\frac{m-1}2}(1-y_2^2)^{\frac{1}{2}} \partial_{y_1}^{m-1}\partial_{y_2}\breve U\big\|_{L^2_\omega(\Lambda^2)} $ in the upper bound.
Then, we  derive from  \eqref{pair},  \eqref{Doper} and \eqref{proofL1x} that for $m\ge 2,$
\begin{equation*}\label{proofL1r}
\begin{split}
&\| I^{C,2}_N \breve U -\breve U\|_{L^2_\omega(\Lambda^2)}
\\& \le cN^{-m} \Big\{ \big\|(1+x^2_1)^{-\frac{m+1}{2}} (1+x^2_2)^{-\frac{1}{2}}D^m_{x_1} u\big\|_{L^2(\mathbb{R}^2)} +
\big\|(1+x^2_1)^{-\frac12}(1+x^2_2)^{-\frac{m+1}{2}}D^m_{x_2} u\big\|_{L^2(\mathbb{R}^2)}\\
& \quad +\big\|(1+x_1^2)^{-1}(1+x_2^2)^{-\frac{m}{2}}D_{x_1}D_{x_2}^{m-1} u\big\|_{L^2(\mathbb{R}^2)}
\\&
\quad  +\big\|(1+x_1^2)^{-\frac{m}{2}}(1+x_2^2)^{-1}D_{x_1}^{m-1}D_{x_2} u\big\|_{L^2(\mathbb{R}^2)} \Big\}\le cN^{-m} [[u]]_{B^m(\mathbb R^2)}.
\end{split}
\end{equation*}
It is straightforward to extend the above derivation to $d\ge 3.$ This completes the proof.
\end{proof}

To conduct error analysis  for the proposed  MCF scheme,  we  assume that the error for solving the elliptic problem \eqref{bSeqn}
is negligible (or equivalently, the quadrature errors in evaluating the fractional Laplacian in the scheme can be ignored), so formally, we have $w_N= (\mathbb I-t^2 \Delta)^{-1} u_N. $
It is noteworthy that the analysis of such an error is feasible for the finite element approximation of the fractional Laplacian in bounded domain  based on the Dunford-Taylor formulation in a bounded domain,  though the proof is  lengthy and much involved (see  \cite{bonito2019numerical}). However, the
analysis is largely open in this situation,   mainly because  the spectrum estimate of the fractional Laplacian operator in $\mathbb R^d$ appears unavailable, as opposite to the bounded domain case (see  \cite{bonito2019numerical}).
\begin{prop}\label{th3.3w}  Assume that
the elliptic problem \eqref{bSeqn}  in the scheme \eqref{auveqn23} can be solved exactly.  Then we have the estimate: for $u\in B^{m_1}(\mathbb{R}^d)$ with integer $m_1\geq1,$ and $f\in B^{m_2}(\mathbb{R}^d)$ with integer $m_2\geq2,$
%
\begin{equation}
\label{Hermerr1w}\begin{split}
\|u-u_N\|_{H^{s}(\mathbb{R}^d)}\leq& cN^{s-m_1}|u|_{B^{m_1}(\mathbb{R}^d)}+cN^{-m_2}[[f]]_{B^{m_2}(\mathbb{R}^d)},\quad  s\in (0,1),
\end{split}\end{equation}
where $c$ is a positive constant independent of $u,f$ and $N.$
\end{prop}
\begin{proof}  Under this assumption,  we find from Lemma \ref{DTthm} that
 the scheme  \eqref{auveqn23} can be written as: find  $u_N\in \mathbb V_{\!N}^d$ such that
 $$
 \mathcal B(u_N, v_N)=(I_N^d f, v_N),\quad \forall\, v_N\in \mathbb V_{\!N}^d.
 $$
  Then by \eqref{uvsh} and \eqref{contcoer}, we infer from a standard argument that
  \begin{equation*}
  \| u-u_N\|_{H^s(\mathbb R^d)}\le c(\|\pi_N^d u-u\|_{H^s(\mathbb R^d)} +\|I_N^d f-f\|_{L^2(\mathbb R^d)}).
  \end{equation*}
  Thus, the estimate  \eqref{Hermerr1w} follows from   Theorems \ref{th3.1} and \ref{th3.2}  immediately.
    \end{proof}
    \begin{remark}  {\em In what follows, we shall validate the above assumption  through  several numerical  tests. Indeed, we shall observe that the errors of solving  \eqref{bSeqn} are insignificant, and the order of the numerical errors perfectly agrees with the estimated order. }
    \end{remark}


\subsection{Useful  analytic formulas}  We first derive  analytical formulas for fractional Laplacian  of
some functions with typical exponential or algebraic decay,  upon which we construct  the exact solutions  to  test the accuracy of the proposed method,  and to validate the assumption in Proposition \ref{th3.3w}. Moreover, we reveal that the fractional Laplacian has a very
different property from the usual Laplacian. For example, the image of an exponential function  decays  algebraically  (see Proposition  \ref{case1:expo} below), as opposite to the usual one.


We have the following exact formulas for the Gaussian function  and rational function, whose derivations are sketched in Appendix \ref{AppendixA} and Appendix \ref{AppendixB}, respectively.  
\begin{prop}\label{case1:expo} For real $s>0$ and integer $d\geq 1$, we have that
  \begin{equation}\label{2Dcase}
(-\Delta)^s \big\{e^{-|x|^2}\big\}
 =\frac{ 2^{2s} \Gamma(s+d/2)}{\Gamma(d/2)}\, {}_{1} F_{1}\Big(s+\frac{d}{2};\frac{d}{2};-|x|^2\Big).
 \end{equation}
 Moreover for non-integer $s>0$ and $|x|\to\infty$, we have the asymptotic behaviour
 \begin{equation}\label{asymexpo}
 (-\Delta)^s \big\{e^{-|x|^2}\big\}=-\frac{2^{2s}\sin(\pi s)}{\pi}\frac{\Gamma(s+d/2)\Gamma(1+s)}{|x|^{d+2s}}\big\{1+O(|x|^{-2})\big\}.
 \end{equation}
\end{prop}


\begin{prop}\label{case1:alg}
For real $s,r>0,$ and integer $d\geq 1$, we have
\begin{equation}\label{FLalge}
(-\Delta)^{s}\Big\{\frac{1}{(1+|x|^2)^{r}}\Big\}=\frac{2^{2s}\Gamma(s+\gamma)\Gamma(s+d/2)}{\Gamma(\gamma)\Gamma(d/2)}{}_{2}F_{1}\Big(s+r,s+\frac{d}{2};\frac{d}{2};-|x|^2\Big).
\end{equation}
Moreover for non-integer $s>0$ and $|x|\to\infty$, we have the asymptotic properties: for  $r\not=d/2,$
\begin{equation}\label{gmaao}
(-\Delta)^{s}\Big\{\frac{1}{\left(1+|x|^2\right)^{r}}\Big\}\sim \frac{1} {(1+|x|^2)^{s+\mu}},
\end{equation}
where $\mu=\min\{r, d/2\}.$ if $r=d/2,$ then
\begin{equation}\label{neweqnAln}
(-\Delta)^{s}\Big\{\frac{1}{(1+|x|^2)^r}\Big\}\sim \frac{\ln (1+|x|^2)} {(1+|x|^2)^{s+d/2}};
\end{equation}
%
\end{prop}

\subsection{Numerical results}\label{subsect43}
We apply the MCF spectral-Galerkin method to solve the model problem \eqref{dDsLap} in various situations.
\begin{example}\label{some_exact} {\bf (Accuracy test).} We first consider \eqref{dDsLap} with the following exact solutions:
 \begin{equation}\begin{split}\label{test1}
 &u_e(x)=e^{-|x|^2},\quad u_a(x)=(1+|x|^2)^{-r},\quad r>0,\;\; x\in {\mathbb R^d}.
  \end{split}\end{equation}
In view of  \eqref{2Dcase} and \eqref{FLalge}, the source terms $f_e(x)$ and $f_a(x)$ are respectively given by
\begin{equation}\begin{split}\label{fasy}
&f_e(x)= \gamma e^{-|x|^2}+\frac{ 2^{2s} \Gamma(s+d/2)}{\Gamma(d/2)}\, {}_{1} F_{1}\Big(s+\frac{d}{2};\frac{d}{2};-|x|^2\Big) ,
\\&f_a(x)= \gamma (1+|x|^2)^{-r}+\frac{2^{2s}\Gamma(s+r)\Gamma(s+d/2)}{ \Gamma(r)\Gamma(d/2)} {}_2F_1\Big(s+r,s+\frac{d}{2};\frac{d}{2};-|x|^2\Big).
\end{split}\end{equation}



 \begin{figure}[!th]
\subfigure[$d=1$ and $u_e(x)=e^{-x^2}$]{
\begin{minipage}[t]{0.42\textwidth}
\centering
\rotatebox[origin=cc]{-0}{\includegraphics[width=0.9\textwidth]{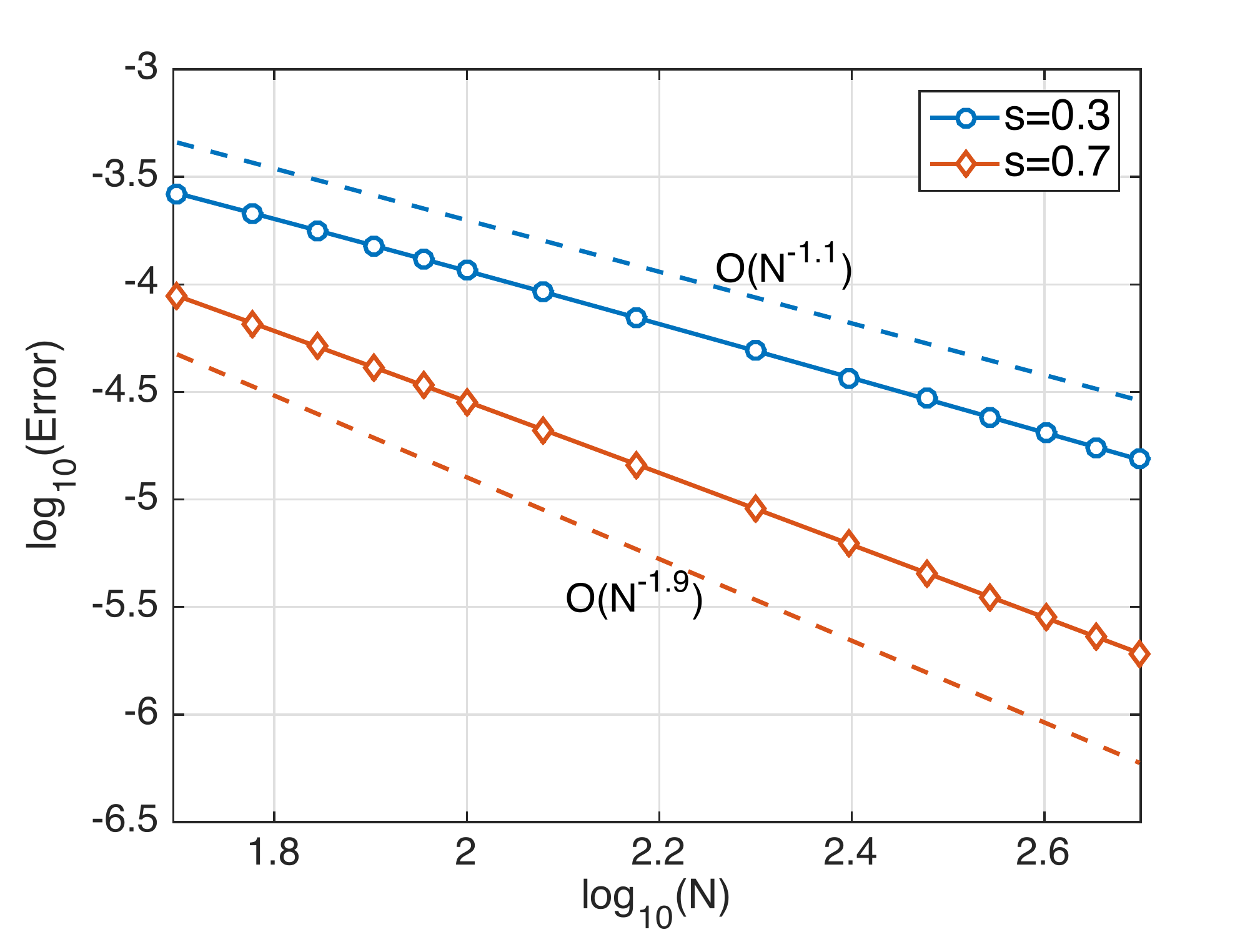}}
\end{minipage}}
\subfigure[$d=1$ and $u_a(x)=(1+|x|^2)^{-2.3}$]{
\begin{minipage}[t]{0.42\textwidth}
\centering
\rotatebox[origin=cc]{-0}{\includegraphics[width=0.9\textwidth]{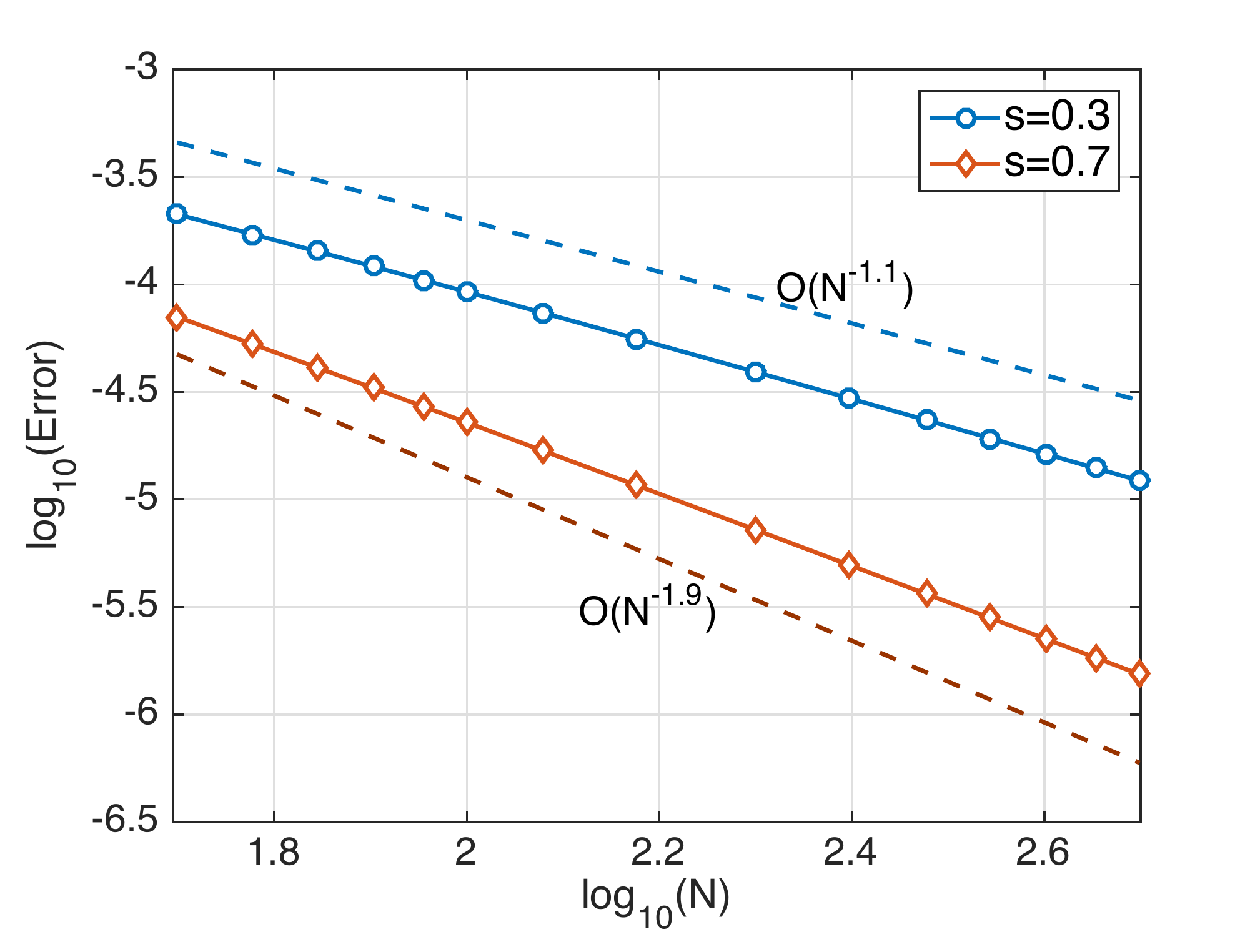}}
\end{minipage}}\vskip -5pt

\subfigure[$d=2$ and $u_e(x)=e^{-|x|^2}$]{
\begin{minipage}[t]{0.42\textwidth}
\centering
\rotatebox[origin=cc]{-0}{\includegraphics[width=0.9\textwidth]{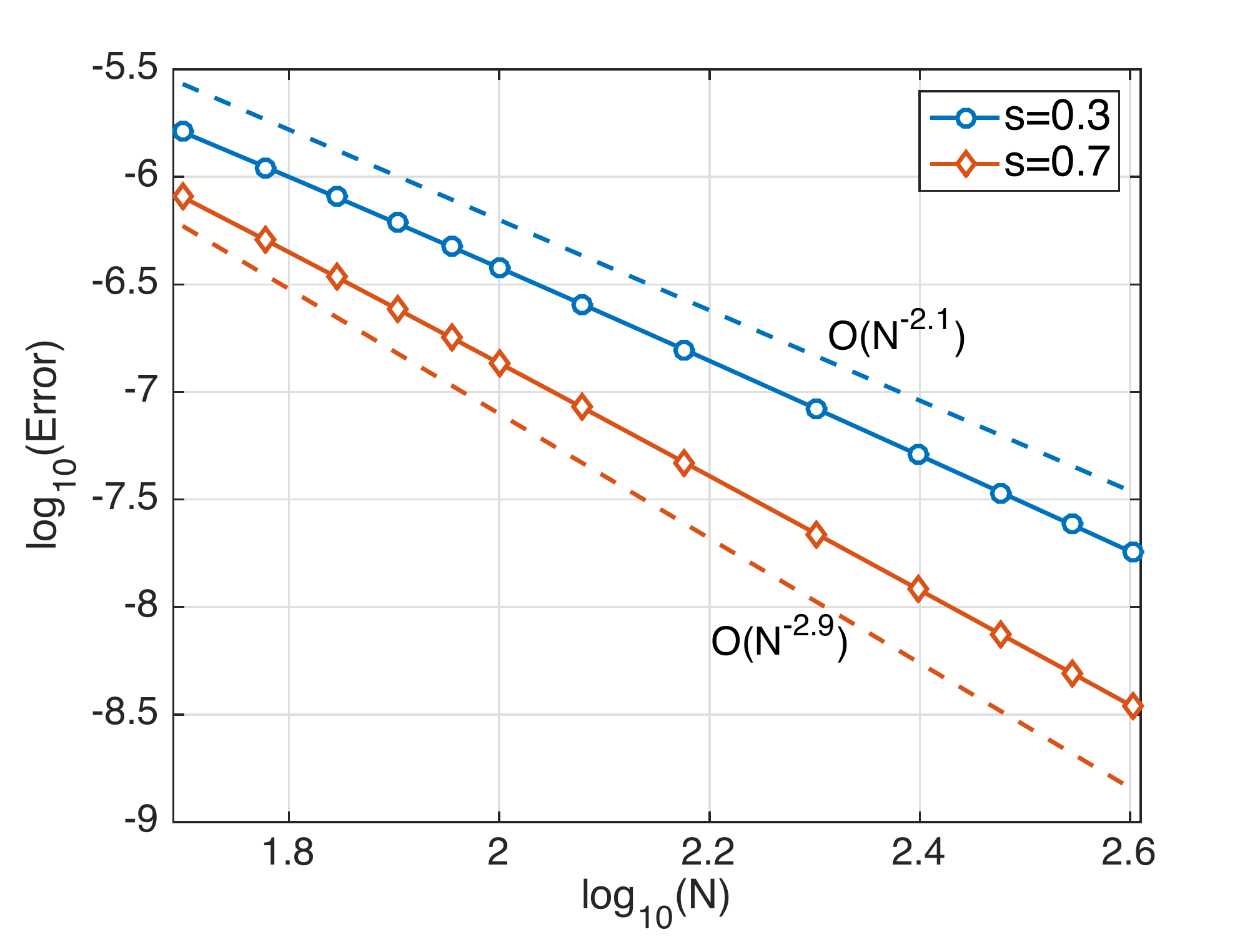}}
\end{minipage}}
\subfigure[$d=2$ and $u_a(x)=(1+|x|^2)^{-2.3}$]{
\begin{minipage}[t]{0.42\textwidth}
\centering
\rotatebox[origin=cc]{-0}{\includegraphics[width=0.9\textwidth]{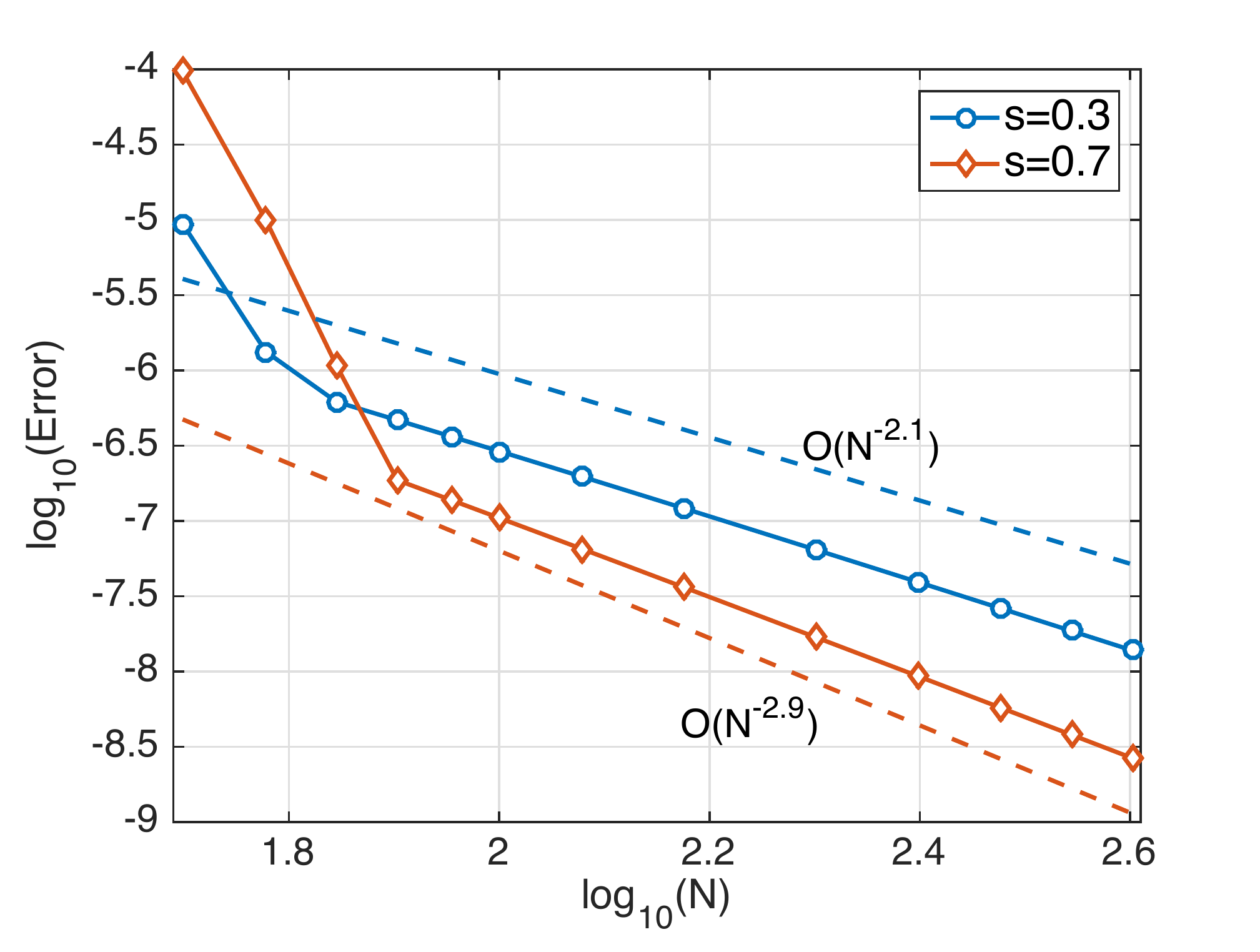}}
\end{minipage}}\vskip -5pt

\subfigure[$d=3$ and $u_e(x)=e^{-|x|^2}$]{
\begin{minipage}[t]{0.42\textwidth}
\centering
\rotatebox[origin=cc]{-0}{\includegraphics[width=0.9\textwidth]{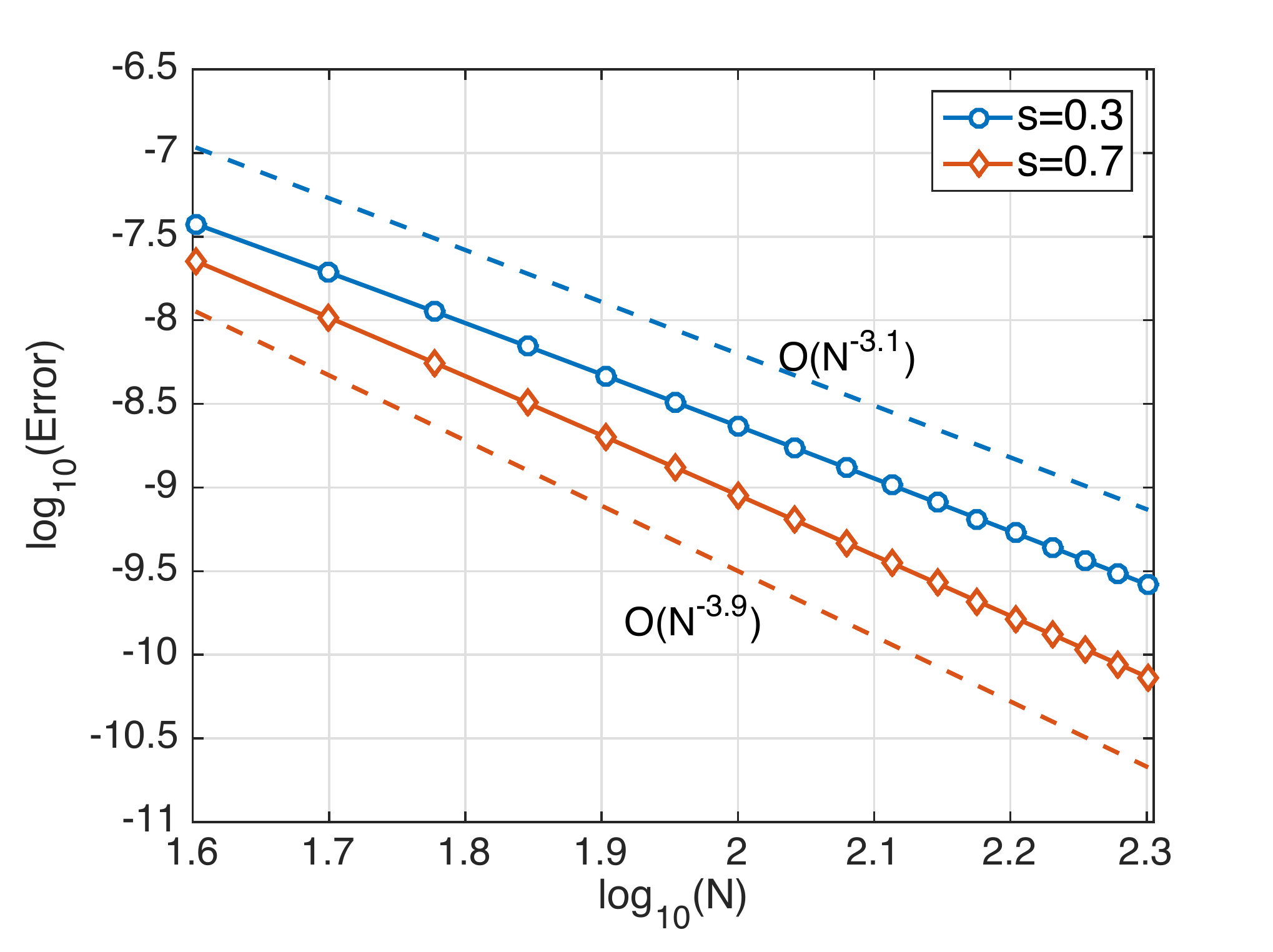}}
\end{minipage}}
\subfigure[$d=3$ and $u_a(x)=(1+|x|^2)^{-2.3}$]{
\begin{minipage}[t]{0.42\textwidth}
\centering
\rotatebox[origin=cc]{-0}{\includegraphics[width=0.9\textwidth]{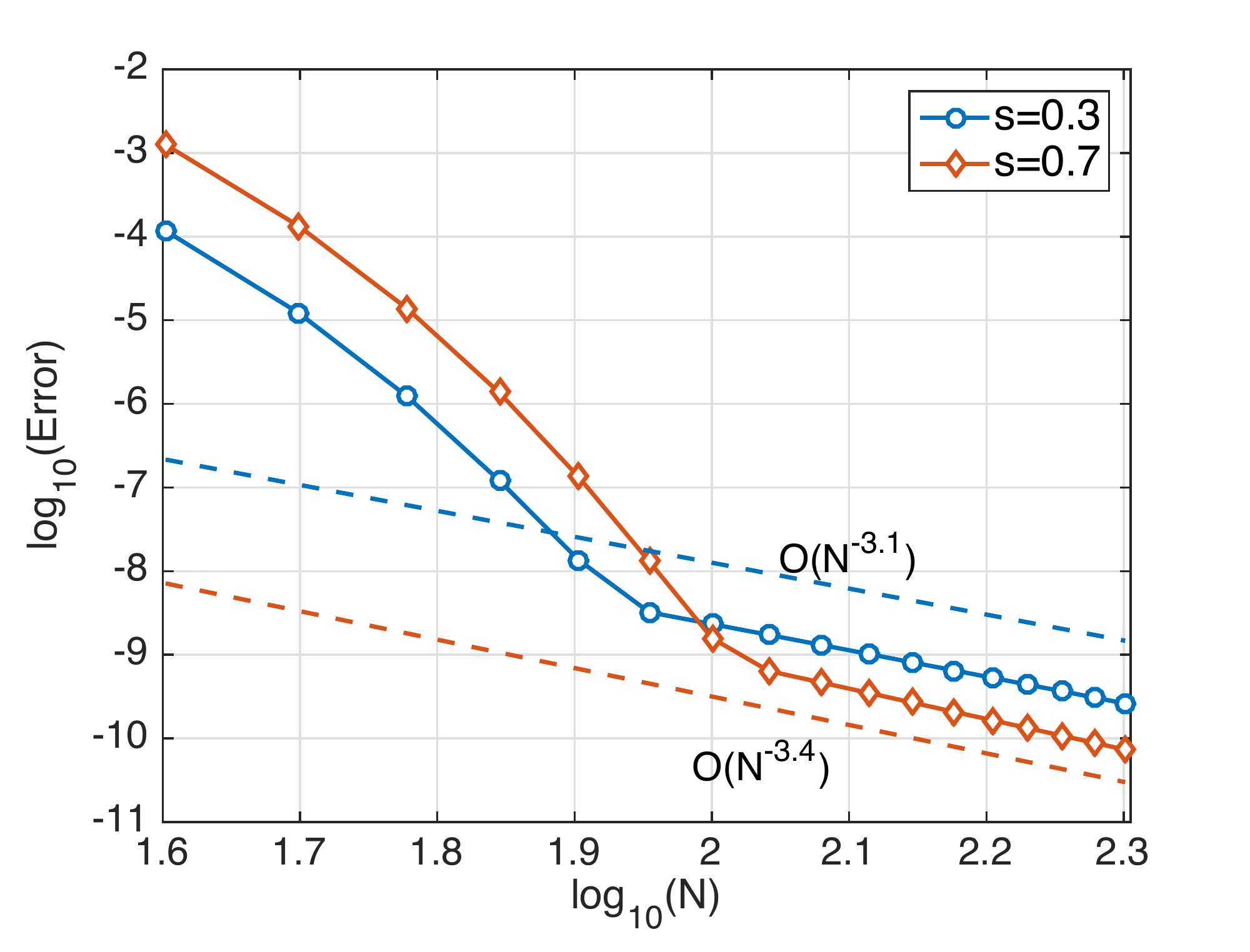}}
\end{minipage}}
\caption
{\small  Decay of $H^s$-errors of the MCF scheme with $\gamma=1$ and the scaling factor $\nu=2.5$ for Example \ref{some_exact}  with exact solutions in \eqref{test1}. Here  $s=0.3,~0.7$ and $r=2.3$.
The dashed reference lines are expected orders   predicted by Proposition \ref{th3.3w}.
  }\label{Example1}
\end{figure}

Now, we intend to use the error estimates in  Propositions \ref{case1:expo} and \ref{case1:alg} to analytically calculate the expected order of convergence by the MCF scheme, and then verify the convergence order numerically.
For this purpose, we consider a generic function of algebraic decay as follows
\begin{equation}\label{formuf}
 w(x)=\frac{1}{(1+|x|^2)^\mu},\quad x\in \mathbb R^d,\;\; \mu>0,
\end{equation}
Using the definitions \eqref{Doper} and \eqref{neweqn}, we obtain from   direct calculation that
\begin{equation}
\begin{split}
D_{x_j}w(x)&=(1+x^2_j)^{\frac32}\partial_{x_j}\big\{(1+x^2_j)^{\frac12}w(x)\big\} =(1+x_j^2)(1+|x|^2)^{-\mu-1}(-2\mu x_j(1+x_j^2)
\\&\quad +x_j(1+|x|^2))\sim |x_j|^{-2\mu+3}\prod_{ l\neq j}|x_l|^{-2\mu},
\end{split}
\end{equation}
and similarly,
\begin{equation*}
\begin{split}
&D_{x_k}D_{x_j}w(x)
\sim  |x_k|^{-2\mu+3}|x_j|^{-2\mu+3}\prod_{l\neq j}|x_l|^{-2\mu},
\;\;\;  D_{x_k}^2D_{x_j}w(x)\sim |x_k|^{-2\mu+5} |x_j|^{-2\mu+3}\prod_{l\neq j,k}|x_l|^{-2\mu}.
\end{split}
\end{equation*}
By an induction argument, we can show
\begin{equation}\label{D435}\begin{split}
&D_{x_k}^{m-1}D_{x_j}w(x)\sim |x_k|^{-2\mu+2m-1} |x_j|^{-2\mu+3}\prod_{ l\neq j,k}|x_l|^{-2\mu}, \quad j\neq k,
\\&D_{x_k}^{m}w(x)\sim |x_k|^{-2\mu+2m+1}\prod_{l\neq k}|x_l|^{-2\mu}.
\end{split}\end{equation}
Thus, for $j\neq k$, we have
\begin{equation}\label{D436}\begin{split}
&I_{kj}(x):=|D_{x_k}^{m-1}D_{x_j}w(x)|^{2}\varpi^{1+(m-1)e_k+e_j}(x)
\sim |x_k|^{-4\mu+2m-2} |x_j|^{-4\mu+4}\prod_{l\neq j,k}|x_l|^{-2\mu-2},
\end{split}\end{equation}
and for $1\le k\le d,$
\begin{equation}\label{D4366}\begin{split}
I_{kk}(x):=|D_{x_k}^{m}w(x)|^{2}\varpi^{1+me_k}(x)\sim |x_k|^{-4\mu+2m} \prod_{l\neq k}|x_l|^{-2\mu-2}. 
\end{split}\end{equation}
Then by \eqref{Bnormex2}, \eqref{neweqn} and \eqref{D4366},  we find that   if $m<2\mu-\frac12,$
then
\begin{equation}\label{neweqnBA}
\begin{split}
&  |w|_{B^m(\mathbb{R}^d)}^2=\dsum^d_{k=1}\int_{\mathbb R^d} I_{kk}(x){\rm d}x <\infty,  \;\;  [[w]]_{B^m(\mathbb R^d)}^2 =|w|_{B^m(\mathbb R^d)}^2
 +\sum_{j=1}^d  \sum_{k\not =j}\int_{\mathbb R^d} I_{kj}(x){\rm d}x <\infty.
 \end{split}
\end{equation}


For the exact solution  $u_e(x)=e^{-|x|^2},$  we have from \eqref{asymexpo} and \eqref{fasy} that  $f_e(x)\sim (1+ |x|^2)^{s+d/2}.$ Therefore,  in this case, the error is dominated by the MCF interpolation approximation of $f_e(x).$
Therefore, using  \eqref{neweqnBA} with $\mu=s+d/2,$ we conclude from Proposition \ref{th3.2} that the expected convergence is  $O(N^{-(2s+d)+1/2+\varepsilon})$ for small $\varepsilon>0.$
Remarkably,  the numerical results in  Figure \ref{Example1} (a), (c), (e) perfectly agree with the theoretical prediction (see the dashed reference lines).  Indeed, it is very different from the usual Laplacian (see Proposition \ref{case1:expo}),  we do not expect the exponential convergence, but  algebraic decay of the errors.

We now turn to the second case with the exact solution $u_a(x)=(1+|x|^2)^{-r},$ where we take $r=2.3$ in the numerical tests.  As $r>d/2,$ we derive from Proposition \ref{case1:alg} that
$f_a(x) \sim (1+ |x|^2)^{s+d/2}.$ Then by  \ref{neweqnBA} and Proposition \ref{th3.2}, we have the convergence behaviour
$$
\|u_a-u_N\|_{H^s(\mathbb R^d)}= O(N^{s-m_1})+O(N^{-m_2}),\quad m_1<2r-\frac 1 2, \;\; m_2<2s+d-\frac 1 2.
$$
This implies the  convergence order: $O(N^{-\min\{2r-s, 2s+d\}+\frac 1 2+\varepsilon}).$
Indeed, we observe from Figure \ref{Example1} (b), (d), (f) a perfect agreement again.  For example,
for $ d=3$ and $r=2.3,$ we have the rate $O(N^{-3.1+\varepsilon})$ for $s=0.3, $ while  the rate  $O(N^{-3.4+\varepsilon})$ for $s=0.7.$ Interestingly, there is a pre-asymptotic range where one observes a sub-geometric convergence from Figure \ref{Example1}  (d), (f)  for $d=2,3$,
but after the pre-asymptotic range, the convergence rates become algebraic as predicted.
Such a phenomenon has been also observed for the Laguerre function approximation  (cf. \cite[P. 278]{shen2011spectral}).

We highlight that the above numerical tests validate

\end{example}

\begin{figure}[!th]
\centering
\subfigure[Different scaling factors]{
\begin{minipage}[t]{0.45\textwidth}
\centering
\rotatebox[origin=cc]{-0}{\includegraphics[width=0.9\textwidth]{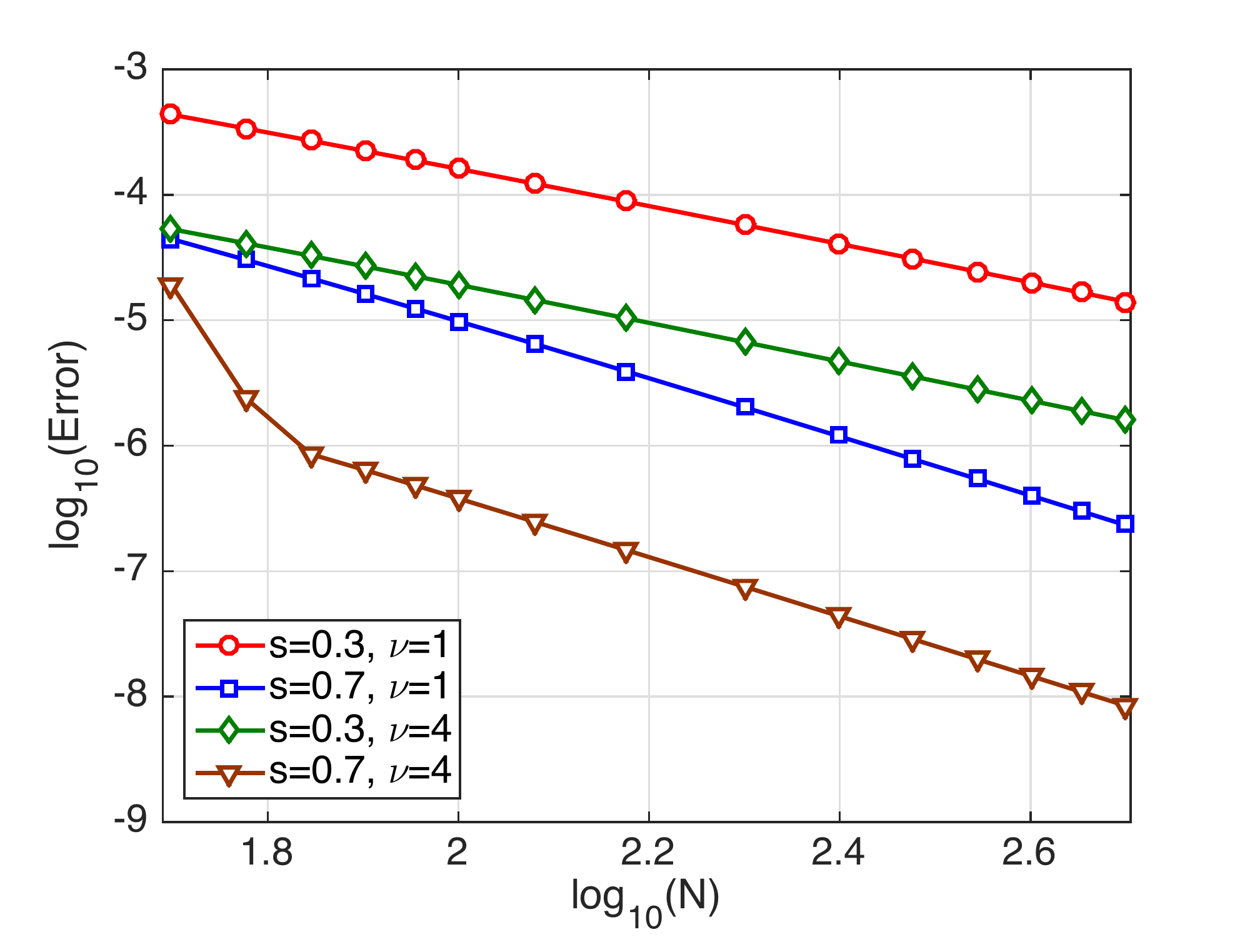}}
\end{minipage}}
\subfigure[Comparison with Hermite approach]{
\begin{minipage}[t]{0.45\textwidth}
\centering
\rotatebox[origin=cc]{-0}{\includegraphics[width=0.9\textwidth]{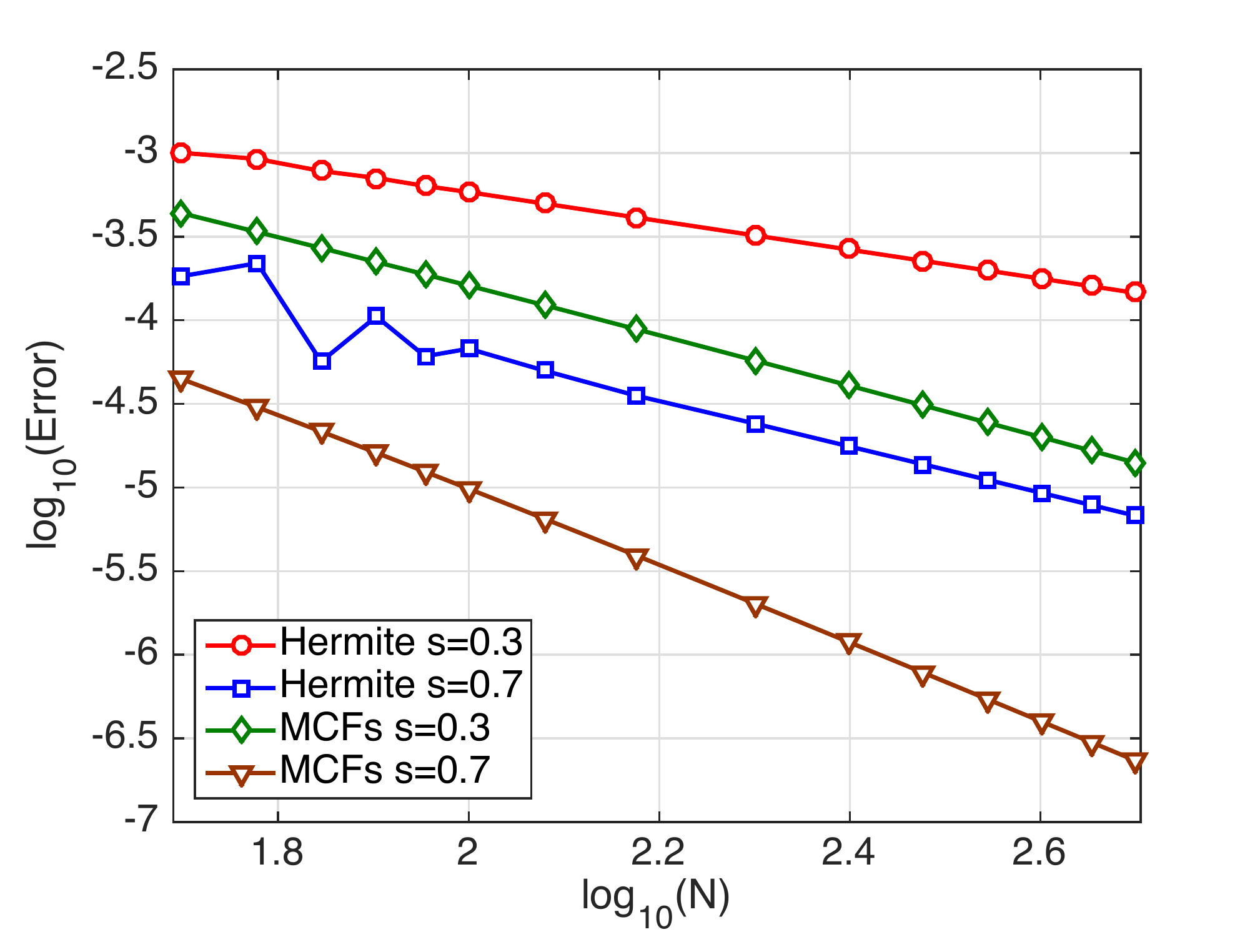}}
\end{minipage}}\vskip -10pt
\caption
{\small (a). Maximum error for the exact solution $u(x)=(1+x^2)^{-2.3}$ with different scaling factor $\nu$, and $s=0.3,~0.7$;
  (b). A comparison of maximum error between our method and Hermite-Galerkin method \cite{mao2017hermite} for exact solution $u(x)=(1+x^2)^{-2.3}$, with $s=0.3,~0.7$, $\nu=1$.}\label{fig1dexact}
\end{figure}

\begin{example}\label{1dcase_exact} {\bf (Effect of the scaling factor).}
  In this example, we first show the influence of the scaling factor $\nu$ to the accuracy.
 It is  known that, with a proper choice of the scaling parameter, the accuracy of spectral method on the unbounded domain can all be improved.
Here,  we plot in Figure \ref{fig1dexact} (a) the maximum error in log-log scale of our MCF algorithm  with different scaling parameter $\nu$.
 We observe that for any fixed $s$, the two error curves are nearly  parallel, which implies a proper scaling will improve the accuracy, but will not change the convergence rate. In Figure \ref{fig1dexact} (b), we compare the maximum errors of our algorithm using MCFs as basis functions with the Hermite spectral method in \cite{mao2017hermite}, for which we take $r=2.3$. As we can see from Figure  \ref{fig1dexact} (b) that the convergence rates of our approach are faster than that of the Hermite spectral method in \cite{mao2017hermite}.
\end{example}
%


\begin{example}\label{1dcase_f}{\bf (Accuracy for given source term $f(x)$).} Here, we further compare our MCF method with the Hermite function approximation in \cite{mao2017hermite}, where the tests were provided for   given  source terms with unknown solutions. 
  We therefore  compute the reference ``exact"  solutions with large  $N=600$.
In Figure \ref{Example3} (a)-(c), we compare the $L^2$-errors of our algorithm with the Hermite spectral method in \cite{mao2017hermite} in one and two dimensions.  It is noteworthy that the algorithm in  \cite{mao2017hermite} is computationally prohibitive for $d=3.$  In all cases, our approach outperforms the Hermite method in both accuracy and efficiency.
We report in  Figure \ref{Example3} (d)-(f)  the maximum point-wise errors against various $N$ with $d=2,3.$ The MCF method performs consistently  well.

We also tabulate  in Table \ref{tab1} the $L^2$-errors and  the convergence orders  of two methods  (see Table 2 and 3 of \cite{mao2017hermite} for the data of the Hermite method).
Here,  $f(x)=(1+x)e^{-\frac{x^2}{2}}$, $s=0.6,~0.9$ and $\nu=2.5$. Observe that the MCF method possesses higher convergence rates.

\end{example}

 \begin{figure}[!th]
\subfigure[$d=1$ and  $f(x)=(1+x)e^{-\frac{x^2}{2}}$]{
\begin{minipage}[t]{0.42\textwidth}
\centering
\rotatebox[origin=cc]{-0}{\includegraphics[width=0.9\textwidth]{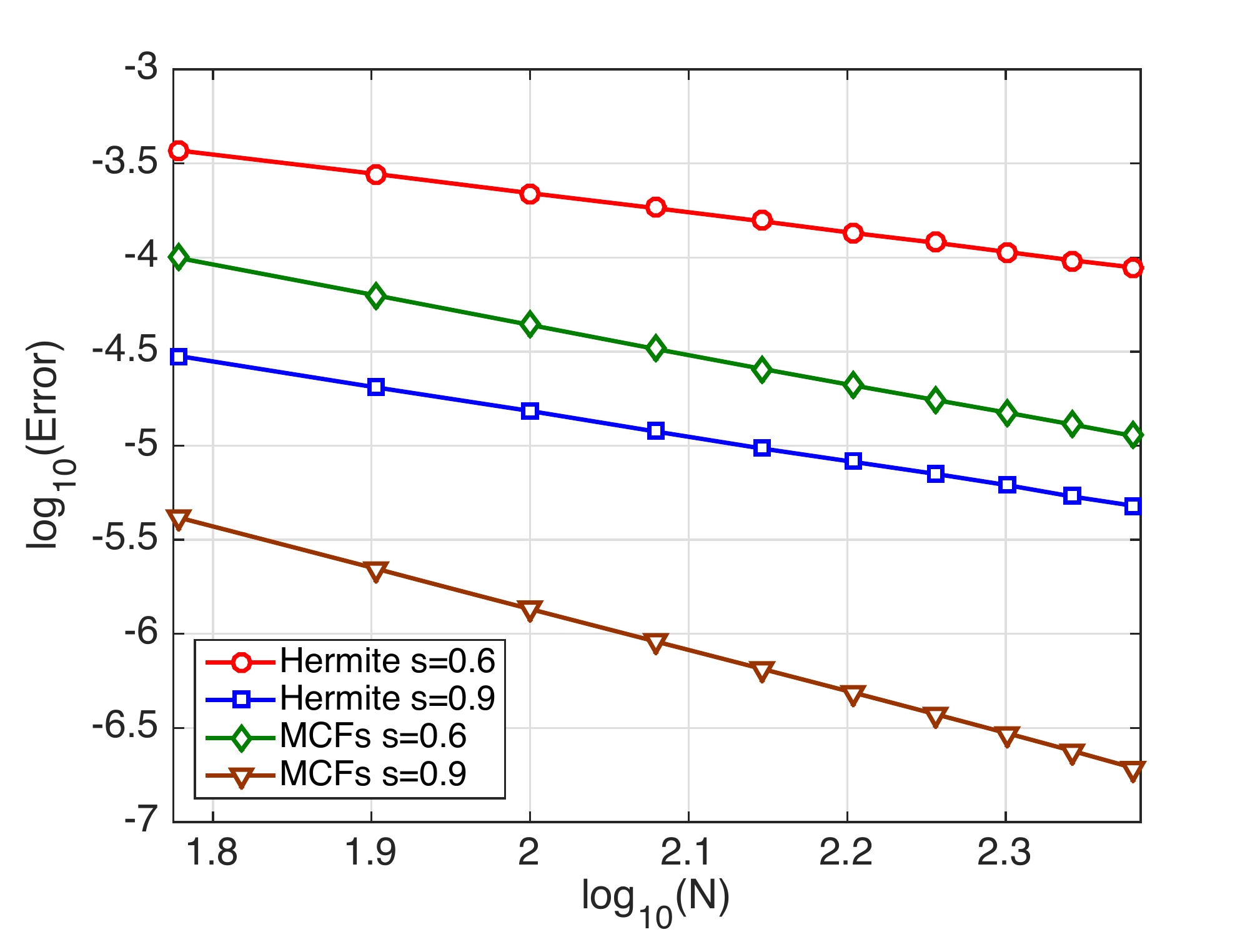}}
\end{minipage}}
\subfigure[$d=1$ and $f(x)=(1+x^2)^{-2}$]{
\begin{minipage}[t]{0.42\textwidth}
\centering
\rotatebox[origin=cc]{-0}{\includegraphics[width=0.9\textwidth]{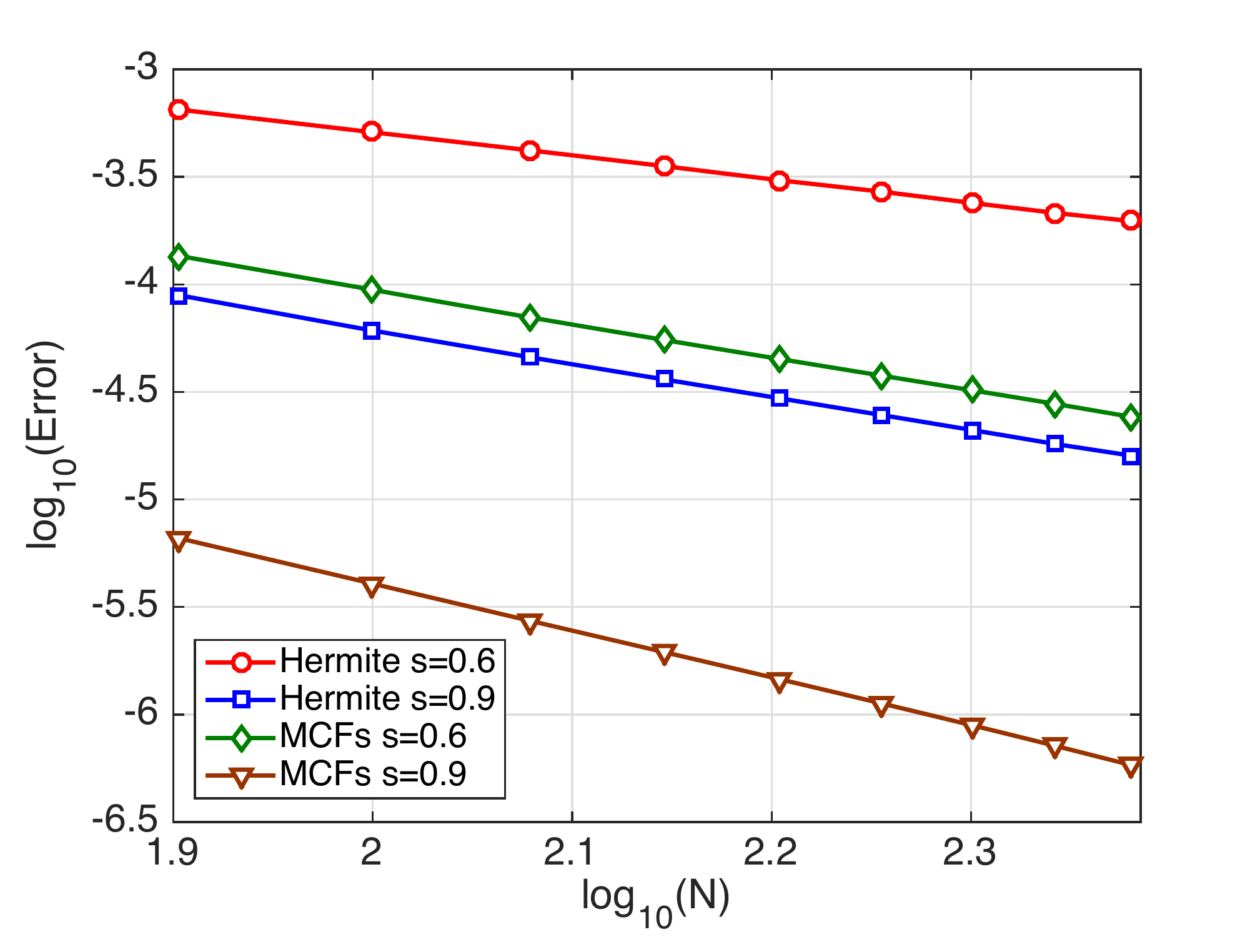}}
\end{minipage}}\vskip -5pt

\subfigure[$f(x_1,x_2)=(1+x_1)(1+x_2)e^{-\frac{|x|^2}{2}}$]{
\begin{minipage}[t]{0.42\textwidth}
\centering
\rotatebox[origin=cc]{-0}{\includegraphics[width=0.9\textwidth]{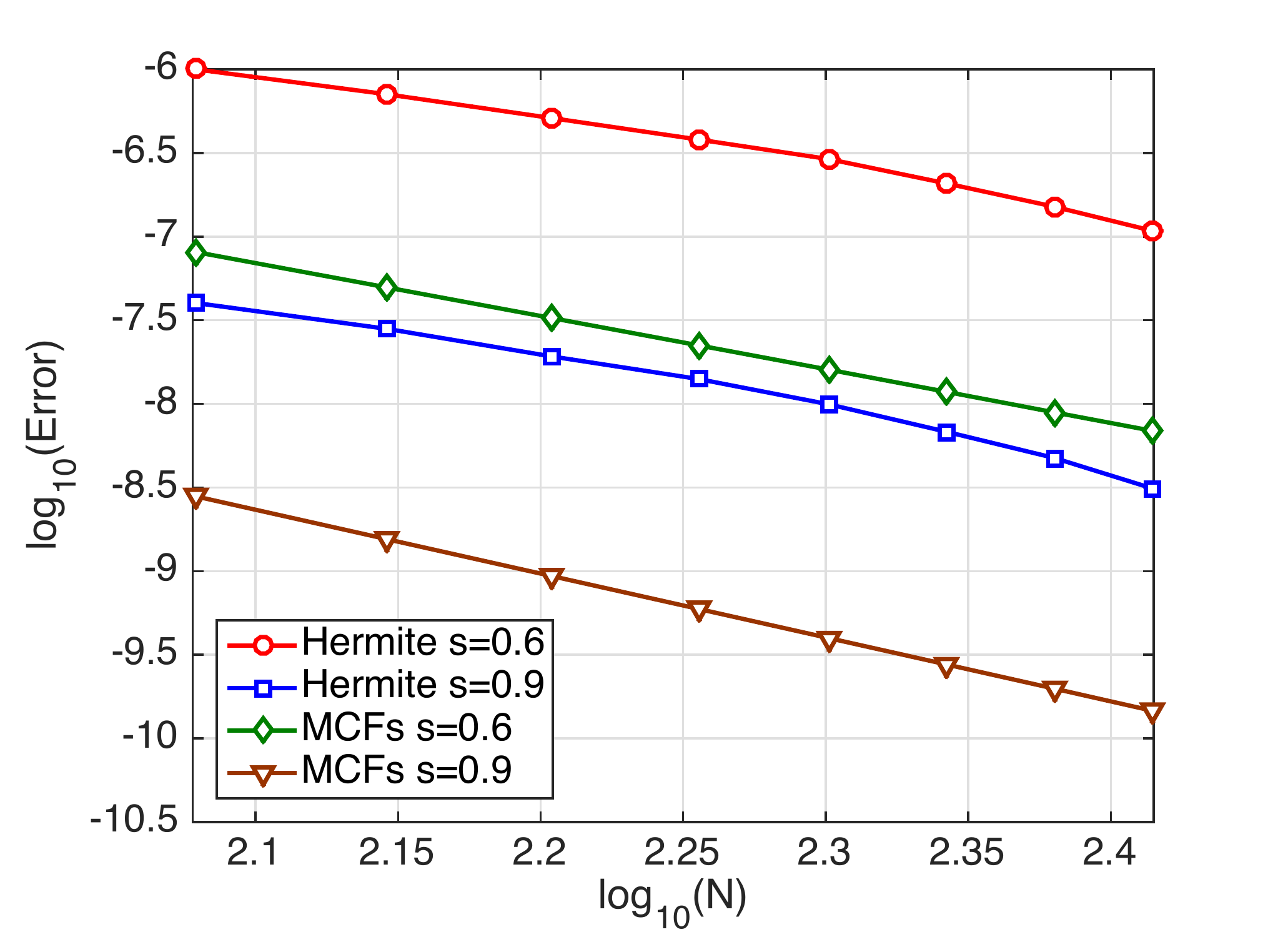}}
\end{minipage}}
\subfigure[$f(x_1,x_2)=(1+2x_1^2+5x_2^2)^{-2}$]{
\begin{minipage}[t]{0.42\textwidth}
\centering
\rotatebox[origin=cc]{-0}{\includegraphics[width=0.9\textwidth]{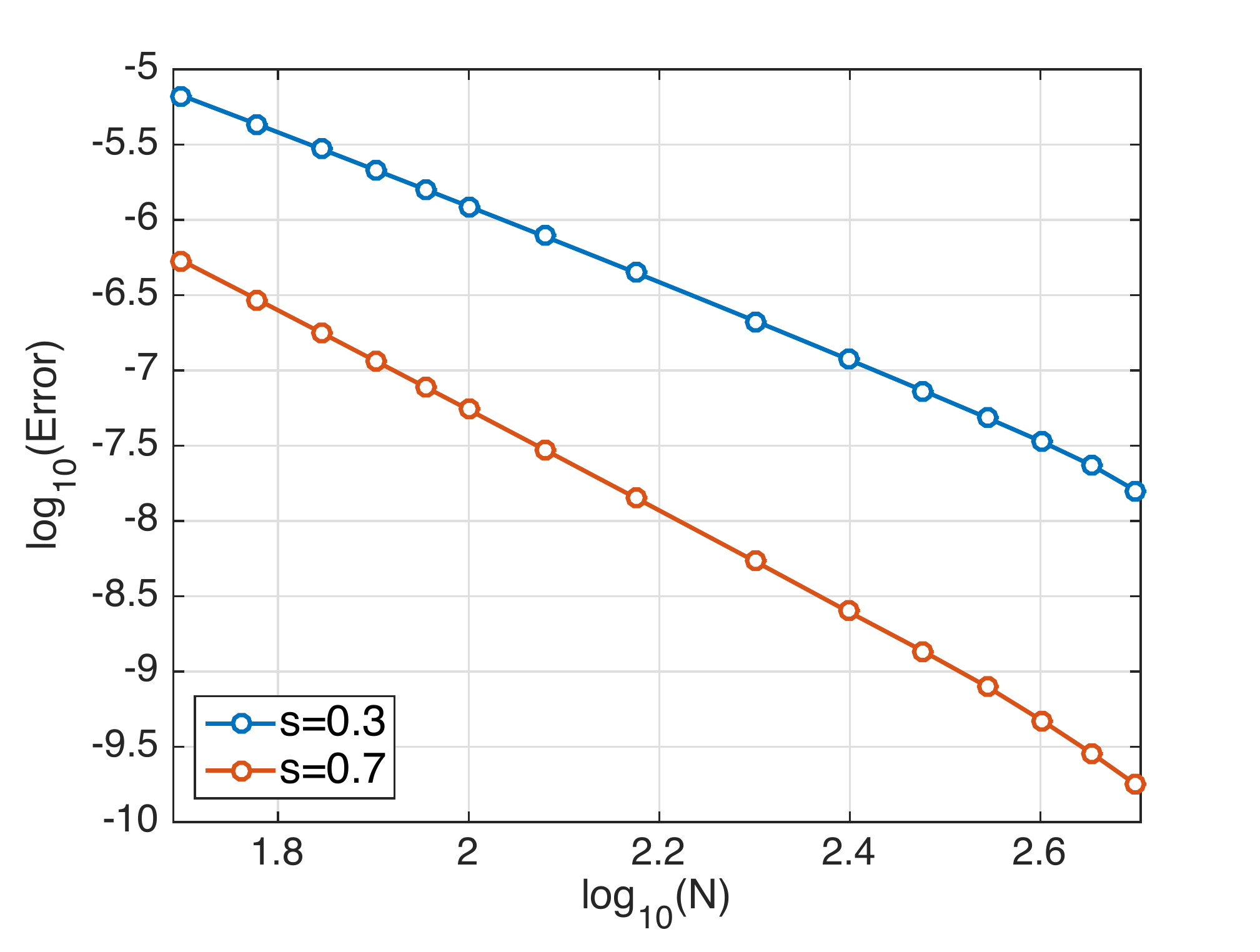}}
\end{minipage}}\vskip -5pt

\subfigure[$f(x_1,x_2,x_3)=(1+x_1+2x_2^2+3x_3^2)e^{-\frac{|x|^2}{2}}$]{
\begin{minipage}[t]{0.42\textwidth}
\centering
\rotatebox[origin=cc]{-0}{\includegraphics[width=0.9\textwidth]{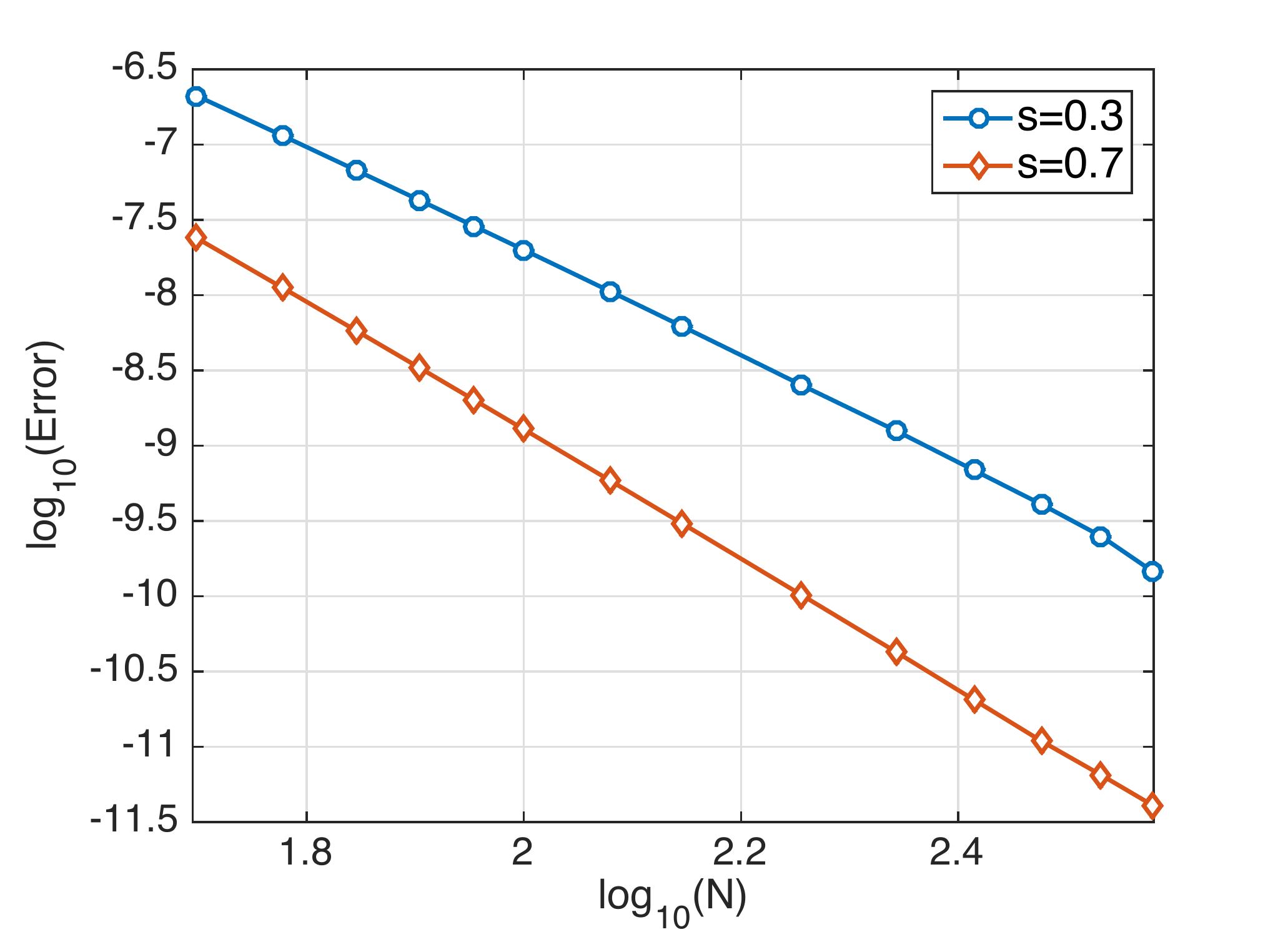}}
\end{minipage}}
\subfigure[$f(x_1,x_2,x_3)=(1+x_1^2+2x_2^2+3x_3^2)^{-2.7}$]{
\begin{minipage}[t]{0.42\textwidth}
\centering
\rotatebox[origin=cc]{-0}{\includegraphics[width=0.9\textwidth]{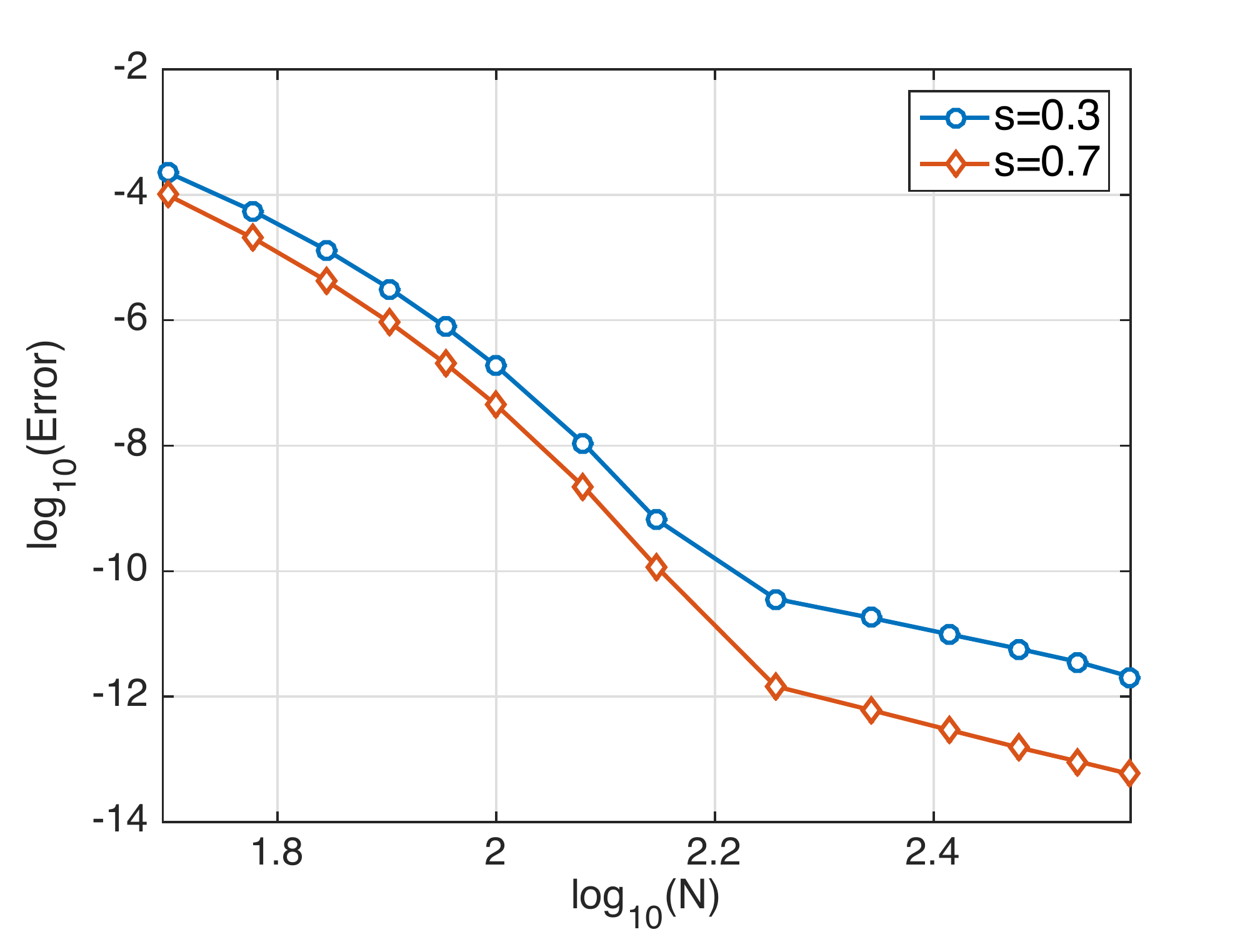}}
\end{minipage}}
\caption
{\small  (a)-(c): A comparison of $L^2$-errors between our method and Hermite-Galerkin method in \cite{mao2017hermite} for different source function $f(x)$;
   (d)-(f): The maximum errors for different source function $f(x)$ with $d=2,3$. In the tests, we take and $\gamma=1, \nu=2.5$.}\label{Example3}
\end{figure}



%
\vspace*{-10pt}
\begin{table}[h!tbp]
\centering
\caption{\small A comparison of $L^2$-error for $f(x)=(1+x)e^{-\frac{x^2}{2}}$.}\small
\begin{tabular}{||c||c|c|c|c||c|c|c|c||}\hline
&\multicolumn{4}{c||}{$s=0.6$}  & \multicolumn{4}{c||}{$s=0.9$} \\
\cline{2-5}\cline{6-9}
$N$ &Hermite  \cite{mao2017hermite} & Order & MCF & Order & Hermite \cite{mao2017hermite}  &Order  &MCF  & Order\\
\hline
$80$&2.77e-04 &         &6.29e-05&        &2.06e-05&        &2.21e-06 &\\
$100$&2.21e-04&1.01&4.38e-05&1.61&1.53e-05&1.32&1.35e-06 &2.20\\
$120$&1.84e-04&1.02&3.26e-05&1.61&1.20e-05&1.33&9.10e-07&2.19\\
$140$&1.57e-04&1.03&2.56e-05&1.57&9.80e-06&1.33&6.51e-07&2.18\\
$160$&1.36e-04&1.03&2.08e-05&1.52&8.20e-06&1.34&4.86e-07&2.18\\
$180$&1.21e-04&1.04&1.75e-05&1.49&7.00e-06&1.34&3.75e-07&2.20\\
$200$&1.08e-04&1.04&1.49e-05&1.49&6.08e-06&1.34&2.96e-07&2.24\\
$220$&9.79e-05&1.05&1.29e-05&1.52&5.35e-06&1.34&2.38e-07&2.29\\
$240$&8.93e-05&1.06&1.12e-05&1.58&4.76e-06&1.35&1.94e-07&2.35\\
\hline
\end{tabular}\label{tab1}\end{table}

\begin{example}\label{3dmulti} {\bf (Multi-term fractional equations).}
Consider the three-dimensional  multi-term fractional Laplacian equation
\begin{equation}
\label{multieq}
\sum_{j=1}^{J} \rho_j(-\Delta)^{s_j}u(x) = f(x),\;\;\;  {\rm in}\;\;  \mathbb{R}^3; \quad
  u(x)=0 \;\;\;  {\rm as}\;\;  |x|\to \infty.
\end{equation}
In Figure \ref{fig3dmulti} (a), we plot in log-log scale  the maximum errors of \eqref{multieq} against various $N$, where we take  $u(x)=(1+|x|^2)^{-\frac{3\pi}{4}}$, $J=4$ and
\begin{equation}\label{parasr}\begin{split}
&s_1=0.77,\;\;s_2=0.33,\;\;s_3=0.21,\;\;s_4=0, \;\; \rho_1=1,\;\;\rho_2=2,\;\;\rho_3=\sqrt{2},\;\;\rho_4=1.
\end{split}\end{equation}
In Figure \ref{fig3dmulti} (b), we plot in log-log scale the maximum errors of \eqref{multieq} against various $N$, where we take  $f(x)=(1+x_1+2x_2^2+3x_3^2)e^{-\frac{|x|^2}{2}}$, $J=4$ and
\begin{equation}\label{parasr2}\begin{split}
&s_1=0.76,\;\;s_2=0.41,\;\;s_3=0.23,\;\;s_4=0,\;\; \rho_1=2,\;\;\rho_2=1,\;\;\rho_3=0.5,\;\;\rho_4=1.
\end{split}\end{equation}
We observe the algebraic decay of the errors, and the method is as accurate and efficient as  the previous cases.
\end{example}
\vspace*{-18pt}
  \begin{figure}[!h]
\subfigure[$d=3$ with given exact solution]{
\begin{minipage}[t]{0.45\textwidth}
\centering
\rotatebox[origin=cc]{-0}{\includegraphics[width=0.9\textwidth]{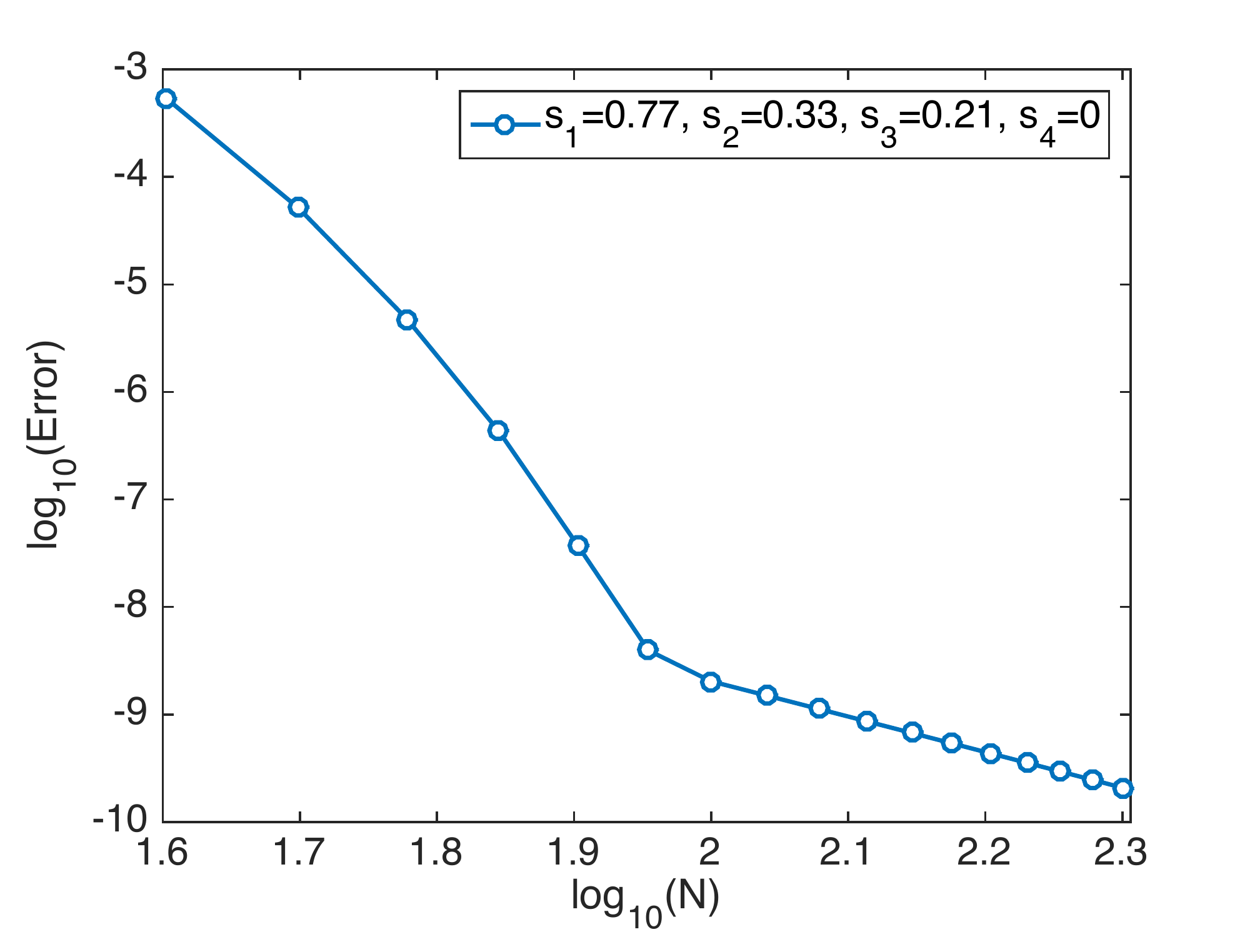}}
\end{minipage}}
\subfigure[$d=3$ with given source term]{
\begin{minipage}[t]{0.45\textwidth}
\centering
\rotatebox[origin=cc]{-0}{\includegraphics[width=0.9\textwidth]{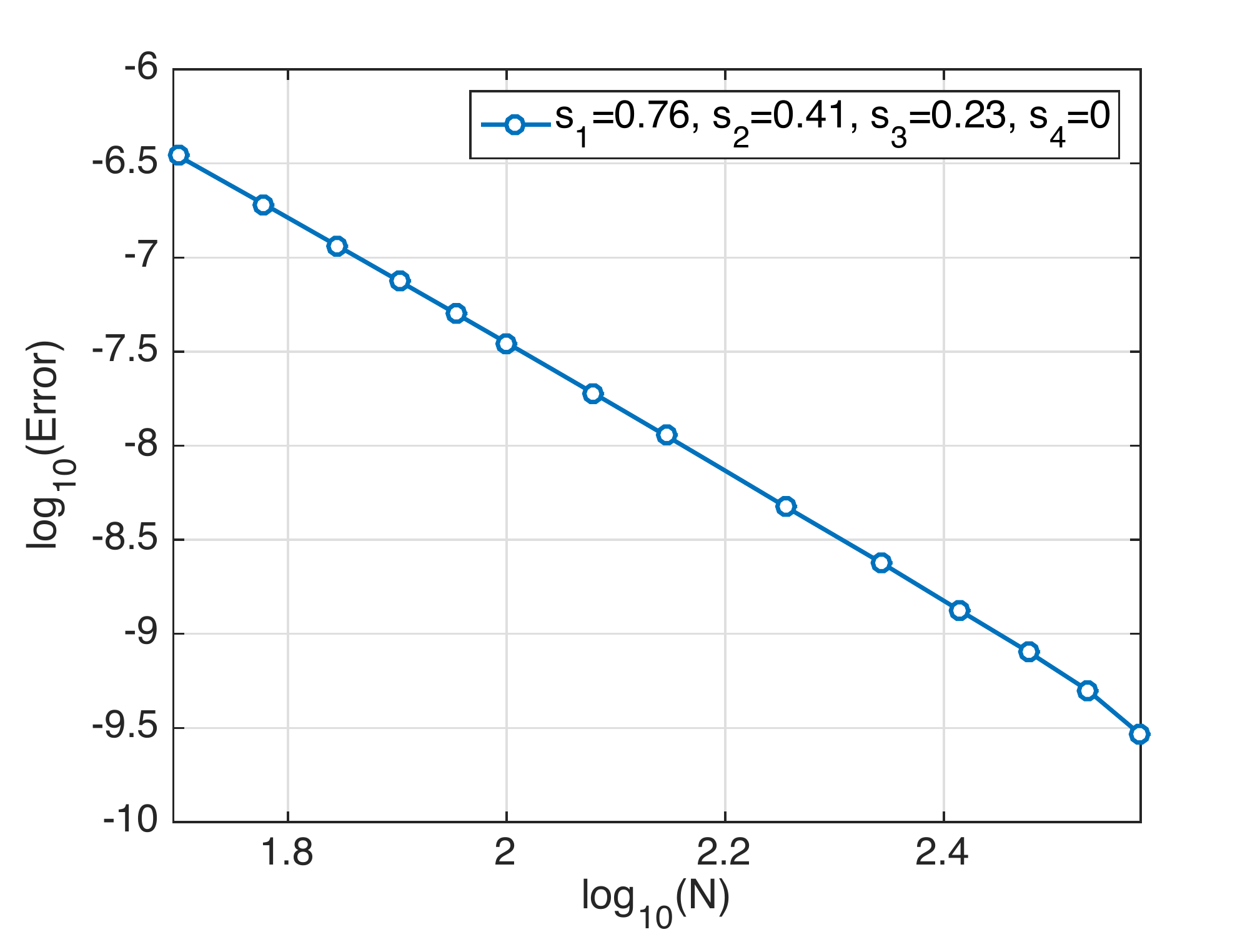}}
\end{minipage}}\vskip -10pt
\caption
{\small (a). The maximum error for \eqref{multieq} with $u(x)=(1+|x|^2)^{-\frac{3\pi}{4}}$ and $\nu=2.5$; (b). The maximum error for \eqref{multieq} with $f(x)=(1+x_1+2x_2^2+3x_3^2)e^{-\frac{|x|^2}{2}}$ and $\nu=2.5$.}\label{fig3dmulti}
\end{figure}

\section{MCF approximation of  nonlinear fractional Schr{\"o}dinger equations}\label{section4}
In this section, we apply the  fast algorithm to some nonlinear PDEs involving fractional Laplacian.
As an example,  we consider  nonlinear fractional Schr{\"o}dinger equation (fNLS) (cf. \cite{klein2014numerical}):
\begin{equation}
\label{fNLS}\begin{split}
&{\rm i}\psi_t=\frac{1}{2}(-\Delta)^s\psi+\gamma|\psi|^{2p}\psi,\quad x\in \R^d,\quad t\in(0,T],\\&
\psi(x,0)=\psi_0(x),\quad x\in
 \R^d, \quad |\psi|\to 0,\quad |x|\to \infty, \end{split}\end{equation}
 where ${\rm i}^2=-1$, $\psi(x,t)$ is a complex-valued wave function, the parameters $\gamma$ and $p$ are real constants, and $\psi_0$ is  given. It is noteworthy  that the mass is conserved  (cf. \cite{bao2005fourth,klein2014numerical}):
\begin{equation}
\label{conserve}
M(t)=\dint_{\mathbb{R}^d}|\psi(x,t)|^2{\rm d}x=M(0),\quad t>0.
\end{equation}



\subsection{The scheme}  We adopt the time-splitting technique, 
and  start with rewriting the fNLS \eqref{fNLS} as follows
\begin{equation}
\label{splifNLS}
\ri \psi_t=A\psi+B\psi,
\end{equation}
where
$$A\psi=\gamma|\psi(x,t)|^2\psi(x,t),\quad B\psi=\dfrac12(-\Delta)^s\psi(x,t).$$
The notion of time-splitting is to  solve the following
two subproblems:
\begin{equation}
\label{spliA}
\ri\dfrac{\partial\psi(x,t)}{\partial t}=A\psi(x,t)=\gamma|\psi(x,t)|^2\psi(x,t),\quad x\in\R^d,
\end{equation}
and
\begin{equation}
\label{spliB}
\ri\dfrac{\partial\psi(x,t)}{\partial t}=B\psi(x,t)=\dfrac12(-\Delta)^s\psi(x,t),\quad x\in\R^d.
\end{equation}
The essence of the splitting method is to solve the two sub-problems iteratively  at each time step.

(i). We first consider the subproblem \eqref{spliA}.
Multiplying \eqref{spliA} by ${\bar \psi(x,t)}$, we find from  the resulted equation that $|\psi(x,t)|$ invariant in $t$ (see e.g., \cite{bao2005fourth}). More precisely, for $t\geq t_s$ ($t_s$ is any given time),
\eqref{spliA} becomes
\begin{equation}
\label{spliA1}
\ri \dfrac{\partial\psi(x,t)}{\partial t}=\gamma|\psi(x,t_s)|^2\psi(x,t),\quad t\geq t_s,\quad x\in\R^d,
\end{equation}
which can be integrated exactly, i.e.,
\begin{equation}
\label{spliA2}
\psi(x,t)=e^{-\ri\gamma|\psi(x,t_s)|^2(t-t_s)}\psi(x,t_s),\quad t\geq t_s,\quad x\in\R^d.
\end{equation}

(ii). We now turn to the subproblem \eqref{spliB}.  Remarkably, the Fourier-like basis can diagonalise  the operator $B$ so that $e^{-\ri B\Delta t}\psi$ can be efficiently evaluated (which is crucial  for the final scheme to be time reversible and time transverse invariant). More precisely,  we seek $\psi_N(x,t)\in \mathbb{V}_{\!N}^d$ as an approximate solution to  \eqref{spliB}, such that
\begin{equation}
\label{splifNLS2}
 {\rm i}\big(\partial_t\psi_N,v\big)_{L^2(\mathbb{R}^d)}=\big(B\psi_N,v\big)_{L^2(\mathbb{R}^d)}= \dfrac12 \big( (-\Delta)^s\psi_N,v\big)_{L^2(\mathbb{R}^d)},\quad \forall\, v\in \mathbb{V}^d_{\!N}.
\end{equation}
Using the Fourier-like MCF basis, we write
\begin{equation}
\label{solufNLS}
\psi_N(x,t)=\dsum_{k\in \Upsilon_{\!N}}\hat{\psi}_k(t)\widehat{\mathbb{T}}_k(x),\quad x\in \R^d.
\end{equation}
Substituting it into \eqref{splifNLS2}, and taking the inner product with $\widehat{\mathbb{T}}_m(x)$,  we deduce from  \eqref{newtheorem}  that
\begin{equation}
\label{splitODE2}
\ri\dfrac{\partial\hat{\psi}_m(t)}{\partial t}=\dfrac12|\lambda_m|_1^s\hat{\psi}_{m}(t),\quad m\in \Upsilon_{\!N}.
\end{equation}
Then, we derive from \eqref{splitODE2} that the solution for \eqref{splifNLS2}, i.e., the numerical solution of \eqref{spliB}, is given by
\begin{equation}
\label{coefmatr}
\begin{split}
\psi_N(x,t)=e^{-\ri B(t-t_s)}\psi_N(x,t_s)=\dsum_{k\in \Upsilon_{\!N}}e^{-\frac{\ri}{2}|\lambda_k|_1^s(t-t_s)}
\hat{\psi}_{k}(t_s)\widehat{\mathbb{T}}_k(x),\quad t\geq t_s.
\end{split}\end{equation}

With the exact solution \eqref{spliA2} and the approximate solution \eqref{coefmatr} for two subproblems \eqref{spliA} and \eqref{spliB}, respectively,
we now describe the implementation of the fourth-order time splitting method (TS4) for solving \eqref{fNLS}.  Let $\{x_p\}_{p\in \Upsilon_{\!N}}$ be tensorial grids as in \eqref{inter2}, and $t_n=n\Delta t$ be the time-stepping grids.  Let $\psi^n_{p}$ be the approximation of $\psi(x_p,t_n),$ and denote by $\bs \psi^n$  the solution vector with components $\{\psi^n_{p}\}_{p\in \Upsilon_{\!N}}$.
For notational convenience, we define the solution map related to
\eqref{coefmatr}:
\begin{equation}
\label{defiF}\begin{split}
&\mathcal{T}_{N}[\omega; \bs \varPsi_p](x)=\dsum_{k\in \Upsilon_{\!N}}e^{-\ri\omega|\lambda_k|^s\Delta t}\, \hat{\varPsi}_k\,  \widehat{\mathbb{T}}_k(x),
\end{split}\end{equation}
where $\{\hat \varPsi_k\}$ are the MCF expansion coefficients computed from the sampling of
 $\varPsi\in {\mathbb V}_{\!N}^d$ on the grids $\{x_p\},$ and $\omega>0$ is some weight.

Following \cite{bao2005fourth}, we carry out  the fourth-order time-splitting method for the fNLS \eqref{fNLS}, from time $t=t_n$ to  $t=t_{n+1}$, as follows
\begin{equation}\label{schespli4}
\begin{cases}
 \psi^{(1)}_{p}=e^{-2\ri \omega_1\gamma \Delta t |\psi^n_{p}|^2}\psi^n_{p},\quad
&\psi^{(2)}_{p}=\mathcal{T}_{N}[\omega_2; \bs \psi^{(1)}_p](x_p),\\[4pt]
\psi^{(3)}_{p}=e^{-2\ri\omega_3 \gamma \Delta t~|\psi^{(2)}_{p}|^2}\psi^{(2)}_{p},\quad
&\psi^{(4)}_{p}=\mathcal{T}_{N}[\omega_4; \bs \psi^{(3)}_p](x_p),\\[4pt]
\psi^{(5)}_{p}=e^{-2\ri\omega_3 \gamma \Delta t |\psi^{(4)}_{p}|^2}\psi^{(4)}_{p},\quad
&\psi^{(6)}_{p}=\mathcal{T}_{N}[\omega_2,\bs \psi^{(5)}_p](x_p),\\
\psi^{n+1}_{p}=e^{-2\ri \omega_1\gamma\Delta t~|\phi^{(6)}_{p}|^2}\psi^{(6)}_{p},\quad
&\forall \, p\in \Upsilon_{\!N},
\end{cases}
\end{equation}
where the weights are given by
\begin{equation}
\label{schespli41}
\begin{array}{ll}
\omega_1=0.33780~17979~89914~40851,\quad&\omega_2=0.67560~35959~79828~81702,\\[3pt]
\omega_3=-0.08780~17979~89914~40851,\quad&\omega_4=-0.85120~71979~59657~63405.
\end{array}\end{equation}

To show the stability of fourth-order splitting method, we further define
\begin{equation}
\label{e1}
\|\psi^n\|^2_{N}=\sum_{j\in\Upsilon_{\!N}}|\psi_{j}^n|^2\omega_j:=\sum_{j_1=0}^{N_1}\cdots\sum_{j_d=0}^{N_d}\psi(x_{j_1},\cdots,x_{j_d})\omega_{j_1}\cdots\omega_{j_d},
\end{equation}
where $\psi^n_j=\psi^n(x_j)$,
and  $\{x_j,\omega_j\}_{j\in\Upsilon_{\!N}}$ are  the corresponding tensorial nodes and weights as in \eqref{mapnodes}. Following \cite[Lemma 3.1]{bao2005fourth}, we can show the property stated below.
\begin{thm}
\label{lmstab}
The $TS4$ has the normalisation conservation, i.e.,
\begin{equation}
\label{defiFs}
\|\psi^n\|^2_{N}=\sum_{j\in\Upsilon_{\!N}}|\psi_{j}^n|^2\omega_j=\sum_{j\in\Upsilon_{\!N}}|\psi_0(x_j)|^2\omega_j=\|\psi_0\|^2_{N},\quad n\geq0.
\end{equation}
\end{thm}
\subsection{Numerical results.}
In the computation, we take $d=2$, and the initial condition to be
\begin{equation}
\label{exam1}\begin{split}
\psi_0(x_1,x_2)={\rm sech}(x_1){\rm sech}(x_2)\exp(\ri (x_1+x_2)),\quad  (x_1,x_2)\in \R^2.
\end{split}\end{equation}
In order to test the fourth-order accuracy in time of the TS4 method,
we compute a numerical solution with focusing case $\gamma=-1$, $s=0.7$, a very fine mesh, e.g., $N=300$, and a very small time step  $\Delta t=0.0001$, as the ``exact'' solution $\psi$. Let $\psi^{\Delta t}$ be the numerical solution with $N=300$ and time step side $\Delta t$. 
Table \ref{tab2} lists the maximum error and  $L^2$-error at $T=2$ for different time step size $\Delta t$. The results in Table \ref{tab2} demonstrate the fourth-order accuracy in time of the TS4 method \eqref{schespli4}. 

In Figure \ref{fig5}, we plot the maximum errors and $L^2$-error versus space discretization $N$ and time discretization $\Delta t$.
They indicate that the numerical errors decay algebraically as $N$ increases/or $\Delta t$ decreases.

\begin{table}[h!tbp]
\begin{center}
\caption{{\small Time discretization errors for the TS4 method \eqref{schespli4} at $T=2$ with $N=300$.}}\small
\centering{
\begin{tabular}{c c c c c c c c}\hline
$\Delta t$&1/10&1/20&1/40&1/80&1/160&1/320 \\
\hline \hline
$\max$-error    &1.059e-02&1.092e-03&8.747e-05&5.782e-06&3.641e-07&2.301e-08\\
   order             &      $-$       & 3.2             &3.6             &3.9             &4.0             &3.9\\
   \hline
$L^2$-error &2.557e-03& 2.235e-04&1.616e-05&1.084e-06&6.553e-08&4.435e-09\\
order           &      $-$       & 3.5             &3.7             &3.9             &4.0             &3.9\\
\hline
\end{tabular}}\label{tab2}
\end{center}
\end{table}

\begin{figure}[!h]
\subfigure[Errors vs. $N$]{
\begin{minipage}[t]{0.45\textwidth}
\centering
\rotatebox[origin=cc]{-0}{\includegraphics[width=0.9\textwidth]{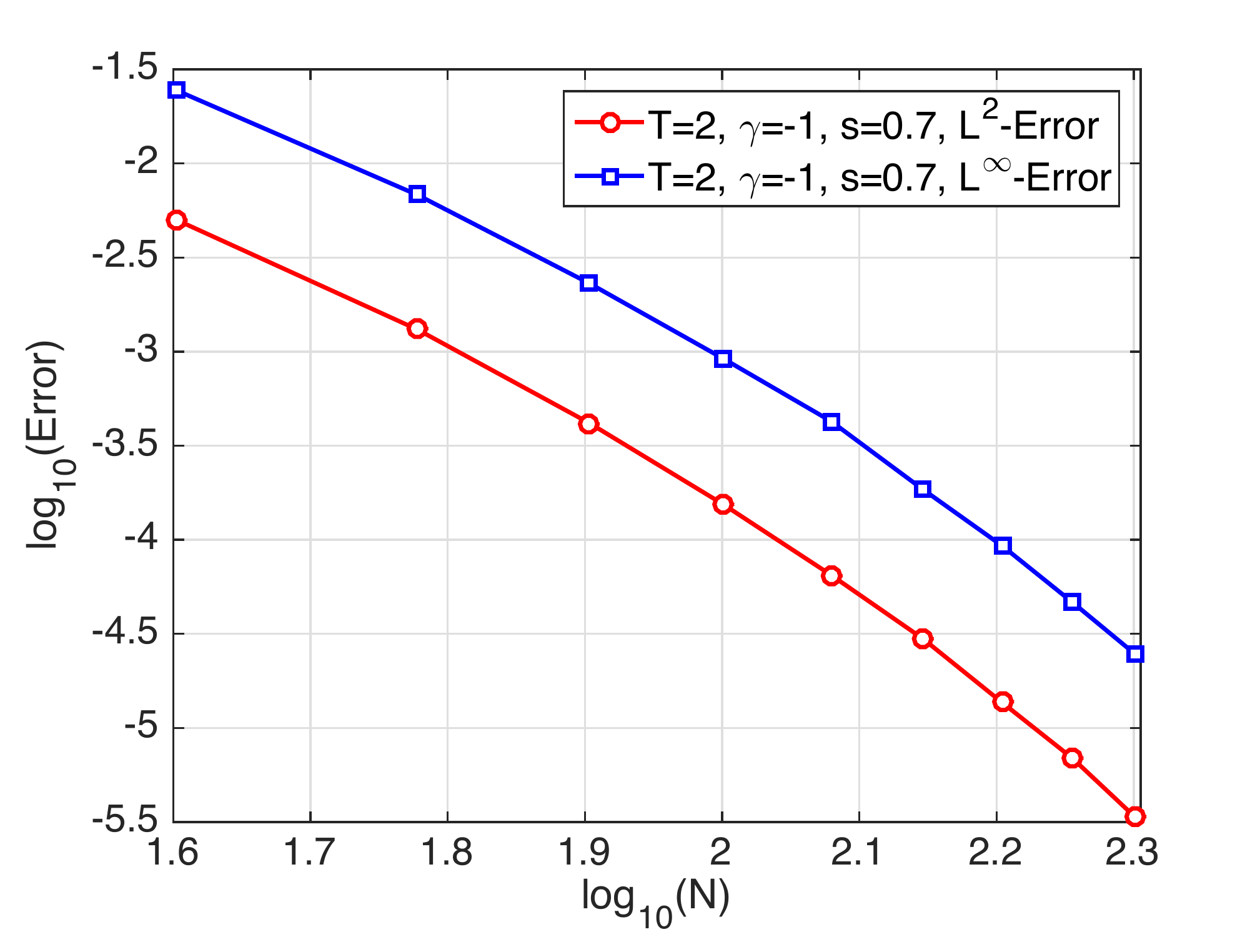}}
\end{minipage}}
\subfigure[Errors vs. $\Delta t$]{
\begin{minipage}[t]{0.45\textwidth}
\centering
\rotatebox[origin=cc]{-0}{\includegraphics[width=0.9\textwidth]{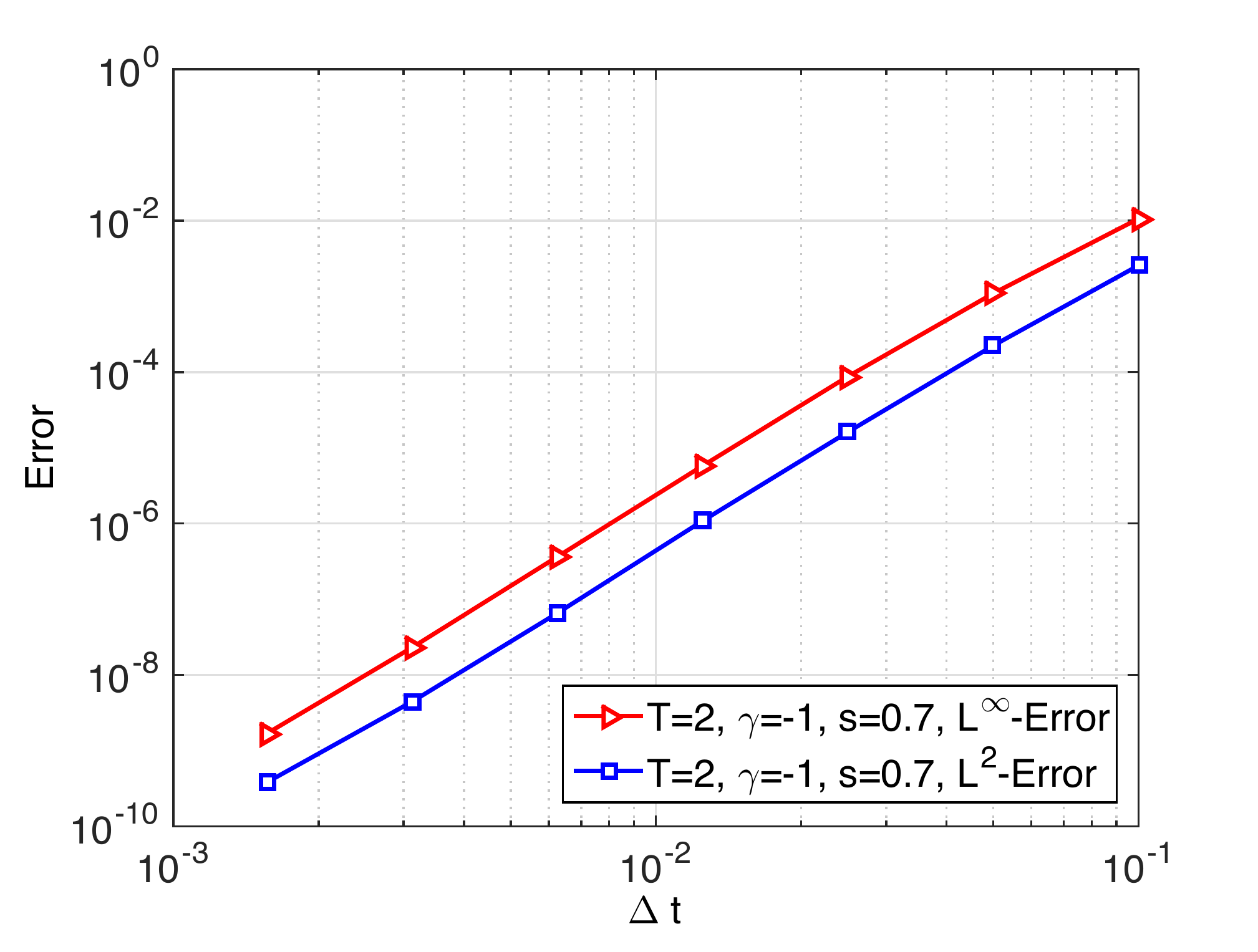}}
\end{minipage}}\vskip -10pt
 \caption
 {\small (a). The numerical error of \eqref{exam1} with  $s=0.7,~\gamma=-1,~T=2$; (b). The numerical error of \eqref{exam1} with  $s=0.7,~\gamma=-1,~T=2$.}\label{fig5}
\end{figure}


In Figure \ref{ExamplefNLS1}  (a)-(d), we depict the modulus squared of the numerical solution with defocusing case ($\gamma=1$) obtained by TS4. Here, we take $N=200$, $T=1,2$, and different values of fractional order $s=0.3,0.7$. We observe that the solution diffused as expected.
On the other hand, the blow-up of the solution might happen for focusing case $\gamma=-1$ (cf. \cite{klein2014numerical}). In Figure \ref{ExamplefNLS1}  (e)-(f),  we plot the profiles of the modulus square  of the  numerical solution at $T=1$ with $N=200$ and $s=0.3, 0.7.$ We can observe the expected blow-up phenomenon.

 \begin{figure}[!th]
\subfigure[$T=1$, $s=0.3$ and $\gamma=1$]{
\begin{minipage}[t]{0.42\textwidth}
\centering
\rotatebox[origin=cc]{-0}{\includegraphics[width=0.9\textwidth]{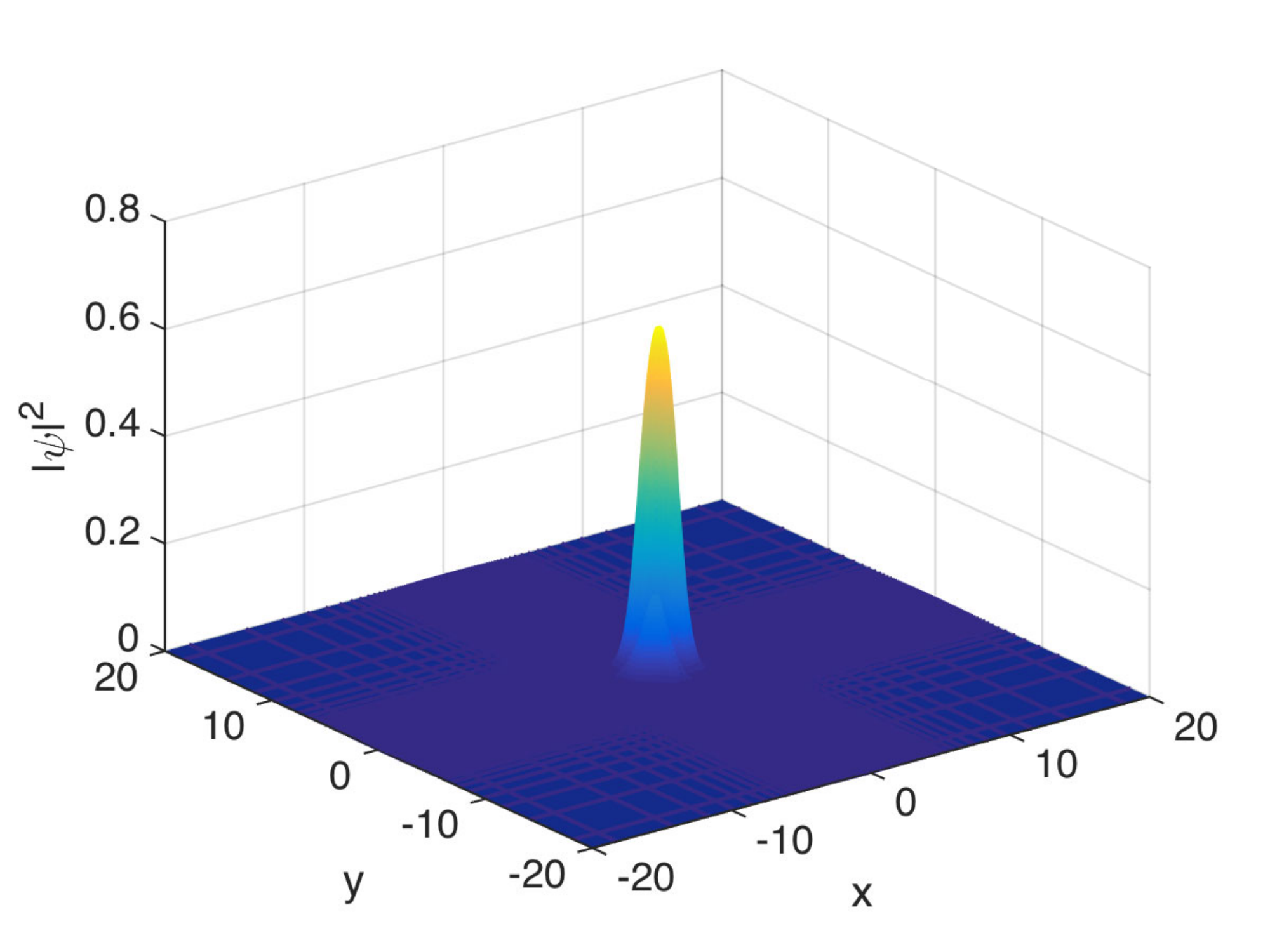}}
\end{minipage}}
\subfigure[$T=2$, $s=0.3$ and $\gamma=1$]{
\begin{minipage}[t]{0.42\textwidth}
\centering
\rotatebox[origin=cc]{-0}{\includegraphics[width=0.9\textwidth]{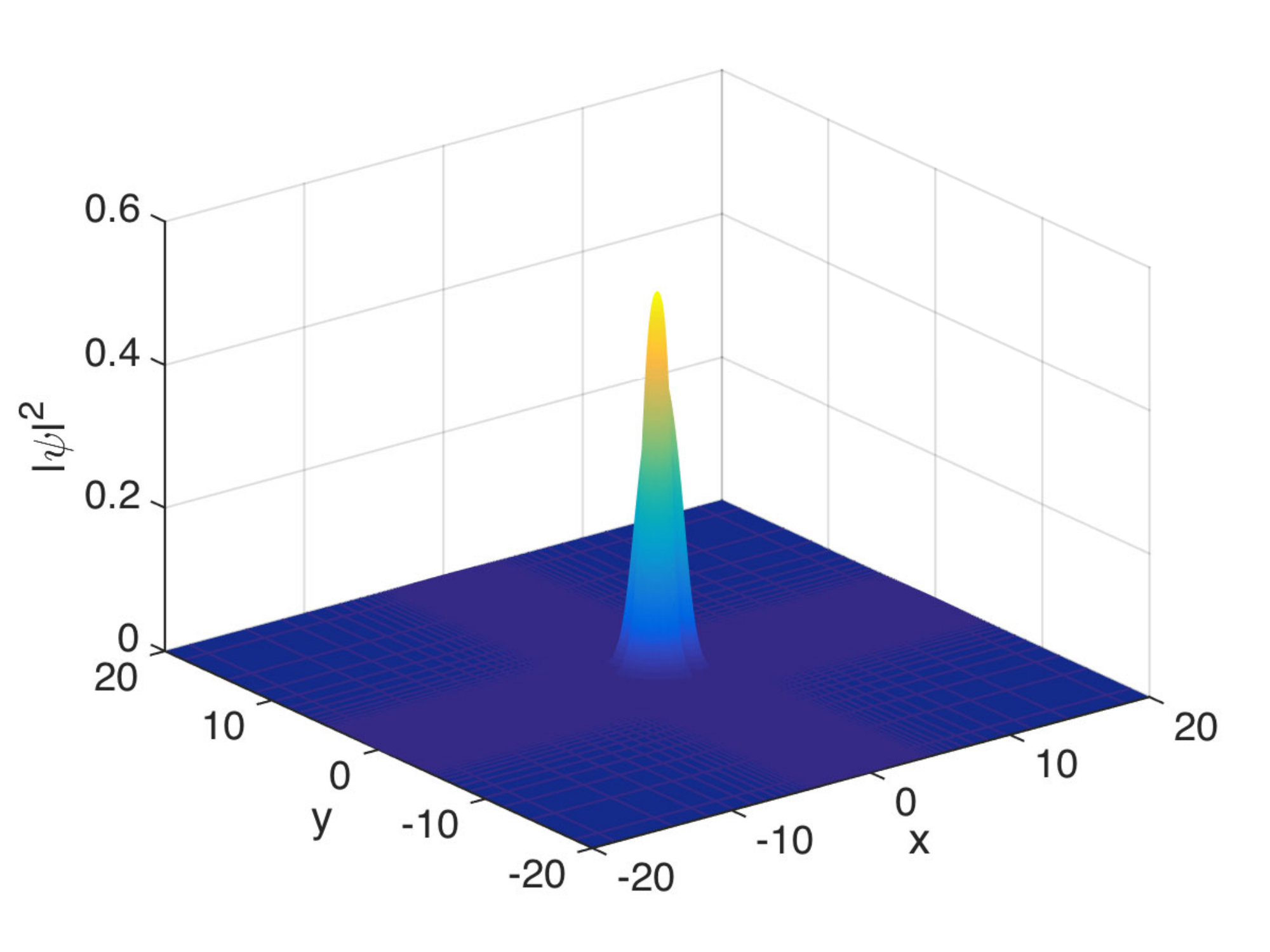}}
\end{minipage}}
%

\subfigure[$T=1$, $s=0.3$ and $\gamma=-1$]{
\begin{minipage}[t]{0.42\textwidth}
\centering
\rotatebox[origin=cc]{-0}{\includegraphics[width=0.9\textwidth]{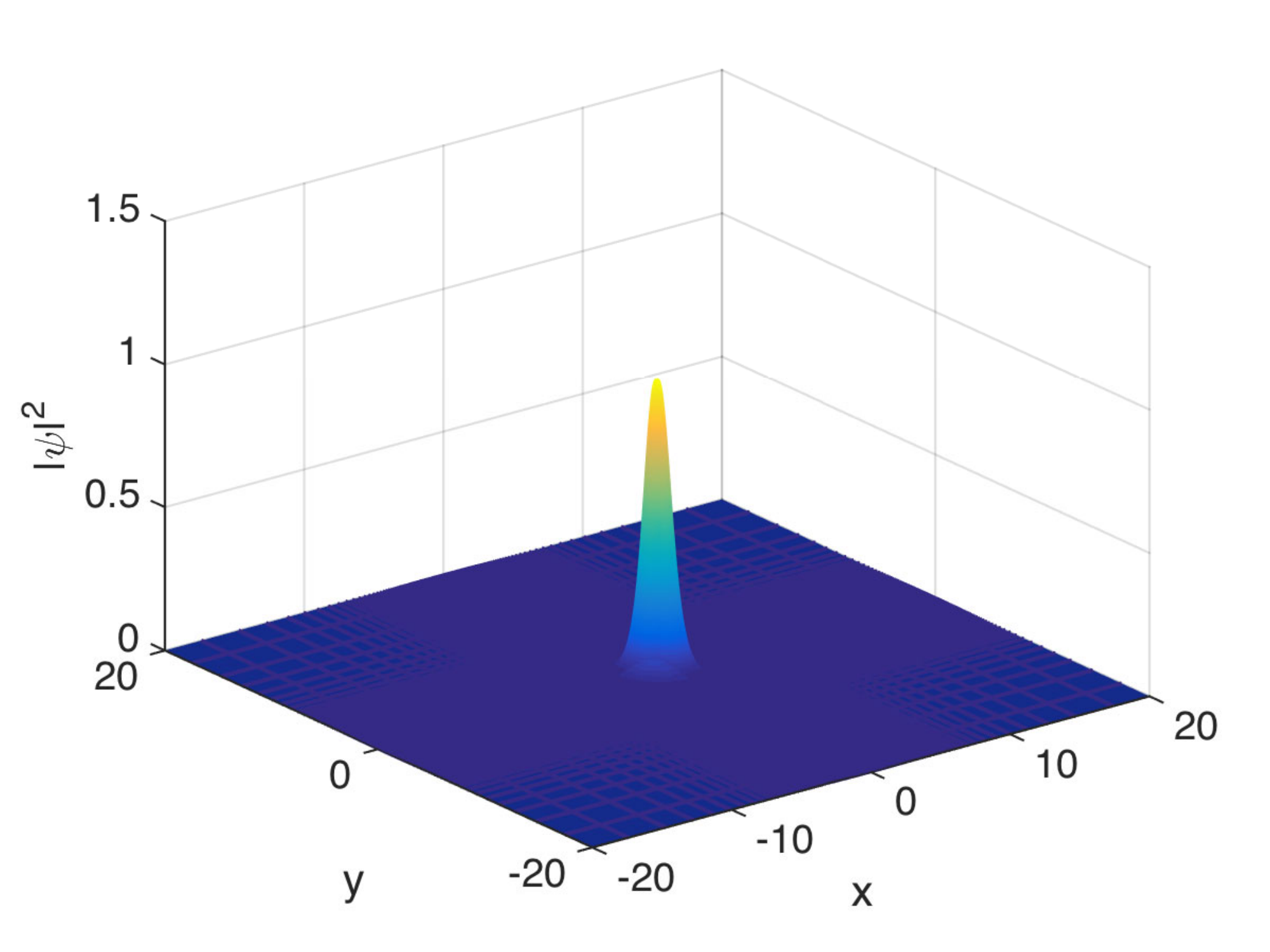}}
\end{minipage}}
\subfigure[$T=1$, $s=0.7$ and $\gamma=-1$]{
\begin{minipage}[t]{0.42\textwidth}
\centering
\rotatebox[origin=cc]{-0}{\includegraphics[width=0.9\textwidth]{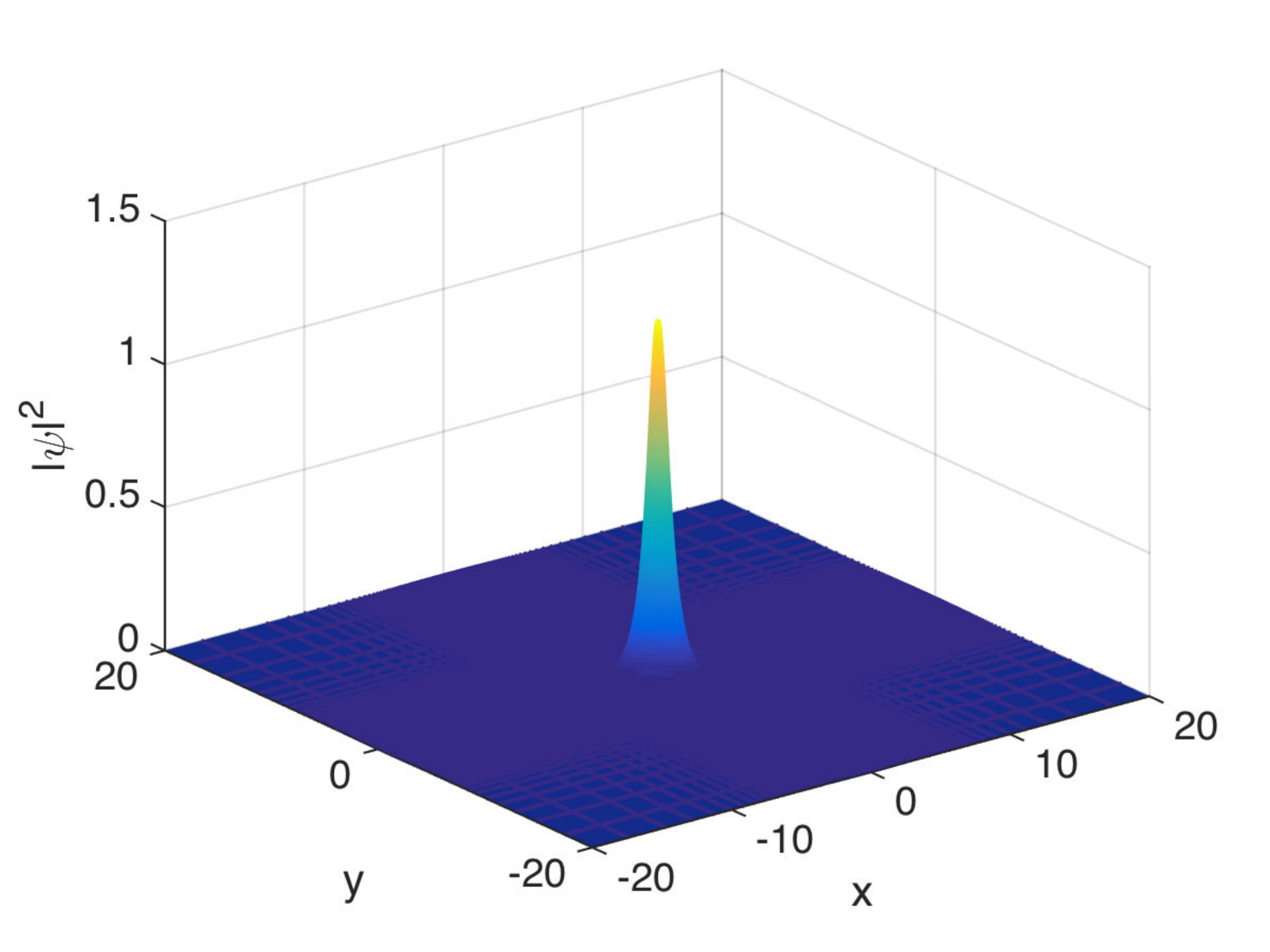}}
\end{minipage}}

\caption
{\small Profiles of the modulus square  of the numerical solutions at different time and with different  fractional orders.}\label{ExamplefNLS1}
\end{figure}

\subsection{Concluding Remarks}
We developed a fast MCF-spectral-Galerkin method for PDEs involving integral fractional Laplacian
in $\mathbb R^d$.  The fast solver is integrated with two critical components: (i) the Dunford-Taylor formulation for the fractional Laplacian; and (ii) Fourier-like bi-orthogonal MCFs as basis functions.
The fast spectral algorithm could achieve a quasi-optimal computational cost.
Different from the existing works on bounded domains (cf. \cite{bonito2019numerical,bonito2019sinc}), the integration  in $t$ is evaluated explicitly,  and the fractional Laplacian can be fully diagonalised under (i) and (ii).  Indeed, the existing approaches for fractional Laplacian in unbounded domains are either too complicated or computational prohibitive even for $d=2.$ However, the fast solver works for any dimension,  and can be easily incorporated with e.g., the hyperbolic cross and sparse grids (cf. \cite{shen2014approximations}) when the dimension is high.

The proposed method can be extended to invert the operator $\mathbb D^s:=(-\Delta +\gamma\mathbb  I)^{s}$ with $s\in (0,1)$ and $\gamma>0.$ In fact, one can verify readily that the Dunford-Taylor formulation in Lemma \ref{DTthm} takes the form
\begin{equation}\label{DTfor--}
\left(\mathbb D^{\frac{s}2} u,\mathbb D^{\frac{s}2} v\right)_{L^2(\mathbb{R}^d)}=C_s \int_0^\infty t^{1-2s} \int_{\mathbb R^d}
 \big((-\Delta+\gamma\mathbb I)\big(\mathbb I+t^2 (-\Delta+\gamma\mathbb I)\big)^{-1}  u \big)(x)\,  v(x)\, {\rm d}x\,{\rm d}t.
\end{equation}
Then the fast algorithm in Theorem \ref{newtheorem} is extendable to this case straightforwardly.


\begin{appendix}

\renewcommand{\theequation}{A.\arabic{equation}}
\section{Proof of Proposition \ref{case1:expo}} \label{AppendixA}

 The results with $d=1$ were derived in \cite{tang2018hermite}, so it suffices to prove them for integer $d\geq 2$. Note that
\begin{equation*}
\begin{split}
\mathscr{F}\big\{e^{-|x|^2}\big\}(\xi) & =\frac{1}{(2\pi)^{d/2}}\int_{\mathbb{R}^d}e^{-|x|^2}e^{-\textmd{i}x\cdot\xi}{\rm d}x\\
&=\frac{1}{(2\pi)^{d/2}}\int_{\mathbb{R}}e^{-x_{1}^2}e^{-\textmd{i}x_{1}\xi_{1}}{\rm d}x_{1}\cdots\int_{\mathbb{R}}e^{-x_{d}^2}e^{-\textmd{i}x_{d}\xi_{d}}{\rm d}x_{d}=\frac{1}{2^{d/2}}e^{-\frac{|\xi|^2}{4}},
\end{split}
\end{equation*}
where we used the identity (cf.  \cite[P. 339]{Gradshteyn2015Book}):
\begin{equation*}
\int_{\mathbb{R}}e^{-x^2}e^{-\textmd{i}x\xi}{\rm d}x=\sqrt{\pi} e^{-\frac{\xi^2}{4}}.
\end{equation*}
Thus from the definition \eqref{Ftransform}, we obtain
\begin{equation}\label{comeIr}
\begin{split}
(-\Delta)^{s}\big\{e^{-|x|^2}\big\}(x)&=\mathscr{F}^{-1}\Big\{|\xi|^{2s}\mathscr{F}\big\{e^{-|x|^2}\big\}(\xi)\Big\}=\frac{1}{2^{d/2}(2\pi)^{d/2}}\!\int_{\mathbb{R}^d}|\xi|^{2s}
e^{-\frac{|\xi|^2}{4}}e^{\textmd{i}x\cdot\xi}{\rm d}\xi\\
&=\frac{2^d}{2^{d/2}(2\pi)^{d/2}}\int_{\mathbb{R}_{+}^d}|\xi|^{2s}
e^{-\frac{|\xi|^2}{4}}\cos(x_{1}\xi_{1})\cos(x_{2}\xi_{2})\cdots\cos(x_{d}\xi_{d}){\rm d}\xi.
\end{split}
\end{equation}
We proceed with the calculation by  using the $d$-dimensional spherical coordinates:
\begin{equation}\label{d_sphere}
\begin{split}
&\xi_{1}=r\cos\theta_{1};\; \xi_{2}=r\sin\theta_{1}\cos\theta_{2};\;\cdots\cdots; \;\xi_{d-1}=r\sin\theta_{1}\cdots\sin\theta_{d-2}\cos\theta_{d-1}; \\ &\xi_{d}=r\sin\theta_{1}\cdots\sin\theta_{d-2}\sin\theta_{d-1},\quad r=|\xi|,
\end{split}
\end{equation}
so we can write
\begin{equation}\label{comeIr00}
\begin{split}
(-\Delta)^{s}\big\{e^{-|x|^2}\big\}(x)=\frac{1}{\pi^{d/2}}\int_{0}^{\infty}{r^{2s+d-1}}e^{-\frac{r^2}{4}}\,\mathcal{I}(r; x){\rm d}r,
\end{split}
\end{equation}
where
\begin{equation*}
\begin{split}
&\mathcal{I}(r; x)=\!\!\int_{[0, \frac \pi 2]^{d-1}}
\!\cos\!\big(r x_{1}\!\cos{\theta_{1}}\big)\cos\!\big(r x_{2}\!\sin\theta_{1}\!\cos\theta_{2}\big)\!\cdots\cos\!\big(r x_{d-1}\!\sin\theta_{1}\cdots\sin\theta_{d-2}\!\cos\theta_{d-1}\big)\\
&\quad\quad \cos\!\big(r x_{d}\sin\theta_{1}\cdots\sin\theta_{d-2}\!\sin\theta_{d-1}\big)(\sin{\theta_{1}})^{d-2}(\sin{\theta_{2}})^{d-3}\cdots (\sin{\theta_{d-2}})\,{\rm d}\theta_{1}{\rm d}\theta_{2}\cdots{\rm d}\theta_{d-1}.
\end{split}
\end{equation*}
We first integrate $\mathcal{I}(r; x)$ with respect to $\theta_{d-1}$. To do this, we recall the  integral formula involving the Bessel functions (cf.  \cite[P. 732]{Gradshteyn2015Book}): for real $\mu,\,\nu>-1$ and $a,b>0$, 
 \begin{align}\label{bessel_cossin}
\int_{0}^{\frac{\pi}{2}} J_{\nu}(a \sin \theta) J_{\mu}(b \cos \theta) \sin ^{\nu+1}\theta \cos ^{\mu+1}\theta\, {\rm d} \theta=\frac{a^{\nu} b^{\mu} J_{\nu+\mu+1}\left(\sqrt{a^{2}+b^{2}}\right)}{(a^{2}+b^{2})^{(\nu+\mu+1)/2}},
\end{align}
Then using the identity $\cos z=\sqrt{\pi z/2}J_{-1/2}(z)$ and \eqref{bessel_cossin} (with $a=r x_{d-1}\!\sin\theta_{1}\cdots\sin\theta_{d-2}$, $b=r x_{d}\sin\theta_{1}\cdots\sin\theta_{d-2}$ and $\mu=\nu=-1/2$), we derive
\begin{equation*}
\begin{split}
&\int_{0}^{\frac{\pi}{2}}\cos\!\big(r x_{d-1}\!\sin\theta_{1}\cdots\sin\theta_{d-2}\!\cos\theta_{d-1}\big)\cos\!\big(r x_{d}\sin\theta_{1}\cdots\sin\theta_{d-2}\!\sin\theta_{d-1}\big)\, {\rm d}\theta_{d-1}\\
&\qquad\qquad=\frac{\pi}{2}J_{0}\big(r \sin\theta_{1}\cdots\sin\theta_{d-2}\sqrt{x_{d-1}^2+x_{d}^2}\,\big).
\end{split}
\end{equation*}
Substituting the above into $\mathcal I(r,x)$, and applying  the same argument   to
$\theta_{d-2}, \theta_{d-3},\cdots, \theta_{1}$ iteratively $d-2$ times,  we  obtain
\begin{equation}\label{Ir}
\mathcal{I}(r; x)=\Big(\frac{\pi}{2}\Big)^{\frac{d}{2}}(r|x|)^{1-\frac{d}{2}}J_{\frac{d}{2}-1}(r|x|).
\end{equation}

We proceed  with  the  integral identity  (cf.  \cite[P. 713]{Gradshteyn2015Book}): for real $\mu+\nu>-1$ and $p>0$,
\begin{equation}\label{bessel_exp}
\int_{\mathbb{R}^+}J_\mu(bt)e^{-p^2 t^2} t^{\nu-1}{\rm d}t = \frac{ b^{\mu} \Gamma((\mu + \nu)/2)}{2^{\mu+1}p^{\nu+\mu} \Gamma(\mu+1)}  {}_1F_1\Big(\frac{\mu+\nu}{2}; \mu+1; -\frac{b^2}{4p^2}\Big).
\end{equation}
Then, substituting  \eqref{Ir} into \eqref{comeIr00} and using \eqref{bessel_exp} (with $\mu=d/2-1$ and $\nu=2s+d/2+1$), we derive
\begin{align*}
(-\Delta)^{s}\big\{e^{-|x|^2}\big\}=\frac{|x|^{1-\frac{d}{2}}}{2^{d/2}}\int_{0}^{\infty}\!\!r^{2s+\frac{d}{2}}e^{-\frac{r^2}{4}}J_{\frac{d}{2}-1}(r|x|){\rm d}r=\frac{2^{2s}\Gamma(s+d/2)}{\Gamma(d/2)}{}_{1}F_{1}\Big(s+\frac{d}{2};\frac{d}{2};-|x|^2\Big).
\end{align*}
This yields \eqref{2Dcase}.  The asymptotic behaviour \eqref{asymexpo} follows from the property  (cf. \cite[P. 278]{bateman1953higher}):
\begin{equation}\label{asyex1}
{}_{1}F_{1}(a;b;z)=\frac{\Gamma(b)}{\Gamma(b-a)}(-z)^{-a}\big\{1+O(|z|^{-1})\big\}.
\end{equation}
Then \eqref{asymexpo} follows.  This completes the proof.

\renewcommand{\theequation}{B.\arabic{equation}}
\section{Proof of Proposition \ref{case1:alg}}\label{AppendixB}
The identity with $d=1$ can be found in \cite{tang2019rational}, so we assume that $d\geq 2$. Using the $d$-spherical  coordinate system in \eqref{d_sphere}, we obtain from \eqref{Ir} that
\begin{align*}
&{\mathscr F} \Big\{\frac{1}{(1+|x|^2)^{\gamma}}\Big\}(\xi)\!=\!\frac{1}{(2\pi)^{d/2}}\!\!\int_{\mathbb{R}^{d}}\!\frac{e^{-\textmd{i}x\cdot\xi}}{\left(1+|x|^2\right)^{\gamma}}{\rm d}x\!
=\!\!\frac{2^d}{(2\pi)^{d/2}}\!\!\int_{\mathbb{R}_{+}^{d}}\!\!\frac{\cos(x_{1}\xi_{1})\cos(x_{2}\xi_{2})\cdots\cos(x_{d}\xi_{d})}{\left(1+|x|^2\right)^{\gamma}}\,{\rm d}x\\
&\qquad=\Big(\frac{2}{\pi}\Big)^{\frac{d}{2}}\!\int_{0}^{\infty}\!\frac{r^{d-1}}{\left(1+r^2\right)^{\gamma}}\mathcal{I}(r; \xi)\,{\rm d}r=|\xi|^{1-\frac{d}{2}}\int_{0}^{\infty}\frac{r^{\frac{d}{2}}}{\left(1+r^2\right)^{\gamma}}J_{\frac{d}{2}-1}(r|\xi|){\rm d}r.
\end{align*}
 Recall the integral formula (cf. \cite[P. 686]{Gradshteyn2015Book}): for $-1<\nu<2\mu+\frac{3}{2}$ and $a,b>0$,
\begin{equation}\label{JtoK}
\int_{0}^{\infty}\frac{x^{\nu+1}}{(x^2+a^2)^{\mu+1}}J_{\nu}(bx){\rm d}x=\frac{a^{\nu-\mu}b^{\mu}}{2^{\mu}\Gamma(\mu+1)}K_{\nu-\mu}(ab),
\end{equation}
where $K_{\nu}(x)$ is the modified Bessel functions of the second kind. Note that $K_{-\nu}(x)=K_{\nu}(x)$. Then letting $\mu=\gamma-1$ and $\nu=d/2-1$ in \eqref{JtoK}, we obtain
\begin{align*}
&\quad{\mathscr F} \Big\{\frac{1}{(1+|x|^2)^{\gamma}}\Big\}(\xi)=\frac{|\xi|^{\gamma-\frac{d}{2}}}{2^{\gamma-1}\Gamma(\gamma)}K_{\gamma-\frac{d}{2}}(|\xi|).
\end{align*}
We also use  the integral formula (cf. \cite[P. 692]{Gradshteyn2015Book}):  for real $a>0$, real $b$, and $\nu-\lambda+1>|\mu|$,
\begin{equation}\label{JK}
\begin{split}
\int_{0}^{\infty}x^{-\lambda}K_{\mu}(ax)J_{\nu}(bx){\rm d}x&=\dfrac{b^{\nu}\Gamma\big((\nu-\lambda+\mu+1)/2\big)\Gamma\big((\nu-\lambda-\mu+1)/2)}{2^{\lambda+1}a^{\nu-\lambda+1}
\Gamma(\nu+1)}\times \\
&\qquad{}_{2}F_{1}\Big(\frac{\nu-\lambda+\mu+1}{2},\frac{\nu-\lambda-\mu+1}{2};\nu+1;-\frac{b^2}{a^2}\Big).
\end{split}
\end{equation}
Once again, using  the $d$-spherical  coordinate system  \eqref{d_sphere}, \eqref{Ir} and \eqref{JK} (with $\lambda=-2s-\gamma$, $\mu=\gamma-d/2$ and $\nu=d/2-1$), we have
\begin{align*}
&(-\Delta)^{s}\Big\{\frac{1}{(1+|x|^2)^{\gamma}}\Big\}=\frac{1}{(2\pi)^{\frac{d}{2}}2^{\gamma-1}\Gamma(\gamma)}\int_{\mathbb{R}^d}e^{\textmd{i}x\cdot\xi}|\xi|^{2s+\gamma-\frac{d}{2}}K_{\gamma-\frac{d}{2}}(|\xi|)\,{\rm d}\xi \\
=&\frac{2^d}{(2\pi)^{\frac{d}{2}}2^{\gamma-1}\Gamma(\gamma)}\int_{\mathbb{R}_{+}^d}\cos(x_{1}\xi_{1})\cos(x_{2}\xi_{2})\cdots\cos(x_{d}\xi_{d})|\xi|^{2s+\gamma-\frac{d}{2}}K_{\gamma-\frac{d}{2}}(|\xi|)\,{\rm d}\xi\\
=&\frac{2^{\frac{d}{2}-\gamma+1}}{\pi^{\frac{d}{2}}\Gamma(\gamma)}\int_{0}^{\infty}r^{2s+\gamma+\frac{d}{2}-1}K_{\gamma-\frac{d}{2}}(r)\mathcal{I}(r,x)\,{\rm d}r
=\frac{2^{-\gamma+1}}{\Gamma(\gamma)}|x|^{1-\frac{d}{2}}\int_{0}^{\infty}r^{2s+\gamma}K_{\gamma-\frac{d}{2}}(r)J_{\frac{d}{2}-1}(r|x|){\rm d}r\\
=&\frac{2^{2s}\Gamma(s+\gamma)\Gamma(s+\frac{d}{2})}{\Gamma(\gamma)\Gamma(\frac{d}{2})}{}_{2}F_{1}\Big(s+\gamma,s+\frac{d}{2};\frac{d}{2};-|x|^2\Big).
\end{align*}
This completes the derivation of \eqref{FLalge}.

 According to  \cite[P. 76]{bateman1953higher},  
the asymptotic behaviour of the hypergeometric function for large $|x|$  (unless $a-b$ is an integer) is
\begin{equation}
_{2} F_{1}(a, b;c;x) = \lambda_{1}|x|^{-a}+\lambda_{2}|x|^{-b}+O(|x|^{-a-1})+O(|x|^{-b-1}).
\end{equation}
where $\lambda_{1}$ and $\lambda_2$ are constants;   if $a-b$ is an integer, $z^{-a}$ or $z^{-b}$ has to be multiplied by a factor $\ln(x)$. Then  we have the asymptotic behaviour of $(-\Delta)^{s}\big\{\frac{1}{(1+|x|^2)^{\gamma}}\big\}$ as $|x| \rightarrow \infty$ in \eqref{gmaao}-\eqref{neweqnAln}.

\end{appendix}

\bibliographystyle{siam}

\bibliography{ref,rational}

\begin{thebibliography}{10}

\bibitem{acosta2017short}
{\sc G.~Acosta, F.~M. Bersetche, and J.~P. Borthagaray}, {\em {A short FE
  implementation for a 2d homogeneous Dirichlet problem of a fractional
  Laplacian}}, Comput. Math. Appl., 74 (2017), pp.~784--816.

\bibitem{acosta2017fractional}
{\sc G.~Acosta and J.~P. Borthagaray}, {\em {A fractional Laplace equation:
  regularity of solutions and finite element approximations}}, SIAM J. Numer.
  Anal., 55 (2017), pp.~472--495.

\bibitem{Agranovich2015Book}
{\sc M.~Agranovich}, {\em Sobolev spaces, their generalizations and elliptic
  problems in smooth and Lipschitz domains}, Springer, 2015.

\bibitem{ainsworth2017aspects}
{\sc M.~Ainsworth and C.~Glusa}, {\em Aspects of an adaptive finite element
  method for the fractional laplacian: a priori and a posteriori error
  estimates, efficient implementation and multigrid solver}, Comput. Methods
  Appl. Mech. Engrg., 327 (2017), pp.~4--35.

\bibitem{ainsworth2018hybrid}
\leavevmode\vrule height 2pt depth -1.6pt width 23pt, {\em {Hybrid finite
  element--spectral method for the fractional Laplacian: approximation theory
  and efficient solver}}, SIAM J. Sci. Comput., 40 (2018), pp.~A2383--A2405.

\bibitem{babuska1972survey}
{\sc I.~Babuska}, {\em Survey lectures on the mathematical foundations of the
  finite element method}, The Mathematical Foundations of the Finite Element
  Method with Applicaions to Partial Differential Equations,  (1972),
  pp.~3--359.

\bibitem{bao2005fourth}
{\sc W.~Bao and J.~Shen}, {\em {A fourth-order time-splitting Laguerre--Hermite
  pseudospectral method for Bose--Einstein condensates}}, SIAM J. Sci. Comput.,
  26 (2005), pp.~2010--2028.

\bibitem{bateman1953higher}
{\sc H.~Bateman}, {\em Higher transcendental functions [volumes i-iii]}, 1953.

\bibitem{benson2000application}
{\sc D.~A. Benson, S.~W. Wheatcraft, and M.~M. Meerschaert}, {\em {Application
  of a fractional advection-dispersion equation}}, Water Resour. Res., 36
  (2000), pp.~1403--1412.

\bibitem{Bonito2018Numer}
{\sc A.~Bonito, J.~P. Borthagaray, R.~H. Nochetto, E.~Ot\'{a}rola, and A.~J.
  Salgado}, {\em Numerical methods for fractional diffusion}, Comput. Vis.
  Sci., 19 (2018), pp.~19--46.

\bibitem{bonito2019numerical}
{\sc A.~Bonito, W.~Lei, and J.~E. Pasciak}, {\em {Numerical approximation of
  the integral fractional Laplacian}}, Numer. Math., 142 (2019), pp.~235--278.

\bibitem{bonito2019sinc}
\leavevmode\vrule height 2pt depth -1.6pt width 23pt, {\em On sinc quadrature
  approximations of fractional powers of regularly accretive operators}, J.
  Numer. Math., 27 (2019), pp.~57--68.

\bibitem{brockmann2006scaling}
{\sc D.~Brockmann, L.~Hufnagel, and T.~Geisel}, {\em The scaling laws of human
  travel}, Nature, 439 (2006), p.~462.

\bibitem{caffarelli2007extension}
{\sc L.~Caffarelli and L.~Silvestre}, {\em {An extension problem related to the
  fractional Laplacian}}, Comm. Partial Differential Equations, 32 (2007),
  pp.~1245--1260.

\bibitem{carmichael2015fractional}
{\sc B.~Carmichael, H.~Babahosseini, S.~Mahmoodi, and M.~Agah}, {\em The
  fractional viscoelastic response of human breast tissue cells}, Phys. Biol.,
  12 (2015), p.~046001.

\bibitem{chen2018jacobi}
{\sc L.~Chen, Z.~Mao, and H.~Li}, {\em {Jacobi-Galerkin spectral method for
  eigenvalue problems of Riesz fractional differential equations}}, arXiv
  preprint arXiv:1803.03556,  (2018).

\bibitem{chen2018laguerre}
{\sc S.~Chen, J.~Shen, and L.-L. Wang}, {\em Laguerre functions and their
  applications to tempered fractional differential equations on infinite
  intervals}, J. Sci. Comput., 74 (2018), pp.~1286--1313.

\bibitem{cushman1993nonlocal}
{\sc J.~H. Cushman and T.~Ginn}, {\em Nonlocal dispersion in media with
  continuously evolving scales of heterogeneity}, Transp. Porous Media, 13
  (1993), pp.~123--138.

\bibitem{deng2008finite}
{\sc W.~Deng}, {\em {Finite element method for the space and time fractional
  Fokker--Planck equation}}, SIAM J. Numer. Anal., 47 (2008), pp.~204--226.

\bibitem{deng2018time}
{\sc W.~Deng, B.~Li, Z.~Qian, and H.~Wang}, {\em {Time discretization of a
  tempered fractional Feynman--Kac equation with measure data}}, SIAM J. Numer.
  Anal., 56 (2018), pp.~3249--3275.

\bibitem{dubook}
{\sc Q.~Du}, {\em {Nonlocal modeling, analysis, and computation}}, vol.~94,
  CBMS-NSF Regional Conference Series in Applied Mathematics, SIAM, 2019.

\bibitem{duo2015computing}
{\sc S.~Duo and Y.~Zhang}, {\em {Computing the ground and first excited states
  of the fractional Schr{\"o}dinger equation in an infinite potential well}},
  Commun. Comput. Phys., 18 (2015), pp.~321--350.

\bibitem{duo2018finite}
\leavevmode\vrule height 2pt depth -1.6pt width 23pt, {\em {Finite difference
  methods for two and three dimensional fractional Laplacian with applications
  to solve the fractional reaction-diffusion equations}}, arXiv preprint
  arXiv:1804.02718,  (2018).

\bibitem{Gradshteyn2015Book}
{\sc I.~S. Gradshteyn and I.~M. Ryzhik}, {\em Table of Integrals, Series, and
  Products}, Elsevier/Academic Press, Amsterdam, eighth~ed., 2015.
\newblock Translated from the Russian, Translation edited and with a preface by
  Daniel Zwillinger and Victor Moll, Revised from the seventh edition
  [MR2360010].

\bibitem{ben2002modified}
{\sc B.~Guo and Z.~Wang}, {\em {Modified Chebyshev rational spectral method for
  the whole line}}, in Proceedings of the fourth international conference on
  dynamical systems and differential equations, 2002, pp.~365--374.

\bibitem{guo2018high}
{\sc X.~Guo, Y.~Li, and H.~Wang}, {\em {A high order finite difference method
  for tempered fractional diffusion equations with applications to the CGMY
  model}}, SIAM J. Sci. Comput., 40 (2018), pp.~A3322--A3343.

\bibitem{hatano1998dispersive}
{\sc Y.~Hatano and N.~Hatano}, {\em {Dispersive transport of ions in column
  experiments: An explanation of long-tailed profiles}}, Water Resour. Res., 34
  (1998), pp.~1027--1033.

\bibitem{hou2017fractional}
{\sc D.~Hou and C.~Xu}, {\em A fractional spectral method with applications to
  some singular problems}, Adv. Comput. Math., 43 (2017), pp.~911--944.

\bibitem{huang2014numerical}
{\sc Y.~Huang and A.~Oberman}, {\em {Numerical methods for the fractional
  Laplacian: a finite difference-quadrature approach}}, SIAM J. Numer. Anal.,
  52 (2014), pp.~3056--3084.

\bibitem{jin2013error}
{\sc B.~Jin, R.~Lazarov, and Z.~Zhou}, {\em Error estimates for a semidiscrete
  finite element method for fractional order parabolic equations}, SIAM J.
  Numer. Anal., 51 (2013), pp.~445--466.

\bibitem{jin2018numerical}
{\sc B.~Jin, B.~Li, and Z.~Zhou}, {\em Numerical analysis of nonlinear
  subdiffusion equations}, SIAM J. Numer. Anal., 56 (2018), pp.~1--23.

\bibitem{klein2014numerical}
{\sc C.~Klein, C.~Sparber, and P.~Markowich}, {\em {Numerical study of
  fractional nonlinear Schr{\"o}dinger equations}}, Proc. Ser. A Math. Phys.
  Eng. Sci., 470 (2014), p.~20140364.

\bibitem{lischke2018fractional}
{\sc A.~Lischke, G.~Pang, M.~Gulian, F.~Song, C.~Glusa, X.~Zheng, Z.~Mao,
  W.~Cai, M.~M. Meerschaert, M.~Ainsworth, et~al.}, {\em {What is the
  fractional Laplacian?}}, arXiv preprint arXiv:1801.09767,  (2018).

\bibitem{mao2016efficient}
{\sc Z.~Mao, S.~Chen, and J.~Shen}, {\em {Efficient and accurate spectral
  method using generalized Jacobi functions for solving Riesz fractional
  differential equations}}, Appl. Numer. Math., 106 (2016), pp.~165--181.

\bibitem{mao2017hermite}
{\sc Z.~Mao and J.~Shen}, {\em {Hermite spectral methods for fractional PDEs in
  unbounded domains}}, SIAM J. Sci. Comput., 39 (2017), pp.~A1928--A1950.

\bibitem{mccay1981theory}
{\sc B.~McCay and M.~N.~L. Narasimhan}, {\em Theory of nonlocal electromagnetic
  fluids}, Arch. Mech., 33 (1981), pp.~365--384.

\bibitem{metzler2000random}
{\sc R.~Metzler and J.~Klafter}, {\em The random walk's guide to anomalous
  diffusion: a fractional dynamics approach}, Phys. Rep., 339 (2000),
  pp.~1--77.

\bibitem{metzler2004restaurant}
\leavevmode\vrule height 2pt depth -1.6pt width 23pt, {\em The restaurant at
  the end of the random walk: recent developments in the description of
  anomalous transport by fractional dynamics}, J. Phys. A, 37 (2004), p.~R161.

\bibitem{montroll1965random}
{\sc E.~W. Montroll and G.~H. Weiss}, {\em Random walks on lattices. {II}}, J.
  Math. Phys., 6 (1965), pp.~167--181.

\bibitem{Nezza2012BSM}
{\sc E.~D. Nezza, G.~Palatucci, and E.~Valdinoci}, {\em {Hitchhiker's guide to
  the fractional Sobolev spaces}}, Bull. Sci. Math., 136 (2012), pp.~521--573.

\bibitem{nochetto2014pde}
{\sc R.~H. Nochetto, E.~Ot{\'a}rola, and A.~J. Salgado}, {\em {A PDE approach
  to fractional diffusion in general domains: a priori error analysis}}, Found.
  Comput. Math., 15 (2014), pp.~733--791.

\bibitem{nochetto2016pde}
{\sc R.~H. Nochetto, E.~Otarola, and A.~J. Salgado}, {\em {A PDE approach to
  space-time fractional parabolic problems}}, SIAM J. Numer. Anal., 54 (2016),
  pp.~848--873.

\bibitem{shen2011spectral}
{\sc J.~Shen, T.~Tang, and L.-L. Wang}, {\em Spectral methods: algorithms,
  analysis and applications}, vol.~41, Springer Science \& Business Media,
  2011.

\bibitem{She.W09}
{\sc J.~Shen and L.~Wang}, {\em Some recent advances on spectral methods for
  unbounded domains}, Commun. Comput. Phys., 5 (2009), pp.~195--241.

\bibitem{shen2014approximations}
{\sc J.~Shen, L.-L. Wang, and H.~Yu}, {\em {Approximations by orthonormal
  mapped Chebyshev functions for higher-dimensional problems in unbounded
  domains}}, J. Comput. Appl. Math., 265 (2014), pp.~264--275.

\bibitem{shlesinger1987levy}
{\sc M.~Shlesinger, B.~West, and J.~Klafter}, {\em L{\'e}vy dynamics of
  enhanced diffusion: Application to turbulence}, Phys. Rev. Lett., 58 (1987),
  p.~1100.

\bibitem{silling2000reformulation}
{\sc S.~A. Silling}, {\em Reformulation of elasticity theory for
  discontinuities and long-range forces}, J. Mech. Phys. Solids, 48 (2000),
  pp.~175--209.

\bibitem{sims2008scaling}
{\sc D.~W. Sims, E.~J. Southall, N.~E. Humphries, G.~C. Hays, C.~J. Bradshaw,
  J.~W. Pitchford, A.~James, M.~Z. Ahmed, A.~S. Brierley, M.~A. Hindell,
  et~al.}, {\em Scaling laws of marine predator search behaviour}, Nature, 451
  (2008), p.~1098.

\bibitem{szeg1939orthogonal}
{\sc G.~Szeg{\"o}}, {\em Orthogonal polynomials}, vol.~23, American
  Mathematical Soc., 1939.

\bibitem{tang2019rational}
{\sc T.~Tang, L.-L. Wang, H.~Yuan, and T.~Zhou}, {\em {Rational spectral
  methods for PDEs involving fractional Laplacian in unbounded domains}}, arXiv
  preprint arXiv:1905.02476,  (2019).

\bibitem{tang2018hermite}
{\sc T.~Tang, H.~Yuan, and T.~Zhou}, {\em {Hermite spectral collocation methods
  for fractional PDEs in unbounded domains}}, Commun. Comput. Phys., 24 (2018),
  pp.~1143--1168.

\bibitem{zhang2018riesz}
{\sc Z.~Zhang, W.~Deng, and G.~E. Karniadakis}, {\em {A Riesz basis Galerkin
  method for the tempered fractional Laplacian}}, SIAM J. Numer. Anal., 56
  (2018), pp.~3010--3039.

\end{thebibliography}
\end{document}